\documentclass[12pt, a4paper]{amsart}

\usepackage{amscd,verbatim}
\usepackage{amssymb}

\usepackage{bm}

\usepackage{amssymb, amsmath}
\usepackage[all]{xy}
\usepackage[inline]{enumitem}
\usepackage{hyphenat}
\usepackage[colorlinks,linkcolor=blue,citecolor=blue,urlcolor=red]{hyperref}
\usepackage{xcolor}
\usepackage{mathtools}
\usepackage{tikz-cd}
\usepackage[labelformat=empty]{caption}
 \usepackage[ left=3cm, right=3cm, top = 2.8cm, bottom = 2.8cm ]{geometry}

\makeatletter
\def\@tocline#1#2#3#4#5#6#7{\relax
  \ifnum #1>\c@tocdepth % then omit
  \else
    \par \addpenalty\@secpenalty\addvspace{#2}%
    \begingroup \hyphenpenalty\@M
    \@ifempty{#4}{%
      \@tempdima\csname r@tocindent\number#1\endcsname\relax
    }{%
      \@tempdima#4\relax
    }%
    \parindent\z@ \leftskip#3\relax \advance\leftskip\@tempdima\relax
    \rightskip\@pnumwidth plus4em \parfillskip-\@pnumwidth
    #5\leavevmode\hskip-\@tempdima
      \ifcase #1
      \or\or \hskip 2em \or \hskip 2em \else \hskip 3em \fi%
      #6\nobreak\relax
    \dotfill\hbox to\@pnumwidth{\@tocpagenum{#7}}\par
    \nobreak
    \endgroup
  \fi}
\makeatother

\newcommand{\A}{\mathbf{A}}

\newcommand{\G}{\mathbf{G}}

\renewcommand{\P}{\mathbf{P}}

\newcommand{\Z}{\mathbb{Z}}

\newcommand{\sC}{\mathcal{C}}

\newcommand{\sX}{\mathcal{X}}
\newcommand{\sY}{\mathcal{Y}}

\newcommand{\Cor}{\operatorname{\mathbf{Cor}}}

\newcommand{\HI}{\operatorname{\mathbf{HI}}}

\newcommand{\ul}[1]{{\underline{#1}}}

\newcommand{\Cpx}{{\operatorname{\mathbf{Cpx}}}}
\newcommand{\D}{{\operatorname{\mathbf{D}}}}
\newcommand{\PST}{{\operatorname{\mathbf{PST}}}}

\newcommand{\NST}{\operatorname{\mathbf{NST}}}

\newcommand{\DM}{\operatorname{\mathbf{DM}}}

\newcommand{\Hom}{\operatorname{Hom}}
\newcommand{\uHom}{\operatorname{\underline{Hom}}}
\newcommand{\uExt}{\operatorname{\underline{Ext}}}

\newcommand{\Coker}{\operatorname{Coker}}

\newcommand{\Spec}{\operatorname{Spec}}

\newcommand{\Comp}{\operatorname{Comp}}
\newcommand{\Sm}{\operatorname{\mathbf{Sm}}}

\newcommand{\Shv}{\operatorname{\mathbf{Shv}}}

\newcommand{\pro}[1]{\text{\rm pro}_{#1}\text{\rm--}}
\newcommand{\tr}{{\operatorname{tr}}}

\newcommand{\fin}{{\operatorname{fin}}}

\newcommand{\red}{{\operatorname{red}}}

\newcommand{\Zar}{{\operatorname{Zar}}}
\newcommand{\Nis}{{\operatorname{Nis}}}

\newcommand{\id}{{\operatorname{Id}}}

\newcommand{\codim}{{\operatorname{codim}}}
\newcommand{\ch}{{\operatorname{ch}}}

\renewcommand{\lim}{\operatornamewithlimits{\varprojlim}}
\newcommand{\colim}{\operatornamewithlimits{\varinjlim}}

\newcommand{\ol}{\overline}

\renewcommand{\phi}{\varphi}
\renewcommand{\epsilon}{\varepsilon}
\renewcommand{\div}{\operatorname{div}}

\newcommand{\MNS}{\operatorname{\mathbf{MNS}}}
\newcommand{\MNST}{\operatorname{\mathbf{MNST}}}

\newcommand{\MCor}{\operatorname{\mathbf{MCor}}}

\newcommand{\MSm}{\operatorname{\mathbf{MSm}}}

\newcommand{\MPST}{\operatorname{\mathbf{MPST}}}
\newcommand{\CI}{\operatorname{\mathbf{CI}}}

\newcommand{\Sq}{{\operatorname{\mathbf{Sq}}}}
\newcommand{\MSmsq}{{(\MSm)^{\Sq}}}

\newcommand{\bcube}{{\ol{\square}}}

\newcommand{\M}{\mathbf{M}}

\newcommand{\ulMSm}{\operatorname{\mathbf{\underline{M}Sm}}}
\newcommand{\ulMNS}{\operatorname{\mathbf{\underline{M}NS}}}
\newcommand{\ulMPS}{\operatorname{\mathbf{\underline{M}PS}}}
\newcommand{\ulMPST}{\operatorname{\mathbf{\underline{M}PST}}}
\newcommand{\ulMNST}{\operatorname{\mathbf{\underline{M}NST}}}
\newcommand{\ulMCor}{\operatorname{\mathbf{\underline{M}Cor}}}
\newcommand{\ulomega}{\underline{\omega}}

\newcommand{\uPic}{\underline{\Pic}}

\newcounter{spec}
{\end{list}}%

\newtheorem{lemma}{Lemma}[section]
\newtheorem{thm}[lemma]{Theorem}
\newtheorem{prethm}[lemma]{Pretheorem}

\newtheorem{prop}[lemma]{Proposition}

\newtheorem{cor}[lemma]{Corollary}
\newtheorem{corollary}[lemma]{Corollary}
\newtheorem{conj}[lemma]{Conjecture}
\theoremstyle{definition}
\newtheorem{defn}[lemma]{Definition}
\newtheorem{constr}[lemma]{Construction}
\newtheorem{definition}[lemma]{Definition}

\newtheorem{para}[lemma]{}
\theoremstyle{remark}

\newtheorem{remark}[lemma]{Remark}

\newtheorem{example}[lemma]{Example}
\newtheorem{claim}[lemma]{Claim}

\numberwithin{equation}{section}

\numberwithin{equation}{lemma}

\def\lSm{\mathbf{lSm}}
\def\SmlSm{\mathbf{SmlSm}}
\def\Sm{\mathbf{Sm}}

\newcounter{elno}

\begin{document}

\def\aNis{a_{\Nis}}
\def\ulaNis{\underline{a}_{\Nis}}
\def\ulasNis{\underline{a}_{s,\Nis}}
\def\ulaNisfin{\underline{a}^{\fin}_{\Nis}}
\def\ulasNisfin{\underline{a}^{\fin}_{s,\Nis}}
\def\asNis{a_{s,\Nis}}
\def\ulasNis{\underline{a}_{s,\Nis}}
\def\qaq{\quad\text{ and }\quad}
\def\limcat#1{``\underset{#1}{\lim}"}
\def\Comp{\Comp^{\fin}}
\def\ulc{\ul{c}}
\def\ulb{\ul{b}}
\def\ulgam{\ul{\gamma}}
\def\MSm{\operatorname{\mathbf{MSm}}}
\def\MsigmaS{\operatorname{\mathbf{MsigmaS}}}
\def\ulMSm{\operatorname{\mathbf{\ul{M}Sm}}}
\def\ulMsigmaS{\operatorname{\mathbf{\ul{M}NS}}}

\def\ulMPS{\operatorname{\mathbf{\ul{M}PS}}}

\def\ulMsigmaS{\operatorname{\mathbf{\ul{M}PS}_\sigma}}
\def\ulMsigmaSTfin{\operatorname{\mathbf{\ul{M}PST}^{\fin}_\sigma}}
\def\ulMsigmaST{\operatorname{\mathbf{\ul{M}PST}_\sigma}}
\def\MsigmaS{\operatorname{\mathbf{MPS}_\sigma}}
\def\MsigmaST{\operatorname{\mathbf{MPST}_\sigma}}
\def\MsigmaSTfin{\operatorname{\mathbf{MPST}^{\fin}_\sigma}}

\def\ulMNS{\operatorname{\mathbf{\ul{M}NS}}}
\def\ulMNSTfin{\operatorname{\mathbf{\ul{M}NST}^{\fin}}}
\def\ulMNSfin{\operatorname{\mathbf{\ul{M}NS}^{\fin}}}
\def\ulMNST{\operatorname{\mathbf{\ul{M}NST}}}
\def\MNS{\operatorname{\mathbf{MNS}}}
\def\MNST{\operatorname{\mathbf{MNST}}}
\def\MNSTfin{\operatorname{\mathbf{MNST}^{\fin}}}
\def\RSC{\operatorname{\mathbf{RSC}}}

\def\LogRec{\operatorname{\mathbf{LogRec}}}

\def\MSmsq{\MSm^{\Sq}}
\def\Comp{\operatorname{\mathbf{Comp}}}
\def\uli{\ul{i}}
\def\ulis{\ul{i}_s}
\def\is{i_s}
\def\qfor{\text{ for }\;\;}
\def\CIlog{\operatorname{\mathbf{CI}}^{\mathrm{log}}}
\def\CIltr{\operatorname{\mathbf{CI}}^{\mathrm{ltr}}}
\def\CIt{\operatorname{\mathbf{CI}}^\tau}
\def\CItsp{\operatorname{\mathbf{CI}}^{\tau,sp}}

\def\otCIsp{\otimes_{\CI}^{sp}}
\def\otCINissp{\otimes_{\CI}^{\Nis,sp}}

\def\hM#1{h_0^{\bcube}(#1)}
\def\hMNis#1{h_0^{\bcube}(#1)_{\Nis}}
\def\hMM#1{h^0_{\bcube}(#1)}
\def\hMw#1{h_0(#1)}
\def\hMwNis#1{h_0(#1)_{\Nis}}

\def\ihF#1{F^{#1}}
\def\ihFA{\ihF {\A^1}}

\def\istm{\iota_{st,m}}
\def\im{\iota_m}
\def\est{\epsilon_{st}}
\def\tL{\tilde{L}}
\def\tX{\tilde{X}}
\def\tY{\tilde{Y}}
\def\omegaCI{\omega^{\CI}}
\def\qwith{\;\text{ with} }
\def\aVNis{a^V_\Nis}
\def\ulMCorls{\ulMCor_{ls}}

\def\Zinf{Z_\infty}
\def\Einf{E_\infty}
\def\Xinf{X_\infty}
\def\Yinf{Y_\infty}
\def\Pinf{P_\infty}

\def\Lot{{\cubegm\otimes\cubegm}}

\def\Ln#1{\Lambda_n^{#1}}
\def\tLn#1{\widetilde{\Lambda_n^{#1}}}
\def\tild#1{\widetilde{#1}}
\def\otuCINis{\otimes_{\underline{\CI}_\Nis}}
\def\otCI{\otimes_{\CI}}
\def\otCINis{\otimes_{\CI}^{\Nis}}
\def\tF{\widetilde{F}}
\def\tG{\widetilde{G}}
\def\bcubered{\bcube^{\textrm{red}}}
\def\cubegm{\bcube^{(1)}}
\def\cubegma{\bcube^{(a)}}
\def\cubegmb{\bcube^{(b)}}
\def\cubegmred{\bcube^{(1)}_{red}}
\def\cubegmreda{\bcube^{(a)}_{red}}
\def\cubegmredb{\bcube^{(b)}_{red}}

\def\LT{\bcube^{(1)}_{T}}
\def\LU{\bcube^{(1)}_{U}}
\def\LV{\bcube^{(1)}_{V}}
\def\LW{\bcube^{(1)}_{W}}
\def\LTred{\bcube^{(1)}_{T,red}}
\def\Lred{\bcube^{(1)}_{red}}
\def\LTred{\bcube^{(1)}_{T,red}}
\def\LUred{\bcube^{(1)}_{U,red}}
\def\LVred{\bcube^{(1)}_{V,red}}
\def\LWred{\bcube^{(1)}_{W,red}}
\def\PP{\P}
\def\AA{\A}

\def\LL{\bcube^{(2)}}
\def\LLred#1{\bcube^{(2)}_{#1,red}}
\def\LLredd{\bcube^{(2)}_{red}}
\def\Lredd#1{\bcube_{#1,red}}

\def\Lnredd#1{\bcube^{(#1)}_{red}}

\def\LLT{\bcube^{(2)}_T}
\def\LLTred{\bcube^{(2)}_{T,red}}

\def\LLU{\bcube^{(2)}_U}
\def\LLUred{\bcube^{(2)}_{U,red}}

\def\LLS{\bcube^{(2)}_S}
\def\LLSred{\bcube^{(2)}_{S,red}}
\def\tMCor{\Hom_{\MPST}}
\def\otHINis{\otimes_{\HI}^{\Nis}}

\def\Sh{\operatorname{\mathbf{Shv}}}
\def\Shv{\operatorname{\mathbf{Shv}}}
\def\PSh{\operatorname{\mathbf{PSh}}}
\def\Shltr{\operatorname{\mathbf{Shv}_{dNis}^{ltr}}}
\def\Shlog{\operatorname{\mathbf{Shv}_{dNis}^{log}}}
\def\Shvlog{\operatorname{\mathbf{Shv}^{log}}}
\def\SmlSm{\operatorname{\mathbf{SmlSm}}}
\def\lSm{\operatorname{\mathbf{lSm}}}
\def\lCor{\operatorname{\mathbf{lCor}}}
\def\SmlCor{\operatorname{\mathbf{SmlCor}}}
\def\PShltr{\operatorname{\mathbf{PSh}^{ltr}}}
\def\PShlog{\operatorname{\mathbf{PSh}^{log}}}
\def\lDM{\operatorname{\mathbf{logDM}^{eff}}}
\def\logDM{\operatorname{log\mathcal{DM}^{eff}}}
\def\DM{\operatorname{\mathbf{DM}^{eff}}}
\def\lDA{\operatorname{\mathbf{logDA}^{eff}}}
\def\DA{\operatorname{\mathbf{DA}^{eff}}}
\def\Log{\operatorname{\mathcal{L}\textit{og}}}
\def\Rsc{\operatorname{\mathcal{R}\textit{sc}}}
\def\pro{\textit{pro-}}
\def\plim{\textrm{``lim''}}

\def\hofib{\textrm{hofib}}
\def\triv{\textrm{triv}}
\def\ABl{\mathcal{A}\textit{Bl}}
\def\divsm#1{{#1_\textrm{div}^{\textrm{Sm}}}}

\newcommand{\dNis}{{\operatorname{dNis}}}
\newcommand{\ABNis}{{\operatorname{AB-Nis}}}
\newcommand{\sNis}{{\operatorname{sNis}}}
\newcommand{\sZar}{{\operatorname{sZar}}}
\renewcommand{\uPic}{\underline{\textrm{Pic}}}
\renewcommand{\M}{\mathcal{M}}
\newcommand{\cofib}{\textrm{Cofib}}

\newcommand{\Gmlog}{\G_m^{\log}}
\newcommand{\Gmlogred}{\overline{\G_m^{\log}}}

\newcommand{\varcolim}{\mathop{\mathrm{colim}}}
\newcommand{\varlim}{\mathop{\mathrm{lim}}}
\newcommand{\tensor}{\otimes}

\newcommand{\eq}[2]{\begin{equation}\label{#1}#2 \end{equation}}
\newcommand{\eqalign}[2]{\begin{equation}\label{#1}\begin{aligned}#2 \end{aligned}\end{equation}}

\def\varplim#1{\text{``}\varlim_{#1}\text{''}}

\title{Connectivity and purity for logarithmic motives}
\author{Federico Binda and Alberto Merici}

\address{Dipartimento di Matematica ``Federigo Enriques'',  Universit\`a degli Studi di Milano\\ Via Cesare Saldini 50, 20133 Milano, Italy}
\email{federico.binda@unimi.it}

\address{Institut f\"ur Mathematik, Universit\"at Zurich, Winterthurerstr. 190, CH-8057 Z\"urich, Switzerland}
\email{alberto.merici@math.uzh.ch}

\thanks{F.B.\ is supported by the PRIN “Geometric, Algebraic and Analytic Methods in Arithmetic”.}
\thanks{A.M.\ is supported by the Swiss National Science Foundation (SNSF), project 200020\_178729}

\begin{abstract}The goal of this paper is to extend the work of Voevodsky and Morel  on the homotopy $t$-structure on the category of motivic complexes to the context of motives for logarithmic schemes. To do so, we prove an analogue of Morel's connectivity theorem and show a purity statement for $(\P^1, \infty)$-local complexes of sheaves with log transfers. 
   The homotopy $t$-structure on $\lDM(k)$ is proved to be compatible with Voevodsky's $t$-structure i.e.\ we show that the comparison functor $R^{\bcube}\omega^*\colon \DM(k)\to \lDM(k)$ is $t$-exact. 
   The heart of the homotopy $t$-structure on $\lDM(k)$ is the Grothendieck abelian category of strictly cube-invariant sheaves with log transfers: we use it to build a new version of the category of reciprocity sheaves in the style of Kahn--Saito--Yamazaki and R\"ulling.
\end{abstract}
\maketitle
\tableofcontents
\section{Introduction}
Voevodsky's category of motivic complexes over a perfect field $k$ is based on a simple idea: most cohomology theories for smooth $k$-schemes are insensitive to the affine line, i.e.\ they satisfy $\A^1$-homotopy invariance. This observation led Voevodsky to introduce as a building block of his theory of motives the category of homotopy invariant sheaves with transfers $\HI_{\Nis}(k)$, that is, sheaves $F$ for the Nisnevich topology defined on the category of finite correspondences over $k$ such that $F(X\times \A^1)\xrightarrow{\simeq} F(X)$ for every smooth $k$-scheme $X$. These sheaves enjoy many nice properties: the category $\HI_{\Nis}(k)$ is a Grothendieck abelian subcategory of the category $\Shv_{\Nis}^{\tr}(k)$ of Nisnevich sheaves with transfers, closed under extensions and equipped with a (closed) symmetric monoidal structure $\tensor_{\HI}$. Moreover, a celebrated theorem of Voevodsky shows that the cohomology presheaves $H^n_{\Nis}(-, F)$ of a homotopy invariant sheaf with transfers $F$ are still $\A^1$-homotopy invariant. In fact, $\HI_{\Nis}(k)$ can be identified with the heart of a certain $t$-structure on the triangulated category $\DM(k)$, induced by the standard $t$-structure on the derived category $\D(\Shv_{\Nis}^{\tr}(k))$ and called by Voevodsky the \emph{homotopy $t$-structure}. The $\A^1$-invariance of the cohomology of homotopy invariant sheaves can be rephrased by saying that a sheaf $F\in \HI_{\Nis}(k)$, seen as object of $\D(\Shv_{\Nis}^{\tr}(k))$, is \emph{local} with respect to the Bousfield localization of $\D(\Shv_{\Nis}^{\tr}(k))$ over  the complexes $(\A^1_X)[n] \to X[n]$ for $X\in \Sm(k)$.

Much work has been done around the homotopy $t$-structure, including D\'eglise's extension to the non-effective version of $\DM(k)$ and the identification of its heart with the category of Rost's cycle modules \cite{DegModHomot}, and Morel's work on the stable homotopy category $\mathbf{SH}(k)$ \cite{Morelconectivity}. %and Ayoub's  work in the context of stable homotopy functors \cite{AyoubThesis1}, \cite{AyoubThesis2}. General results were later obtained by Bondarko and D\'eglise \cite{BondarkoDeglise}, where homotopy $t$-structures on various kind of motivic categories were constructed. 
In informal terms, we can interpret the existence of the homotopy $t$-structure as  a manifestation of the interplay between the Postnikov truncation functors $\tau_{\leq n}$ and the $\A^1$-localization functor on the derived category $\D(\Shv_{\Nis}^{\tr}(k))$. This interplay is precisely expressed by Morel's connectivity theorem.

Voevodsky's category of motives over a field has been recently extended to the setting of logarithmic algebraic geometry in \cite{bpo}. The basic objects in this context are no longer smooth $k$-schemes but rather fine and saturated log schemes, log smooth over a base considered with trivial log structure (typically, the base is a perfect field). The Nisnevich topology on the underlying schemes defines naturally a topology, called the \emph{strict Nisnevich topology}, $\sNis$ for short. This topology is not enough to guarantee that the resulting category of motives satisfies a number of nice properties, and needs to be replaced with a subtle variant, the \emph{dividing Nisnevich topology}, $\dNis$ for short, with additional covers given by certain blow-ups with center in the support of the log structure. The affine line $\A^1$ is replaced by its compactified avatar, i.e. the log scheme $\bcube=(\P^1, \infty)$ obtained by considering the compactifying log structure along the embedding $\A^1\hookrightarrow \P^1$. The category of log motives $\lDM(k, \Lambda)$ (with transfers) is then defined as the homotopy category of the $(\dNis, \bcube)$-local model structure on the category of (unbounded) chain complexes of presheaves with logarithmic transfers, $\mathbf{C}(\PSh^{\rm ltr}(k, \Lambda))$ for $\Lambda$ a ring of coefficients. See \cite[4-5]{bpo} and Section \ref{sec:preliminaries} below for more details. The variant without transfers will be denoted $\lDA(k, \Lambda)$, and it is obtained as Bousfield localization of the category of (unbounded) chain complexes of presheaves $\mathbf{C}(\PSh^{\rm log}(k, \Lambda))$.
\medskip

The goal of this paper is to develop in the logarithmic context the analogue of Voevodsky's homotopy $t$-structure, and to derive some consequences from this. As discussed above, the homotopy $t$-structure on (usual) motives is induced by the standard $t$-structure on the derived category of sheaves. In order to restrict this $t$-structure to the subcategory of local objects, one needs to understand how much connectivity (with respect to the homology sheaves) is lost after taking a fibrant replacement for the $(\A^1, \Nis)$-local model structure. This is the content of Morel's connectivity theorem \cite[Thm. 6.1.8]{Morelconectivity}.

Our first main result is the following logarithmic variant.

\begin{thm}(see Theorem \ref{connectivity})
\label{connectivityintro}Assume that $k$ is a perfect\footnote{If $\ch(k)$ is invertible in $\Lambda$, this assumption can be relaxed since $\lCor(k,\Lambda)\cong \lCor(k^{\emph{perf}},\Lambda)$} field and let $\tau \in \{\sNis, \dNis\}$.
Let $C\in \Cpx(\PShlog(k,\Lambda))$ be locally $n$-connected for the $\tau$-topology. Then any $(\tau,\bcube)$-fibrant replacement $C\to L$ is locally $n$-connected.
\end{thm}
A complex of presheaves is said to be locally $n$-connected with respect to a topology $\tau$ if the homology sheaves $a_{\tau} H_i(C)$ vanish below $n$. For the proof of Theorem \ref{connectivityintro} we follow the pattern given by Ayoub in his adaptation of Morel's argument to the $\P^1$-local theory, developed in \cite{P1loc}. In particular, the statement can be reduced to a purity result for local complexes:
\begin{thm}(see Theorem \ref{thm;purity})\label{thm;purityintro}
Let $X$ be a connected fs log smooth $k$-scheme which is essentially smooth over $k$ (in particular, the underlying scheme $\ul{X}$ is an essentially smooth $k$-scheme) such that $\underline{X}$ is an henselian local scheme. Then the map
\[
H_i(C(X))\to H_i(C(\eta_{X},\triv))
\]
is injective for every $(\sNis, \bcube)$-fibrant complex of presheaves $C\in \Cpx(\PShlog(k,\Lambda))$. 
\end{thm}
Here, we write $\eta_X$ for the generic point of $\ul{X}$, and $(\eta_{X}, \triv)$ for $\eta_X$ seen as a log scheme with trivial log structure. The proof is quite long, and for it we use in an essential way the results developed in \cite{bpo}, such as the existence of a number of distinguished triangles in $\lDA(k)$ and a description of the motivic Thom spaces \cite[7.4]{bpo}: in particular, new ingredients (compared to the argument given by Morel or Ayoub) are required when the log structure on $X$ is not trivial.

We remark that the original formulation of Morel's connectivity theorem was given for the  $\A^1$-localization of  presheaves of $S^1$-spectra, rather than presheaves of chain complexes. The arguments given in this paper can be easily adapted to that context. Since our main application is about the motivic category introduced in \cite{bpo}, we decided to state the results  for $\Cpx(\PShlog(k,\Lambda))$.

Having the analogue of Morel's connectivity theorem at disposal, it is possible to characterize $\bcube$-local complexes of sheaves:
\begin{cor}(see Corollary \ref{heart})
\label{heartintro}
Let $C\in \D_{\dNis}(\PSh^{\rm t}(k, \Lambda))$ where $t\in \{{\rm log}, {\rm ltr}\}$. Then the following are equivalent:
\begin{enumerate}
    \item[(a)] $C$ is $\bcube$-local
    \item[(b)] the homology sheaves $a_{\dNis}H_iC$ are strictly $\bcube$-invariant for every $i\in \Z$, i.e. their cohomology presheaves are $\bcube$-invariant.
\end{enumerate}
\end{cor}
We can then consider the inclusions
\begin{align*}
\lDA(k,\Lambda) & \hookrightarrow \D_{\dNis}(\PSh^{\rm log}(k, \Lambda))  \\
\lDM(k,\Lambda) &\hookrightarrow\D_{\dNis}(\PSh^{\rm ltr}(k, \Lambda)) 
\end{align*}
that identify $\lDA(k,\Lambda)$ and $\lDM(k,\Lambda)$ with the subcategories of $\bcube$-local complexes. Using Theorem \ref{connectivityintro} it is easy to show that the truncation functors $\tau_{\leq n}$ and $\tau_{\geq n}$ preserve  the categories of  $\bcube$-local complexes, and therefore that the standard $t$-structures on the categories of (pre)sheaves induce the desired  homotopy $t$-structure on log motives. We denote by $\CI^{\log}_{\dNis}$ (and by $\CI^{\rm ltr}_{\dNis}$ for the variant with transfers) its heart, which is then identified with the category of strictly $\bcube$-invariant $\dNis$-sheaves. It follows from the fact that the $t$-structures are compatible with colimits (in the sense of \cite{ha}) that $\CI^{\log}_{\dNis}$ and $\CI^{\rm ltr}_{\dNis}$ are Grothendieck abelian categories. See Theorem \ref{tstructure}. In particular, the inclusions
\begin{align*}
i\colon \CI^{\log}_{\dNis} & \hookrightarrow \Shv_{\dNis}^{\log}(k, \Lambda) \\
i^{\rm tr}\colon \CI^{\rm ltr}_{\dNis} &\hookrightarrow \Shv_{\dNis}^{\rm ltr}(k, \Lambda)
\end{align*}
admit both a left and a right adjoint. %, analogous to Voevodsky's $h_0^{\A^1}(-)$ and $h^0_{\A^1}(-)$ functors, constructing respectively the maximal $\A^1$-invariant quotient and the maximal $\A^1$-invariant subsheaf of any Nisnevich sheaf.
Objects of $\CIlog_{\dNis}$ and of $\CIltr_{\dNis}$ satisfy the following \emph{purity} property.
\begin{thm}(see Theorem \ref{nonderivedpurity})
\label{nonderivedpurity-intro}
Let $F\in \CIlog_{\dNis}$ (resp. $F\in \CIltr_{\dNis}$). Then for all $X\in \SmlSm(k)$ (see the notation below) and $U\subseteq X$ an open dense, the restriction $F(X)\to F(U)$ is injective.
\end{thm}

In \cite{bpo}, a comparison functor
\[R^{\bcube}\omega^*\colon \DM(k, \Lambda) \to \lDM(k, \Lambda)\] 
has been constructed. Under resolution of singularities, it is known that $R^{\bcube}\omega^*$ is fully faithful, and it identifies $\DM(k,\Lambda)$ with the subcategory of $(\A^1, \triv)$-local objects in $\lDM(k, \Lambda)$ (see \cite[Thm. 8.2.16]{bpo} and the results quoted there). Even without knowing that  $R^{\bcube}\omega^*$  is a full embedding, we can show that it is $t$-exact with respect to the  homotopy $t$-structures on both sides. In fact, when  $R^{\bcube}\omega^*$  is an embedding, it is straightforward to conclude that Voevodsky's homotopy $t$-structure is induced by the $t$-structure on $\lDM(k, \Lambda)$ via $R^{\bcube}\omega^*$. See Prop. \ref{prop;omeganis}.
\medskip

The good properties of the category of strictly $\bcube$-invariant sheaves $\CI^{\rm ltr}_{\dNis}$, deduced from the identification with the heart of the homotopy $t$-structure, allow us to make a further comparison with the category $\RSC_{\Nis}$ of \emph{reciprocity sheaves} of Kahn--Saito--Yamazaki. This is an abelian subcategory of the category of Nisnevich sheaves with transfers $\Shv_{\Nis}^{\tr}(k)$, whose objects satisfy a certain restriction on their sections inspired by the Rosenlicht--Serre theorem on reciprocity for morphisms from curves to commutative algebraic groups \cite[III]{Serre-GACC}. See \cite{KSY2} and the recollection paragraph below. 

In \cite{shuji}, S.\ Saito constructed an exact and fully faithful functor 
\begin{equation}\label{eq:shujifunctorintro}
\Log:\RSC_{\Nis}(k)\to \Shv_{\dNis}^{\rm ltr}(k,\Z)
\end{equation}
having as essential image a subcategory of $\CI^{\rm ltr}_{\dNis}$. In Section \ref{sec:ApplRSC} we study its pro-left adjoint $\Rsc\colon \Shv_{\dNis}^{\rm ltr}(k,\Z)\to \pro\RSC_{\Nis}$ and in particular its behavior with respect to the lax symmetric monoidal structure $(-, -)_{\RSC_{\Nis}}$ constructed in \cite{RSY}. 
See Theorem \ref{thm:natmapRSCtensor} and Corollary \ref{cor;monoidalfunctor}.

The category of reciprocity sheaves $\RSC_{\Nis}$ is defined in terms of the auxiliary category of \emph{modulus pairs}, building block of the theory of motives with modulus as developed in \cite{KMSY1}, \cite{KMSY2} and \cite{KMSY3}. In fact, Saito's functor \eqref{eq:shujifunctorintro} is itself defined by first ``lifting'' a reciprocity sheaf to the category of (semipure) sheaves on modulus pairs, and then applying another functor landing in $\Shv_{\dNis}^{\rm ltr}(k,\Z)$. It turns out that such detour may not be necessary.%, at least if $k$ admits resolution of singularities. 

In fact, we can look at the composite functor
\begin{equation}\label{eq:defLogRecintro}
\begin{tikzcd}\omega_{\CI}^{\log}\colon\CIltr_{\dNis} \arrow[r, hook, "i^{tr}"] & \Shltr(k, \Z) \arrow[r, "\omega_{\sharp}"] &\Shv^{\tr}_{\Nis}(k, \Z)
\end{tikzcd}
    \end{equation}
where $\omega_{\sharp}$ is the left  Kan extension of the restriction functor from smooth log schemes to smooth $k$-schemes $\omega\colon\lSm(k) \to \Sm(k)$, sending $X\in \lSm(k)$ to $X^{o}$, the open subscheme of the underlying scheme $\ul{X}$ of $X$ where the log structure is trivial. Using a comparison result from \cite{bpo} (which relies on the resolution of singularities) and our purity Theorem \ref{nonderivedpurity} we can show that $\omega_{\CI}^{\log}$ in \eqref{eq:defLogRecintro} is faithful and exact (Proposition \ref{lem:omegalowerCIffexact}). If we assume that furthermore it is full (see Conjecture \ref{conj:full}), we will denote by $\LogRec$ its essential image: if Conjecture \ref{conj:full} holds, it is a Grothendieck abelian category, that contains $\RSC_{\Nis}$ as full subcategory, see Theorem \ref{prop:RSCNissubcatLogRec}. Thanks to the purity property for strictly $\bcube$-invariant sheaves, its objects satisfy global injectivity, i.e. for every $F\in \LogRec$ and $U\subset X$ dense open subset of $X\in \Sm(k)$, the restriction map
\[F(X) \hookrightarrow F(U)\]
is injective. See  \cite{KSY2} for a similar statement for reciprocity sheaves (relying on \cite{S-purity}). In fact, this would imply that the cohomology presheaves of any reciprocity sheaf $F\in \RSC_{\Nis}$ satisfy global injectivity, see Corollary \ref{cor:globalinjcohoRSC}.

%If we denote by $i_{\RSC}$ the inclusion $\RSC_{\Nis} \subset\Shv_{\Nis}^{\tr}$, we can then identify the functor $\Log$ of \eqref{eq:shujifunctorintro} with the composite $\omega_{\log}^{\CI} \circ i_{\RSC}$, where $\omega_{\log}^{\CI} $ is the right adjoint to $\omega_{\CI}^{\log}$. The category $\LogRec$ seems to share many of the properties of $\RSC_{\Nis}$: in the rest of Section \ref{sec:logrec} we discuss some of them, in particular in relationship with the monoidal structure. See Proposition \ref{prop:tensorfactors}.

\vskip .5cm

\noindent\emph{Acknowledgements.}
The authors would like to thank Joseph Ayoub for a careful reading of a preliminary version of this manuscript, and for suggesting an improvement that allowed us to weaken the original assumptions on Theorem \ref{thm;purity}. We are also grateful to Kay R\"ulling for useful comments, to  Doosung Park for many  conversations on the subject of this paper, and to Shuji Saito for his interest in our work. Finally, we thank the referee for their detailed report.

\vskip .5cm

\subsection*{Notations and recollections on log geometry}
In the whole paper we fix a perfect base field $k$ and a commutative unital ring of coefficients $\Lambda$. Let $S$ be a Noetherian fine and saturated (fs for short) log scheme. We denote by ${\lSm}(S)$ the category of fs log smooth log schemes over $S$. We are typically interested in the case where $S = \Spec(k)$, considered as a log scheme with trivial log structure. 

For $X\in {\lSm}(S)$, we write $\ul{X}\in \mathbf{Sch}(\ul{S})$ for the underlying $\ul{S}$-scheme, where $\ul{S}$ is the scheme underlying $S$. We also write $\partial X$ for the (closed) subset of $\ul{X}$ where the log structure of $X$ is not trivial. Let $\SmlSm(S)$ be the full subcategory of $\lSm(S)$ having for objects $X\in \lSm(S)$ such that $\ul{X}$ is smooth over $\ul{S}$. By e.g.\ \cite[A.5.10]{bpo}, if $X\in \SmlSm(k)$, then $\partial X$ is a strict normal crossing divisor  on $\ul{X}$ and the log scheme $X$  is isomorphic to $(\ul{X}, \partial {X})$, i.e. to the compactifying log structure associated to the open embedding $(\ul{X}\setminus \partial X) \to \ul{X}$. If $X, Y\in \lSm(S)$, we will write $X\times_S Y$ for the fiber product of $X$ and $Y$ over $S$ computed in the category of fine and saturated log schemes: it exists by \cite[Cor. III.2.1.6]{ogu} and it is again an object of $\lSm(S)$ using \cite[Cor. IV.3.1.11]{ogu}. Unless $S$ has trivial log structure, the underlying scheme $\ul{X\times_S Y}$ does not agree with $\ul{X}\times_{\ul{S}} \ul{Y}$. See \cite[\S III.2.1]{ogu} for more details.

We denote by $\PShlog(S, \Lambda)$ the category of presheaves of $\Lambda$ modules on $\lSm(S)$. It has naturally the structure of closed monoidal category. If $\tau$ is a Grothendieck topology on $\lSm(S)$ (see below), we write $\mathbf{Shv}^{\rm log}_\tau(S, \Lambda)$ for full subcategory of $\PShlog(S, \Lambda)$ consisting of $\tau$-sheaves. We typically write $a_\tau$ for the $\tau$-sheafification functor.

Let $\widetilde{\SmlSm}(S)$ be the category of fs log smooth $S$-schemes $X$ which are essentially smooth over $S$, i.e. $X$ is a limit $\lim_{i \in I} X_i$ over a filtered set $I$, 
where $X_i \in \SmlSm(S)$ and all transition maps are strict \'etale (i.e. they are strict maps of log schemes such that the underlying maps $f_{ij}\colon \ul{X}_i\to \ul{X}_j$ are \'etale)

For $(\underline{X},\partial X)\in \SmlSm(S)$ and $x\in \underline{X}$, let $\iota\colon\Spec(\mathcal{O}_{X,x})\to \underline{X}$, be the canonical morphism. Then the local log scheme $(\Spec(\mathcal{O}_{X,x},\iota^*(\partial X))$ is in $\widetilde{\SmlSm}(S)$.

We frequently allow $F\in \PShlog(S,\Lambda)$ to take values on objects of $\widetilde{\SmlSm}(S)$ by setting
$F(X) := \colim_{i \in I} F(X_i)$ for $X$ as above.

\subsection*{Notations and recollections on reciprocity sheaves}

We briefly  recall some terminology and notations from the theory of modulus sheaves with transfers,
		see \cite{KMSY1}, \cite{KMSY2}, \cite{KSY2}, and \cite{S-purity} for details.
		
		A modulus pair $\sX=(\ol{X}, X_\infty)$ consists of 
		a separated $k$-scheme of finite type  $\ol{X}$ and an effective (or empty) Cartier divisor $X_\infty$
		such that $X:= \ol{X}\setminus |X_\infty|$ is smooth; it is called {\em proper} if $\ol{X}$ is proper over $k$.
		Given  two modulus pairs $\sX=(\ol{X}, X_\infty)$ and $\sY=(\ol{Y}, Y_\infty)$, with opens 
		$X:=\ol{X}\setminus |X_\infty|$ and $Y:=\ol{Y}\setminus |Y_\infty|$,  an
		{admissible left proper prime correspondence} from $\sX$ to $\sY$ is given by
		an integral closed subscheme $Z\subset X\times Y$ which is finite and surjective over a connected component of $X$,
		such that the normalization of its closure $\ol{Z}^N\to \ol{X}\times \ol{Y}$ is proper over $\ol{X}$ and satisfies
	\[ X_{\infty|\ol{Z}^N}\ge Y_{\infty|\ol{Z}^N},\]
		as Weil divisors on $\ol{Z}^N$, where $X_{\infty|\ol{Z}^N}$ (resp. $Y_{\infty|\ol{Z}^N}$) denotes the pullback of $X_\infty$ (resp. $Y_\infty$) to $\ol{Z}^N$. The free abelian group generated by such correspondences is denoted by $\ulMCor(\sX, \sY)$.
		By \cite[Propositions 1.2.3, 1.2.6]{KMSY1}, modulus pairs and left proper admissible correspondences define an additive category that
		we denote by $\ulMCor$. We write $\MCor$ for the full subcategory of  $\ulMCor$  whose objects are proper modulus pairs. We denote by $\tau$ the inclusion functor $\tau\colon \MCor \to \ulMCor$.
        
        We write $\mathbf{MPST}$ for the category of additive presheaves on $\MCor$ and $\mathbf{\ul{M}PST}$ for the category of additive presheaves on $\ulMCor$.  
		%The induced category of presheaves of abelian groups is denoted by $\uMPST$ (resp. $\MPST$). 
		
		%The category $\ulMCor^{\textrm{fin}}$ is the  subcategory of $\ulMCor$ with the same objects of $\ulMCor$ and the following additional condition on morphisms: they are modulus correspondences $Z=\sum n_i Z_i:\sX=(\underline{X},D_X)\to \sY=(\underline{Y},D_Y)$ such that for all $i$, the closure $\overline{Z_i}\subseteq \underline{X}\times \underline{Y}$ induces a finite prime correspondence between $\underline{X}$ and $\underline{Y}$.
		
%		We also write $\ulMSm$ for the category with the same objects as $\ulMCor$ and morphisms given by scheme-theoretic morphisms between the interiors whose graph belongs to $\ulMCor$. See \cite[Def. 1.3.2]{KMSY1}.

Let $\PSh^{\tr}(k)$ be Voevodsky's category of presheaves with transfers. Recall from \cite[Def. 1.34]{S-purity} that $F\in \PSh^{\tr}(k)$ has reciprocity if for any $X\in \Sm(k)$ and $a\in F(X) = \Hom_{\PSh^{\tr}}(\Z_{\rm tr}(X), F)$, there exists $\sX=(\ol{X}, X_\infty)\in {\bf MSm}(X)$ such that the map $\tilde{a}\colon \Z_{\rm tr}(X)\to F$ corresponding to the section $a$  factors through $h_0(\sX)$. Here ${\bf MSm}(X)$ is the category of objects $\sX \in \MCor$ such that $\ol{X}-|X_\infty| = X$, and $h_0(\sX)$ is the presheaf defined as
\[h_0(\sX)(Y) = \Coker(\ulMCor(Y\tensor \bcube, \sX)\xrightarrow{i_0^*-i_1^*} \mathbf{Cor}(Y,X)),\]
where $\bcube=(\P^1, \infty)$ (we will use the same notation for the log scheme in $\lSm(k)$), and the tensor product refers to the monoidal structure in $\ulMCor$, see \cite{KMSY1}.
It is easy to see that $\RSC$ is an abelian category, closed under sub-objects and quotients in $\PSh^{\tr}(k)$. On the other hand, it is a theorem \cite[Thm. 0.1]{S-purity} that $\RSC_{\Nis} = \RSC\cap \NST$ is also abelian, where $\NST = \Shv_{\Nis}^{\rm tr}(k)$ is the category of Nisnevich sheaves with transfers.

\section{Preliminaries on logarithmic motives}\label{sec:preliminaries} In this Section we review the construction and the basic properties of the categories $\lDM(k, \Lambda)$ and $\lDA(S, \Lambda)$ of motives, with and without transfers, as introduced in \cite{bpo}. The standard reference for properties of log schemes is \cite{ogu}. The definitions in this section work for a quite general base log scheme $S$, but in the rest of the paper we will mostly deal with the case $S=\Spec(k)$.
\subsection{Topologies on logarithmic schemes} Recall from \cite[3.1.4]{bpo} that a cartesian square of fs log schemes
\[Q = 
\begin{tikzcd}
    Y' \ar[r, "g'"]\ar[d, "f'"] & Y \ar[d, "f"]\\
    X' \ar[r, "g"] & X
\end{tikzcd}
\]
is a \emph{strict Nisnevich distinguished square} if $f$ is strict \'etale, $g$ is an open immersion and $f$ induces an isomorphism $f^{-1}(\ul{X} - g(\ul{X}'))\xrightarrow{\sim} \ul{X} - g(\ul{X}')$ for the reduced scheme structures. We say that $Q$ is a \emph{dividing distinguished square} (or \emph{elementary dividing square}) if $Y'=X'=\emptyset$ and $f$ is a surjective proper log \'etale monomorphism. According to \cite[A.11.9]{bpo}, surjective proper log \'etale monomorphisms are precisely the log modifications, in the sense of F.\ Kato \cite{FKato}. We similarly say that $Q$ is a \emph{(strict) Zariski distinguished square} if $f$ and $g$ are (strict) open immersions (note that ``strict'' here is redundant, since open immersions in the category of log schemes are automatically strict).

\begin{defn}The strict Nisnevich cd-structure (resp.\ the dividing cd-structure) is the cd structure on $\lSm(S)$ associated to the collection of strict Nisnevich distinguished squares (resp.\ of elementary dividing squares), and the dividing Nisnevich cd structure is the union of the strict Nisnevich and of the dividing cd-structures.     

The associated Grothendieck topologies on $\lSm(S)$ are called the \emph{strict Nisnevich} and the \emph{dividing Nisnevich} topology respectively. \emph{Mutatis mutandis}, we define the (strict) Zariski and the dividing Zariski topologies on $\lSm(S)$ in a similar fashion.
\end{defn}

We write $\Shv_\tau^{\rm log}(S, \Lambda)$ for the category of $\tau$ sheaves of $\Lambda$-modules on $\lSm(S)$, where $\tau$ is one of the above-defined topologies. The inclusion  $\Shv_\tau^{log}(S, \Lambda) \subset \PSh^{\rm log}(S, \Lambda) =\PSh(\lSm(S), \Lambda)$ has an exact left adjoint, $a_\tau$.

Let $S$ be a Noetherian fs log scheme such that $\ul{S}$ has finite Krull dimension. According to \cite[Prop. 3.3.30]{bpo}, the strict Nisnevich and the dividing Nisnevich cd structures on $\lSm(S)$ are complete, regular and quasi-bounded with respect to the dividing density structure (\cite[Def. 3.3.22]{bpo}). In particular, any $X\in \lSm(S)$ has finite cohomological dimension. When $S=\Spec(k)$, we can bound the $dNis$ cohomological dimension by the Krull dimension of the underlying scheme, according to the following Proposition. 
\begin{prop}(see \cite[Cor. 5.1.4]{bpo})
\label{cohdim}Let 
$F\in \Shv_{\dNis}^{\rm log}(k,\Lambda)$ and let  $X\in \lSm(k)$. Let $d= \dim(\ul{X})$. Then $\mathbf{H}^i_{\dNis}(X,F_X)=0$ for $i\geq d+1 $.
\end{prop}
%\begin{proof}This is \cite[Cor. 5.1.4]{bpo}.
%\end{proof}
\begin{remark}Since the dividing Nisnevich cd-structure is clearly squareable in the sense of \cite[Def. 3.4.2]{bpo}, one can apply \cite[Theorem 3.4.6]{bpo} to get a bound on the $dNis$ cohomological dimension for any $X\in \lSm(S)$ in terms of the dimension of a log scheme computed using the dividing density structure: this is, for a general log scheme $X$, larger than the Krull dimension of the underlying scheme $\ul{X}$ (see \cite[Ex. 3.3.25]{bpo}). In view of \cite[Rmk. 3.3.27]{bpo}, for $S=\Spec(k)$ and $X\in \lSm(k)$ such dimension agrees with the Krull dimension.
\end{remark}
The dividing Nisnevich cohomology groups are, a priori, difficult to compute.  The situation looks better for $X\in \SmlSm(k)$ thanks to the following result.
\begin{thm}\cite[Theorem 5.1.8]{bpo} Let $C$ be a bounded below complex of strict Nisnevich sheaves on $\SmlSm(k)$. Then for every $X\in \SmlSm(k)$ and $i\in \Z$ there is an isomorphism 
\begin{equation}
\label{eq:cohomologydNiscolimit}
\mathbf{H}^i_{\dNis}(X, a_{\dNis} C) = \colim_{Y \in X_{\rm div}^{Sm}} \mathbf{H}^i_{\sNis}(Y, C)
\end{equation}
where $X_{\rm div}^{Sm}$ is the category of smooth log modifications $Y\to X$ of $X$. 
\end{thm}
A formula similar to \eqref{eq:cohomologydNiscolimit} holds for $X\in \lSm(S)$ as in the following Theorem.
\begin{thm}\cite[Theorem 5.1.2]{bpo} \label{thm:cohomologycdNiscolimitS}
    Let $S$ be a Noetherian fs log scheme, and let $C$ be a bounded below complex of strict Nisnevich sheaves on $\lSm(S)$. Then for every $X\in \lSm(S)$ and $i\in \Z$ there is an isomorphism 
    \[\mathbf{H}^i_{dNis}(X, a_{dNis} C) = \colim_{Y \in X_{\rm div}} \mathbf{H}^i_{sNis}(X, C)
    \]
where the colimit runs over the set $X_{\rm div}$ of log modifications of $X$ (not necessarily smooth).
\end{thm}
The following result comes in handy to produce long exact sequences:
\begin{lemma}\label{lm;mayervietoris}
Let $X,Y\in \SmlSm$, let $D_X\subseteq \ul{X}$ and $D_Y\subseteq Y$ be Cartier divisors such that $D_X + |\partial X|$ and $D_Y + |\partial Y|$ have simple normal crossings.

Suppose that\[
\begin{tikzcd}
\ul{X}-D_X\ar[r]\ar[d]&\ul{X}\ar[d]\\
\ul{Y}-D_Y\ar[r]&\ul{Y}
\end{tikzcd}
\]
is a $\Zar$- (resp. $\Nis$-) distinguished square in $\Sm$. Let $\partial X^+$ and $\partial Y^+$ be the log structures induced by the divisors $D_X + |\partial X|$ and $D_X + |\partial Y|$, and let $X^+:=(\ul{X},\partial X^+)$ and $Y^+:=(\ul{Y},\partial Y^+)$. Then, for every complex $C\in \PShltr(k,\Lambda)$ which is $\sZar$- (resp. $\sNis$-) fibrant the following square\[
\begin{tikzcd}
C(X)\ar[r]\ar[d]&C(X^+)\ar[d]\\
C(Y)\ar[r]&C(Y^+)
\end{tikzcd}
\]
is a homotopy pullback.

\begin{proof}
	Let $\tau$ be either $\Zar$ or $\Nis$. Since the log structures on $X-D_X$ (resp $Y-D_Y$) induced by $X$ and $X^+$ (resp. $Y$ and $Y^+$) are the same, the following squares are $s\tau$-distinguished:\[
	\begin{tikzcd}
	X-D_X\ar[r]\ar[d]&X\ar[d]&&X-D_X\ar[r]\ar[d]&X^+\ar[d]\\
	Y-D_Y\ar[r]&Y&&Y-D_Y\ar[r]&Y^+
	\end{tikzcd}
	\]
	Moreover, the canonical maps $X^+\to X$ and $Y^+\to Y$, whose underlying maps of schemes are the identities of $\ul{X}$ and $\ul{Y}$, make the following diagram commutative: \[
	\begin{tikzcd}
	C(X)\ar[r]\ar[d]&C(X^+)\ar[r]\ar[d]&C(X-D_X)\ar[d]\\
	C(Y)\ar[r]&C(Y^+)\ar[r]&C(Y-D_X)
	\end{tikzcd}
	\]
	Since $C$ is $s\tau$-fibrant, the big rectangle and the square on the right are homotopy pullbacks. Hence, the square on the left is a homotopy pullback.
	\end{proof}
\end{lemma}

\subsection{log correspondences} Following \cite{bpo}, we denote by $\lCor(k)$ the category of finite log correspondences over $k$. It is a variant of the Suslin--Voevodsky category of finite correspondences $\Cor(k)$ introduced in \cite{V-TCM}, see \cite{MVW}. It has the same objects as $\lSm(k)$, and morphisms are given by the free abelian subgroup
\[ \lCor(X,Y) \subseteq  \Cor(X- \partial X, Y- \partial Y)\]
generated by elementary correspondences  $V^o\subset (X- \partial X) \times (Y- \partial Y)$ such that the closure $V\subset \ul{X}\times \ul{Y}$ is finite and surjective over (a component of) $\ul{X}$ and such that there exists a morphism of log schemes $V^N \to Y$, where $V^N$ is the fs log scheme whose underlying scheme is the normalization of $V$ and whose log structure is given by the inverse image log structure along the composition $\underline{V^N} \to \ul{X}\times \ul{Y} \to \ul{X}$. See \cite[2.1]{bpo} for more details, and for the proof that this definition gives indeed a category. 

Additive presheaves (of $\Lambda$-modules) on the category $\lCor(k)$ will be called \emph{presheaves (of $\Lambda$-modules) with log transfers}. Write $\PShltr(k, \Lambda)$ for the resulting category. We have a natural adjunction 
\[ 
\begin{tikzcd}
\PShlog(k,\Lambda)\arrow[rr,shift left=1.5ex,"\gamma_\sharp" ]\arrow[rr,"\gamma^*" description,leftarrow]\arrow[rr,shift right=1.5ex,"\gamma_*"']&& \PShltr(k,\Lambda)
\end{tikzcd}
\] 
where by convention $\gamma_\sharp$ is left adjoint to $\gamma^*$, which is left adjoint to $\gamma_*$. Here $\gamma \colon \lSm(k) \to \lCor(k)$ is the graph functor. For a topology $\tau$ on $\lSm(k)$, a presheaf with log transfers $F$ is a $\tau$-sheaf if $\gamma^* F$ is a $\tau$-sheaf. We denote by $\mathbf{Shv}_{\tau}^{\rm ltr}(k, \Lambda)\subset \PShltr(k, \Lambda)$ the subcategory of $\tau$-sheaves. By \cite[Prop. 4.5.4]{bpo} and \cite[Thm. 4.5.7]{bpo}, the strict Nisnevich and the dividing Nisnevich topology on $\lSm(k)$ are compatible with log transfers: this means in particular that the inclusion $\mathbf{Shv}_{\tau}^{\rm ltr}(k, \Lambda)\subset \PShltr(k, \Lambda)$ admits an exact left adjoint $a_\tau$ (see \cite[Prop. 4.2.10]{bpo}), and that the category $\mathbf{Shv}_{\tau}^{\rm ltr}(k, \Lambda)$ is a Grothendieck Abelian category (\cite[Prop. 4.2.12]{bpo}).
\subsection{Effective log motives}
We fix again a Noetherian fs log scheme $S$ and a field $k$, and let $\sC$ be either $\lSm(S)$ or $\lCor(k)$. We start by recalling some standard facts.
The category $\Cpx(\PSh(\sC,\Lambda))$ of unbounded complexes of presheaves is equipped with the usual global (projective) model structure $(\mathbf{W},\mathbf{Cof},\mathbf{Fib})$, where the weak equivalences are the quasi-isomorphisms and the fibrations are the degreewise surjective maps (see, for example, the remark after \cite[Thm.\ 9.3.1]{HoveyPalmieriStrickland} or \cite[Proposition 4.4.16]{AyoubThesis2}). 

Let $\tau$ be a topology on $\sC$ (and we require that $\tau$ is compatible with transfers when $\sC= \lCor(k)$). Recall that a morphism of complexes of presheaves $F\to G$ in $\Cpx(\PSh(\sC,\Lambda))$ is called a  $\tau$-local equivalence if it induces isomorphisms  $a_\tau H_i(F)\simeq a_\tau H_i(G)$ for every $i\in \Z$, where $H_i(F)$ denotes the  $i$-th homology presheaf of $F$.  

The left Bousfield localization of the global model structure on  $\Cpx(\PSh(\sC,\Lambda))$ with respect to the class of $\tau$-local equivalences exists and the resulting model structure $(\mathbf{W}_\tau,\mathbf{Cof},\mathbf{Fib}_\tau)$ is called the $\tau$-\emph{local} model structure (see, for example, \cite[Prop. 4.4.31]{AyoubThesis2}). The maps in $\mathbf{W}_\tau$ are precisely the $\tau$-local equivalences. It is well known that the homotopy category of $\Cpx(\PSh(\sC,\Lambda))$ with respect to the local model structure, denoted $\D_\tau(\PSh(\sC,\Lambda))$, is equivalent to the unbounded derived category $\D(\Shv_{\tau}(\sC  ,\Lambda))$ of the Grothendieck abelian category of $\tau$-sheaves $\Shv_{\tau}(\sC, \Lambda)$.

For any $X\in \sC$, we write
\[ R\Gamma_\tau(X, - ) \colon \D_\tau(\PSh(\sC,\Lambda))\to \D(\Lambda) \]
for the right derived functor of the global section functor $\Gamma(X,-)$. The $\tau$-(hyper) cohomology of $X$ with values in a complex of presheaves $C$ is then computed as \[ \mathbf{H}^*_{\tau}(X, a_\tau(C)) = \mathbf{H}^*(R\Gamma_\tau(X, a_\tau C)).\]

Finally, let $\bcube_S:=(\P^1_S,\infty_S)\in \sC$, with $S=\Spec(k)$ if $\sC = \lCor(k)$.
\begin{defn} The $(\tau,\bcube_S)$-local model structure on $\Cpx(\PSh(\sC,\Lambda))$ is the (left) Bousfield localization of the $\tau$-local model structure with respect to the class of maps
\[
\Lambda(\bcube_S\times_S X)[n]\to \Lambda(X)[n]
\]
for all $X\in \sC$ and $n\in \Z$. %For $S=\Spec(k)$, we have a $(\dNis,\bcube_k)$-model structure on $\Cpx(\PShltr(k,\Lambda))$ defined in the same way.
\end{defn}
General properties of the Bousfield localization (see e.g.\ \cite[D\'efinition 4.2.64, Proposition 4.2.66]{AyoubThesis2}) imply that a complex of presheaves $C$ is $(\tau,\bcube_S)$-fibrant if and only if it is $\tau$-fibrant (i.e. fibrant for the $\tau$-local model structure) and 
the morphisms $C(X)\to C(X\times_S \bcube_S)$ induced by the projection, are quasi-isomorphisms for every $X\in \sC$.

\begin{definition}\label{def:local-obj-local-complx}
(1) A complex of presheaves $C$, seen as an object of $\D_{\tau}(\PSh(\sC,\Lambda))$ is called $\bcube_S$-local if for all $X\in \sC$ the map
\[
R\Gamma_{\tau}(X,C)\to R\Gamma_{\tau}(X\times_S \bcube_S,C)
\]
is a quasi isomorphism in $\D(\Lambda)$. Equivalently, $C$ is $\bcube_S$-local if and only any $\tau$-fibrant replacement of $C$ is $(\tau,\bcube_S)$-fibrant. 

(2) Let $L\colon \D_{\tau}(\PSh(\sC,\Lambda)) \to \D_{(\dNis,\bcube_S)}(\PSh(\sC,\Lambda))$ be the localization functor.  A complex of presheaves $K$, seen as an object of $\D_{\tau}(\PSh(\sC,\Lambda))$, is called $(\tau, \bcube_S)$-locally acyclic if $L(K)$ is $\tau$-locally isomorphic to the zero complex, i.e. if $R\Gamma_{\tau}(X,L(K)) \simeq 0$ for all $X\in \sC$.
\end{definition}

\begin{defn}\label{def:logDMlogDA}
The derived category of effective log motives (with transfers) 
\[\lDM(k, \Lambda) = \mathbf{logDM}_{\dNis}^{\rm eff}(k, \Lambda) = \D_{(\dNis,\bcube)}( \Cpx(\PSh^{\rm ltr}(k,\Lambda)))\]
is the homotopy category of $\Cpx(\PSh^{\rm ltr}(k,\Lambda))$ with respect to the $(\dNis,\bcube)$-local model structure. Similarly, if $S$ is an fs Noetherian log scheme of finite Krull dimension, the category of effective log  motives without transfers $\lDA(S, \Lambda) = \mathbf{logDA}_{\dNis}^{\rm eff}(S, \Lambda)$ is the homotopy category of $\Cpx(\PSh^{\rm log}(S,\Lambda))$ with respect to the $(\dNis,\bcube_S)$-local model structure.
\end{defn}
The interested reader can verify that Definition \ref{def:logDMlogDA} is equivalent to \cite[Def. 5.2.1]{bpo}

We collect now some well-known facts about the $(\tau,\bcube_S)$-local model structure, for $\tau\in \{\sNis, \dNis\}$ that we are going to use later. Recall that $\Cpx(\PSh^{\rm log}(S,\Lambda))$ is a closed monoidal model category with respect to the global model structure by \cite[Lemme 4.4.62]{AyoubThesis2}. We write $\uHom(-, -)$ for the internal Hom functor.
%The following remarks are well known to experts, although we would like to include it here since they play a crucial role in our proof and they do not appear in the literature.
\begin{lemma} \label{rmk;magic} Let $I$ be a $\tau$-fibrant object (resp.\ a $(\tau, \bcube_S)$-fibrant object) of $\Cpx(\PSh^{\rm log}(S,\Lambda))$. Then, for every $X\in \lSm(S)$, the  complex $\uHom_S(\Lambda(X),I)$ is $\tau$-fibrant (resp.\ is $(\tau, \bcube_S)$-fibrant).
\end{lemma}
\begin{proof}
Every representable presheaf $\Lambda(X)$ is cofibrant for the projective model structure, and $- \otimes \Lambda(X)$ is a left Quillen functor. So, for every $A\to B\in \textrm{Cof}\cap W_\tau$,  we have that $A\otimes\Lambda(X)\to B\otimes\Lambda(X)$ is a trivial $\tau$-local cofibration (see \cite[Prop. 4.4.63]{AyoubThesis2}, and observe that the small site $Y_{\tau}$ is coherent for every $Y\in \lSm(S)$ since $\ul{S}$ is quasi-compact and quasi-separated, hence it has enough points by \cite[Exp. VI, Prop. 9.0]{SGA4-2} and we can apply \emph{loc.\ cit.}). In particular every $\tau$-fibrant object $I$ satisfies the lifting property:
\[
\xymatrix{
A\otimes \Lambda(X)\ar[r]\ar[d]&I\\
B\otimes \Lambda(X)\ar@{-->}[ur]
}
\]
We conclude that  $- \otimes \Lambda(X)$ is a left Quillen functor for the $\tau$-local model structure, hence $\uHom_S(\Lambda(X),-)$ is a right Quillen functor. In particular, $\uHom_S(\Lambda(X),I)$  is $\tau$-fibrant. 
% By adjunction, we conclude that $\uHom_S(\Lambda(X),I)$ has the right lifting property with respect to $\dNis$-local trivial cofibrations, i.e.\  $\uHom_S(\Lambda(X),I)$  is $\dNis$-fibrant. 
 In a similar way, if $I$ is $(\tau,\bcube)$-fibrant, we have that $\uHom_S(\Lambda(X),I)$ is $\tau$-fibrant and $\bcube$-local, so it is $(\tau,\bcube)$-fibrant.   % if $I$ $\dNis$-fibrant, every $X\in \SmlSm(S)$, every $A\to B\in \textrm{Cof}\cap \textrm{W}_\dNis$, the complex $\uHom_S(\Lambda(X),I)$ satisfies the lifting property
%\[
%\xymatrix{
%A\ar[r]\ar[d]&\uHom(\Lambda(X),I)\\
%B\ar@{-->}[ur]
%}
%\]
%This is equivalent to say that $\uHom_S(\Lambda(X),I)$ is $\dNis$-fibrant.
%
%Moreover, if $I$ is $(\dNis,\bcube)$-fibrant, we have that $\uHom_S(\Lambda(X),I)$ is $\dNis$-fibrant and $\bcube$-local, so it is $(\dNis,\bcube)$-fibrant.
\end{proof}

\begin{para}\label{rmk:changetopos}

Let $X\in \lSm(S)$ and let $\lambda\colon X\to S$ be the structural morphism. We have an induced functor $\lambda^*:\PShlog(S,\Lambda)\to \PShlog(X,\Lambda)$ given by precomposition with $\lambda$.
% \[
%\lambda^*F (U\to X) := F(U\to X\xrightarrow{\lambda} S).
%\]
The functor $\lambda^*$ and its left Kan extension $\lambda_!$ induce two adjoint functors on the categories of complexes: %Both $\lambda^*$ and its left Kan extension $\lambda_!$ are exact functors, hence they induce 

\begin{equation}\label{eq:changetopos}
\lambda_! \colon \Cpx(\PShlog(X,\Lambda))\leftrightarrows   \Cpx(\PShlog(S,\Lambda)): \lambda^*.
\end{equation}
%Consider the global model structure on both $\Cpx(\PShlog(X,\Lambda))$ and $\Cpx(\PShlog(S,\Lambda))$: 
Since $\lambda^*$ is exact,  it preserves by definition global fibrations and global weak equivalences, hence $\lambda_!$ preserves global cofibrations and \eqref{eq:changetopos} is a Quillen adjunction. In fact, by e.g.\ \cite[Thm. 4.4.51]{AyoubThesis2}, the same holds for the $\tau$-local model structure where $\tau$ is a topology on $\lSm(S)$; in particular, $\lambda^*$ preserves $\tau$-fibrant objects.

%The cofibrations in the $\dNis$-local model structures are the same as the cofibrations of the global model structures, hence $\lambda_!$ preserves $\dNis$-cofibrations.

%It is well known that for all $F\to G\in \PShlog(X,\Lambda)$ such that $a_\dNis F\cong a_\dNis G$, we have that $a_\dNis\lambda_! F\cong a_\dNis \lambda_! G$, hence $\lambda_!$ preserves $\dNis$-local equivalences. 

%We conclude $\lambda_!\dashv \lambda^*$ is a Quillen adjunction for the local model structures; in particular, $\lambda^*$ preserves $\dNis$-fibrant objects.

Finally, if $C\in \Cpx(\PShlog(S,\Lambda))$ is $\bcube_S$-local, then $\lambda^*C$ is $\bcube_X$-local, since for all $U\in \lSm(X)$
\[
\lambda^*C(U\times_X \bcube_X)=C(U\times_X X \times_S \bcube_S)\simeq C(U\times_X X)=\lambda^*C(U)
\]
We conclude that $\lambda^*$ preserves $(\tau,\bcube)$-fibrant objects as well.
\end{para}

\begin{para}
%The following result is again a consequence of the existence of the Bousfield localization:
%\fede{Secondo me questo Lemma va riformulato (forse non serve neppure sia un lemma: basta enunciare l'esistenza del funtore di localizzazione). Non credo $L$ sia definito dalla categoria dell'omotopia.}

%\begin{lemma}\label{lem;loc}
%There exists an endofunctor $L$ (resp. $L^{tr}$) on $D_{\dNis}(\Shlog(k,\Lambda))$ (resp. $D_{\dNis}(\Shltr(k,\Lambda))$) equipped with a natural transformation $\lambda:id \to L$ (resp.  $\lambda^{tr}:id \to L^{tr}$) such that for all $C\in D_{\dNis}(\Shlog(k,\Lambda))$ (resp. $K\in D_{\dNis}(\Shltr(k,\Lambda))$) the map $C\to L(C)$ (resp. $K\to L^{tr}(K)$) is a $(\dNis,\bcube)$-weak equivalence and $L(C)$ (resp. $L^{tr}(K)$) is $\bcube$-local. The pair $(L,\lambda)$ (resp. $(L^{tr},\lambda^{tr})$) is unique up to a unique natural isomorphism. 
%\end{lemma}

We end this section with a computation of the localization functor 
\[L=L_{(\tau,\bcube_S)}\colon \Cpx(\PShlog(S,\Lambda))\to \Cpx(\PShlog(S,\Lambda))_{(\tau,\bcube_S)} \subset \Cpx(\PShlog(S,\Lambda)),\] 
where $\Cpx(\PShlog(S,\Lambda))_{(\tau,\bcube_S)}$ denotes the subcategory of $(\tau,\bcube_S)$-local objects. By general properties of the Bousfield localization, $L$ comes equipped with a natural transformation $\lambda\colon id \to L$, and the pair  $(L,\lambda)$  is unique up to a unique natural isomorphism. 

An explicit description of the localization functor  has been worked out by Ayoub in \cite[Section 2]{P1loc} for the $\P^1$-localization. % and in \cite[Section 2.6]{concon} for the most general case.
We spell out the construction for presheaves without transfers and for $\tau\in \{\sNis, \dNis\}$. 

\begin{constr}\label{constr;phi}(see \cite[Construction 2.6]{P1loc})
We fix an endofunctor $(-)_{\tau}$ which gives a \emph{$\tau$-fibrant replacement}. Let $\Lambda(\bcubered_S)$ be the kernel of the map $\Lambda(\bcube_S)\to \Lambda$. For a complex $C\in \Cpx(\PShlog(S,\Lambda))$ we put\[
\Phi(C):= \textrm{Cone}\bigl\{\delta:\Lambda(\bcubered_S)\otimes_{\Lambda} \uHom_S(\Lambda(\bcubered_S),C_{\tau})\to C_{\tau}\bigl\}
\]
where $\delta$ is the counit of the adjuntion $\Lambda(\bcubered_S)\otimes_{\Lambda}\_\dashv \uHom_S(\Lambda(\bcubered_S),\_)$. 

We obtain an endofunctor $\Phi$ equipped with a natural transformation $\phi\colon id\to \Phi$, and we define the endofunctor $\Phi^{\infty}$ by taking the colimit of the following sequence:
\[
C\xrightarrow{\phi_C} \Phi(C)\xrightarrow{\phi_{\Phi(C)}} \Phi^{\circ 2}(L)\xrightarrow{\phi_{\Phi^{\circ 2}(C)}}\ldots\to\Phi^{\circ n}(C)\to\ldots
\]
By construction, the functor $\Phi^{\infty}$ comes equipped with a natural transformation $\phi^{\infty}\colon id \to \Phi^{\infty}$.
\end{constr}
\end{para}
\begin{thm}\label{thm;loc} (see \cite[Th\'eor\`eme 2.7]{P1loc})
Let $C\in \Cpx(\PShlog(S,\Lambda))$. Then $\Phi^{\infty}(C)$ is $(\tau,\bcube_S)$-fibrant and $\phi^{\infty}$ is a $(\tau,\bcube_S)$-local equivalence. In other words, the pair $(\Phi^{\infty},\phi^{\infty})$ is naturally isomorphic to the $(\tau,\bcube_S)$-localization $(L,\lambda)$.
\begin{proof}
We follow the same pattern of the proof in \cite{P1loc}, and we divide the proof in two steps. First, we need to show that for any complex of presheaves $C$, the morphism  $C\to \Phi^\infty(C)$ is a  $(\tau,\bcube_S)$-local equivalence. After that, we have to prove that $\Phi^\infty(C)$ is fibrant for the $(\tau,\bcube_S)$-local model structure.

We begin by observing that for all $F\in \Cpx(\PShlog(S,\Lambda))$, the tensor product $\Lambda(\bcubered)\otimes_\Lambda F$ is $(\tau,\bcube_S)$-locally acyclic (see Def.\ \ref{def:local-obj-local-complx}). Indeed,  the subcategory of $(\tau,\bcube_S)$-locally acyclic complexes is a triangulated subcategory of $\D_\tau(\PShlog(S,\Lambda))$ which is stable by direct sums,  and by construction it contains all the objects of the form $\Lambda(\bcubered_S)\otimes_\Lambda \Lambda(X)$ for any $X\in \lSm(S)$. 

Next, note that since the homotopy fiber of $\phi_C$ is given by 
\[\Lambda(\bcubered_S)\otimes_{\Lambda} \uHom_S(\Lambda(\bcubered_S),C_{\tau}),\] 
which is then $(\tau,\bcube_S)$-locally acyclic in virtue of what we just observed, $\phi_C$ is a $(\tau,\bcube_S)$-local equivalence for all complexes $C$. 
Since filtered colimits preserve $(\tau,\bcube_S)$-local equivalences, we conclude that the map $C\to \Phi^\infty(C)$ is a $(\tau,\bcube_S)$-local equivalence.

We move to the second part of the proof. By construction, the map $\Phi^{\circ n}(C)\to \Phi^{\circ n+1}(C)$ factors through $\Phi^{\circ n}(C)_{\tau}$, which are by construction $\tau$-fibrant. Hence $\Phi^{\infty}(C)$ is a filtered colimit of $\tau$-fibrant objects.

By Lemma \ref{lm;fcolimfib} below, filtered colimits preserve $\tau$-fibrant objects, hence $\Phi^\infty(C)$ is $\dNis$ fibrant.

Finally, we need to show that $\Phi^{\infty}(C)$ is $\bcube_S$-local, which is equivalent to show that $\uHom_S(\bcubered_S,\Phi^{\infty}(C))$ is acyclic. The argument in the proof of part (B) of \cite[Thm. 2.7]{P1loc} goes through without changes. We leave the verification to the reader. 
\begin{comment}
By adjunction, we have the triangular identity:\[
\xymatrix{
\uHom_S(\Lambda(\bcubered_S),\Phi^\infty(C))\ar[r]\ar@/_2.0pc/@{=}[dr] &\uHom_S(\Lambda(\bcubered_S),\Lambda(\bcubered_S)\otimes_\Lambda \uHom_S(\Lambda(\bcubered_S),\Phi^\infty(C))\ar[d] \\
 &\uHom_S(\bcubered_S,\Phi^\infty(C))
}
\]
We show that the vertical map is the zero map in homology. In fact, since the functor $\uHom(\Lambda(\bcube_S),\_)$ is exact, it is enough to show that the map\[
\delta^{\infty}: \uHom_S(\Lambda(\bcubered_S),\Lambda(\bcubered_S)\otimes_\Lambda \uHom_S(\Lambda(\bcubered_S),\Phi^\infty(C)) \to \Phi^\infty(C)
\]
induces the zero map in homology.

Indeed, we have that the composition\[
\xymatrix{\Lambda(\bcube_S)\otimes_\Lambda \uHom_S(\Lambda(\bcubered_S),\Phi^n(C))\ar[r]\ar@/_2.0pc/[rr]^{\delta^n} &\Phi^n(C)\ar[r]&\Phi^\infty (C)}
\]
 
factors through $\Lambda(\bcube_S)\otimes_\Lambda \uHom_S(\Lambda(\bcubered_S),\Phi^n(C)_\dNis)$, which is acyclic, hence the map $\delta^n$ is the zero map in homology. Hence we conclude since $\delta^{\infty}$ is the filtered colimit of all $\delta^n$.
\end{comment}
\end{proof}
\end{thm}

\begin{lemma}\label{lm;fcolimfib}
Let $S$ be a Noetherian scheme of finite Krull dimension and let ${(C_i)}_{i\in I}$ be a filtered diagram in $\Cpx(\PShlog(S,\Lambda))$. Assume that each $C_i$ is $\tau$-fibrant, then $\colim C_i$ is $\tau$-fibrant.
\begin{proof} We argue as in \cite[Proposition 4.5.62]{AyoubThesis2}. For $\tau=\sNis$, it follows from  \cite[\href{https://stacks.math.columbia.edu/tag/0737}{Tag 0737}]{stack}, using that $S$ is Noetherian of finite Krull dimension. For $\tau=\dNis$, we have that for every $X\in \lSm(S)$, and every filtered system $\{F_i\}_{i\in I}\in \Shlog(X,\Lambda)$,  there is a chain of isomorphisms
\begin{align*}
\mathbf{H}^i_{\dNis}(X,\colim_i F_i)\cong^{(1)} &\colim_{Y\in X} \mathbf{H}^i_{\sNis}(Y,\colim_i F_i) \\
\cong^{(2)} &\colim_i\colim_{Y\in X} 	\mathbf{H}^i_{\sNis}(Y,F_i) \\ 
\cong^{(3)} & \colim_i \mathbf{H}^i_{\dNis}(X,F_i),
\end{align*}
where (1) and (3) follow from Thm.\ \ref{thm:cohomologycdNiscolimitS}, and (2) follows from the fact that each $Y$ is also Noetherian of finite Krull dimension.
This implies that filtered colimits preserve $\dNis$-fibrant objects.
\end{proof}	
\end{lemma}

\begin{remark}\label{rmk;loctransfers}
 The proof works verbatim for $C\in \Cpx(\PShltr(k,\Lambda))$, where $\otimes_\Lambda$ is changed with the tensor product $\otimes^{\textrm{ltr}}$.
\end{remark}

\section{The connectivity theorem following Ayoub and  Morel}

In this section we show a $\bcube$-analogue of the $\A^1$-connectivity theorem of Morel \cite[Thm. 6.1.8]{Morelconectivity}, adapting the argument of Ayoub in \cite[Section 4]{P1loc}. As in \cite{P1loc}, we exploit the notion of \emph{preconnected} complex (see Definition \ref{def:preconnected-genconnected} below), and we reduce the proof of the connectivity Theorem \ref{connectivity} to a purity statement, namely Theorem \ref{thm;purity}, whose proof will be given in section \ref{section;purity}. The reader should note that while the results in this section are direct analogues of the results in \cite{P1loc}, new ingredients are necessary to prove the purity Theorem, and this is where our arguments diverge from \cite{P1loc}.

Throughout this section, we fix a ground field $k$ and we work with the categories of presheaves and $\tau$-sheaves on $\lSm(k)$ for $\tau\in \{\sZar,\sNis, \dNis\}$. Recall from \cite[Lemma 4.7.2]{bpo} that $\Shv_{\dNis}^{\rm log}(k, \Lambda)$ is equivalent to the category $\Shv_{\dNis}(\SmlSm(k), \Lambda)$, of sheaves defined on the full subcategory $\SmlSm(k)\subset \lSm(k)$.  If $X=(\underline{X},\partial X)\in \SmlSm(k)$ and $x\in \ul{X}$ is any point, we  consider $\Spec(\mathcal{O}_{X,x})\in \widetilde{\SmlSm}(k)$ with the logarithmic structure induced by the pullback of $\partial X$.

\begin{defn}
Let $n\in \Z$ and let $C$ be a complex of presheaves on a site $(\sC, \tau)$. %$C\in \Cpx(\PShlog(k,\Lambda))$. 
We say that $C$ is \emph{locally $n$-connected} (for the topology $\tau$) if the homology sheaves $a_{\tau} H_j(C)$ are zero for $j\leq n$.% If $\tau$ is the trivial topology, we simply say that $C$ is $n$-connected.
\end{defn}

The main result of this section is the following: %\fede{Qui non abbiamo bisogno di lavorare con $\SmlSm(k)$: direi che basta molto meno. Che dici?}

\begin{thm}
\label{connectivity}Assume that  $k$ is a perfect field and let $\tau \in \{\sNis, \dNis\}$. 
Let $C\in \Cpx(\PShlog(k,\Lambda))$ be locally $n$-connected for the $\tau$-topology. Then for any $(\tau,\bcube)$-fibrant replacement $C\to L$, the complex $L$ is locally $n$-connected.
\end{thm}
%In the rest of this section, we will adapt the argument of \cite[Section 4]{P1loc} to reduce the proof of Theorem \ref{connectivity} to a purity statement, which will be proved in Section \ref{section;purity}.
The proof will be given at the end of this section, assuming Theorem \ref{thm;purity}. We need some preliminary definitions, cfr.\ with  \cite[D\'ef. 4.5]{P1loc}.
\begin{defn}\label{def:preconnected-genconnected}
\begin{enumerate}
    \item[(i)] A complex $C$ of presheaves is called \emph{generically $n$-connected} if for all $X\in \SmlSm(k)$ with $\underline{X}$ connected and generic point $\eta_X$, the homology groups $H_{j}(C(\eta_X))$ are zero for $j\leq n$
    \item[(ii)] A complex $C$ of presheaves is called \emph{$n$-preconnected} if for all $X\in \widetilde{\SmlSm}(k)$, the homology groups $H_{j}(C(X))$ are zero for $j\leq n-\dim(\underline{X})$.
\end{enumerate}
\end{defn}

\begin{remark}\label{rem:def:preconnected-genconnected}
(1) Clearly $(ii)\Rightarrow (i)$ since a generic point has dimension $0$, but it is evident that $(i)\not \Rightarrow (ii)$. 

(2) If $C\in \Cpx(\PShlog(k,\Lambda))$ is locally $n$-connected for a topology $\tau$ where the cohomological dimension equals the Krull dimension of the underlying scheme, then $\mathbf{H}^i_{\tau}(X,C)=0$ for $i\geq \dim(\underline{X})-n$. Hence if $G$ is a $\tau$-fibrant replacement of $C$, $G$ is $n$-preconnected, as $H_i(G(X))=\mathbf{H}^{-i}_\tau(X,G)$.
\end{remark}

We will prove some technical result that will be needed later. Here we let $\tau$ be either $\sZar$, $\sNis$ or $\dNis$. 
\begin{prop}(see \cite[Prop. 4.8]{P1loc})
\label{technical1}
Let $C$ be an $n$-preconnected complex of presheaves, then for all $X\in \SmlSm(k)$ we have $\mathbf{H}^i_\tau(X,C)=0$ for $i\geq \dim(X)-n$.
\begin{proof} Without loss of generality we can suppose $n=-1$, i.e.\ $H_{-j}(C(X))=0$ for $j>\dim(X)$, %, i.e. $H_{j}(C(X))=0$ for $j<\dim(X)$, \fede{segno sbagliato} 
    and we need to show that $\mathbf{H}^i_\tau(X,C)=0$ for $i>\dim(X)$.
Using the descent spectral sequence $\mathbf{H}^i(X, a_\tau H_{-j}(C))\Rightarrow H_\tau^{i+j}(X, C)$, it is enough to show $\mathbf{H}^i_\tau(X,a_\tau H_{-j}(C))=0$ for $i> \dim(X)-j$. 

If $j\leq 0$, this follows Proposition \ref{cohdim}, so suppose $j>0$. By $-1$-preconnectedness, $H_{-j}(C)(\Spec(\mathcal{O}_{X,x}))=0$ for $\codim(x)<j$, since $\codim(x)=\dim(\Spec(\mathcal{O}_{X,x}))$. Using Lemma \ref{lm1} below, the statement then follows for    $\tau\in\{\sZar,\sNis\}$.

The result for $\tau=\dNis$ then can be deduced from the case $\sNis$. Indeed, 
using  Lemma \ref{lm1} below, we get 
%that for $F\in \PShlog(k,\Lambda)$ such that $F(\Spec(\mathcal{O}_{X,x}))=0$ for every $X\in \SmlSm(k)$ and $x\in X$ with $\codim_{X}(x)<j$, the cohomology groups $\mathbf{H}^{j}_{\sNis}(Y,a_{\dNis}F)$ are zero. If $F=H_{-j}(C)$,  we get
in particular $\mathbf{H}^i_{\sNis}(Y, a_{\sNis}H_{-j}(C))=0$ for all $Y\in X_{\div}^{Sm}$, whence,  since  by \eqref{eq:cohomologydNiscolimit} we have that\[
\mathbf{H}^j_{\dNis}(X,a_{\dNis}H_{-j}(C))=\colim_{Y\in X_{\div}^{\Sm}} \mathbf{H}^j_{\sNis}(Y,a_{\sNis}H_{-j}(C)),
\]
the required vanishing holds for $\dNis$ as well.
\end{proof}
\begin{lemma}
\label{lm1}
Let $\tau\in \{\sZar,\sNis\}$. Let $F$ be a presheaf of $\Lambda$-modules on the small site $X_{\tau}$ such that for every $\tau$-cover $X'\to X$ and $x'\in X'$ with $\codim_{X'}(x')<j$, we have $F(\Spec(\mathcal{O}_{X',x'}))=0$. Then $\mathbf{H}^{i}_{\tau}(X,a_{\tau}F)=0$ for $i>\dim(X)-j$.
\end{lemma}
\begin{proof}
This is  \cite[Lemma 4.9]{P1loc}; we reproduce part of the proof in our setting for completeness and to take care of some subtleties. Observe that the forgetful functor $f\colon \SmlSm(k)\to \Sm(k)$ that sends $X$ to the underlying scheme $\underline{X}$ defines an isomorphism of the small sites $f_X\colon X_{\sNis}\xrightarrow{\simeq} \ul{X}_{\rm Nis}$ (and similarly for $\sZar$ and $\Zar$): the inverse functor sends an \'etale scheme $g\colon \ul{U}\to \ul{X}$ to the morphism of log schemes $U\to X$, where $U$ is the log scheme having $\ul{U}$ as underlying scheme and log structure given by the inverse image log structure along $g$ (note that this would be false for the $\dNis$-topology). A presheaf $F$ on $X_{\sNis}$ (resp.\ on $X_{\sZar}$) gives then canonically a presheaf $\ul{F}$ on $\ul{X}_{\Nis}$ (resp.\ $\ul{X}_{\Zar}$), by setting $\ul{X}_{\Nis}\ni \ul{U} \mapsto F(U)$ (resp.\ $\ul{X}_{\Zar}\ni \ul{V} \mapsto F(V)$).  Clearly there is a canonical isomorphism $\mathbf{H}^i_{\sNis}(X, F)\cong \mathbf{H}^i_{\Nis}(\ul{X}, \ul{F})$, and by abuse of notation we drop the underline and write simply $F$ for both presheaves on $X_{\sNis}$ or on $\ul{X}_{\Nis}$ (and the same for the Zariski case).

The rest of the proof of the Lemma goes through as in \cite[Lemme 4.9]{P1loc}. See \emph{loc.cit.}\ for more details.
\end{proof}

\end{prop}
\begin{cor}
\label{dnisfib}
Let $C\in \Cpx(\PShlog(k,\Lambda))$ and let $C\to L$ be a $\tau$-fibrant replacement for $\tau\in \{\dNis, \sNis\}$. If $C$ is $n$-preconnected, then so is $L$.
\proof Follows from the fact that $H_{-j}(L(X))=\mathbf{H}^j_{\tau}(X,C)$ and Proposition \ref{technical1}.
\end{cor}
We have the following set of elementary properties of $n$-preconnected complexes.
\begin{lemma}(see \cite[Lemme 4.11]{P1loc})
\label{tensorhom}
Let $C$ be an $n$-preconnected complex of presheaves on $\lSm(k)$:
\begin{enumerate}
    \item[(i)] For all $G$ $m$-connected, then $C\otimes_\Z G$ is $(n+m+1)$-preconnected.
    \item[(ii)] For all $X\in \lSm(k)$, then $\uHom(X,C)$ is $n-\dim(X)$-preconnected.
    \item[(iii)] If $\alpha:G\to C$ is a morphism of complexes of presheaves on $\lSm(k)$ and $G$ is $(n-1)$-preconnected, then $Cone(\alpha)$ is $n$-preconnected.
\end{enumerate}
\end{lemma}

\begin{prop}(see \cite[Thm. 4.12]{P1loc})\label{lem;replacepreconnected}
Let $F\in \Cpx(\PShlog(k,\Lambda))$ $n$-preconnected and $F\to C$ be a $(\tau,\bcube)$-fibrant replacement for $\tau\in \{\dNis, \sNis\}$. Then $C$ is $n$-preconnected.

\begin{proof} The argument of \cite[Thm. 4.12]{P1loc} goes through. We have an explicit description of $C$ given by Theorem \ref{thm;loc}. Let \[
\Phi(F):= \text{Cone}(\bcube\otimes \uHom(\bcube,F_{\tau}))\to F_{\tau}
\]
where $F_{\tau}$ denotes a $\tau$-fibrant replacement of $F$, which is $n$-preconnected by Corollary \ref{dnisfib}. By Lemma \ref{tensorhom}(i)-(ii) $\bcube\otimes \uHom(\bcube,F_{\tau}))$ is $n-1$-preconnected, hence by lemma \ref{tensorhom}(iii) the cone $\Phi(F)$ is $n$-preconnected. Since $C\simeq \colim_n \Phi^{\circ n}(F)$, we conclude.
%Let $\Phi^{\circ n}(F)=\Phi(\Phi\ldots \Phi(F))$, which is $n$-connected. 
%By Theorem \ref{thm;loc} we conclude since\[
%C\simeq \colim_n \Phi^{\circ n}(F)
%\]
\end{proof}
\end{prop}

\begin{proof}[Proof of Theorem \ref{connectivity}] We give a proof for $\tau=\dNis$, since the case $\tau=\sNis$ is identical. Let $C\in\Cpx(\PShlog(k,\Lambda))$ be a complex of presheaves, locally $n$-connected for the $\dNis$ topology. Since the Krull dimension of any $X\in \lSm(k)$ agrees with the $\dNis$-cohomological dimension by Proposition \ref{cohdim}, the fact that $C$ is locally $n$-connected is equivalent to ask that, for any $X\in \widetilde{\SmlSm(k)}$, we have $\mathbf{H}^i_{\dNis}(X, C) =0$ for $i\geq \dim(X) -n$. If $G$ is a $\dNis$-local fibrant replacement of $C$, this implies that $H$ is $n$-preconnected (see Remark \ref{rem:def:preconnected-genconnected}(2)), and by Proposition \ref{lem;replacepreconnected}, any $(\dNis,\bcube)$-fibrant replacement $L$ of $C$ is then  $n$-preconnected as well. In particular, it is generically $n$-connected. 
    
We are left to show that every $(\dNis,\bcube)$-fibrant complex $L$ which is generically $n$ connected is also locally $n$-connected.    Consider  the canonical map $a_{\dNis}H_i(L)(X)\to H_i(L)(\eta_X, triv)$ for any $X\in \SmlSm(k)$ with $\ul{X}$ connected and generic point $\eta_X$. Here we write $(\eta_X, \triv)$ to indicate the essentially smooth log scheme given by the scheme $\eta_X$ with trivial log structure. By Corollary \ref{cor;purity} below (this is where the assumption that $k$ is perfect is used), this map is injective. This implies that $a_{\dNis}H_i(L)(X)=0$ for any $X\in \SmlSm(k)$ and $i<n$, i.e. the homology sheaves $a_{\dNis}H_i(L)$ are zero for $i<n$, proving the claim.
\end{proof}

%By Proposition \ref{lem;replacepreconnected} $C$ is $n$-preconnected, hence it is generically $n$-connected. 
%By Theorem \ref{thm;purity}, $a_{\dNis}H_iC(X)\subseteq H_iC(\eta_X)$, hence $a_{\dNis}H_iC=0$ for $i<n$, hence $C$ is locally $n$-conncted.

\section{Purity of logarithmic motives}\label{section;purity}

Throughout this section, we fix a base field $k$, and a $(\sNis,\bcube)$-fibrant complex of presheaves $C\in \Cpx(\PShlog(k,\Lambda))$.% \alberto{(In light of what I wrote in blue, $(\sNis,\bcube)$-fibrant is enough)}.

\begin{lemma}(see \cite[Sous-Lemme 4.14]{P1loc})
\label{lem;puritytrivial}
Let $X\in \SmlSm(k)$, $x\in \underline{X}$ and $a\in H_i (C(X))$ such that there is a dense open $U\subseteq X$ and $a_{|U}=0$. Then there exists an open neighborhood $V$ of $x$ such that $a_{|V}=0$ if either one of the following hypotheses is satisfied:

\begin{enumerate}
    \item[(i)]  $\partial X=\emptyset$, i.e. $X$ has trivial log structure.
    \item[(ii)]  $\dim(\underline{X})=1$ and $|\partial X|$ is supported on a finite number of $k$-rational points.
\end{enumerate}

\end{lemma}

\begin{proof}
Let $Z=\underline{X}-\underline{U}$. If $x\not \in Z$, there is nothing to prove, hence we can suppose $x\in Z$. We can apply Gabber's Geometric presentation theorem (\cite[Theorem 3.1.1]{bog} for $k$ infinite, \cite[Theorem 1.1]{gabbergeomfinite} for $k$ finite):  by replacing $X$ with an open neighborhood of $x$ there exist a $k$-scheme $Y$ and an \'etale morphism $e:\underline{X}\to \A^1_Y$ such that
\begin{enumerate}
    \item $Z$ maps isomorphically to $e(Z)$, i.e. there is a Nisnevich distinguished square of schemes\[
    \xymatrix{
    \underline{X}-Z\ar[r]\ar[d]&\underline{X}\ar[d]\\
    \A^1_Y-e(Z)\ar[r]&\A^1_Y
    }\]
    \item The composition\[
    Z\to \underline{X}\to \A^1_Y\to Y
    \]
    is finite
\end{enumerate}
In particular, $e(Z)$ is closed in $\P^1_Y$ and it is disjoint from $\infty_Y$. We now divide the proof in two parts.

\noindent {\bf Case (i):} Let us suppose that $X$ has trivial log structure. In this case  we have two $\sNis$-distinguished squares
\[
\xymatrix{
U\ar[r]\ar[d]&X\ar[d]\\
\P^1_Y-e(Z)\ar[r]&\P^1_Y,
} \quad\quad
\xymatrix{ U\ar[r]\ar[d]&X\ar[d] \\
(\P^1_Y-e(Z),\infty_Y)\ar[r]&(\P^1_Y,\infty_Y) 
}
\]
where $Y$ is seen as a log scheme with trivial log structure, and $\bcube_Y=    (\P^1_Y,\infty_Y)$ (resp. $(\P^1_Y-e(Z), \infty_Y)$) denotes as usual the scheme $\P^1_Y$ (resp. $\P^1_Y-e(Z)$) with compactifying log structure at $\infty_Y = \{\infty\} \times Y$.  Furthermore, the morphisms $(\P^1_Y,\infty_Y)\to \P^1_Y$ and $(\P^1_Y-e(Z),\infty_Y)\to \P^1_Y-e(Z)$, whose underlying morphisms of schemes are the identities on $\P^1_Y$ and $\P^1_Y-e(Z)$, induce a commutative diagram\[
\xymatrix{
U\ar[rr]\ar[dr]\ar[dd]&& X\ar@/_1.0pc/[dd]\ar[dr]\\
&(\P^1_Y-e(Z),\infty_Y)\ar[rr]\ar[dl]&&\bcube_Y\ar[dl]\\
\P^1_Y-e(Z)\ar[rr]&&\P^1_Y
}
\]
We define the following objects of $\D(\Lambda)$:
\begin{align*} C_Z(X)& =\hofib(C(X)\to C(U)),\\ 
    C_Z(\P^1_Y)& =\hofib( C(\P^1_Y)\to C(\P^1_Y-e(Z)))\\
    C_Z(\bcube_Y)& = \hofib(C(\bcube_Y)\to C(\P^1_Y-e(Z),\infty_Y)).
\end{align*}
Since $C$ is $(\sNis,\bcube)$-fibrant, it is in particular $\sNis$-fibrant and therefore the three left vertical arrows of the following diagram 
\begin{equation}\label{eq:lem;puritytrivial}
\xymatrix{
C_Z(\P^1_Y)\ar[r]^{\delta_{\P^1_Y}}\ar[d]^{s_{\P^1_Y}}\ar@/_2.0pc/[dd]_{t}&C(\P^1_Y)\ar[r]\ar[d]\ar@/^3.0pc/[dd]^(.3){r}&C(\P^1_Y-e(Z))\ar[d]\ar@/^3.5pc/[dd]\\
C_Z(X)\ar[r]^{\delta}&C(X)\ar[r]&C(U)\\
C_Z(\bcube_Y)\ar[r]^{\delta_{\bcube_Y}}\ar[u]_{s_{\bcube_Y}}&C(\bcube_Y)\ar[r]\ar[u]&C(\P^1_Y-e(Z),\infty_Y)\ar[u]
}
\end{equation}
denoted $s_{\P^1_Y}$, $s_{\bcube_Y}$ and $t$ respectively, are quasi-isomorphisms.

Let now $\alpha\in H_i(C(X))$ such that $\alpha_{|U}=0$, hence there exists $\beta\in H_i(C_Z(X))$ such that $\alpha=\delta(\beta)$. By the quasi-isomorphism above, there exists a unique $\beta_{\P^1_Y}\in H_iC_Z(\P^1_Y)$ such that $s_{\P^1_Y}(\beta_{\P^1_Y})=\beta$.
Let $\alpha_{\P^1_Y}=\delta_{\P^1_Y}(\beta_{\P^1_Y})$ and let $r\colon C(\P^1_Y)\to C(\P^1_Y,\infty_Y)$ be as in the diagram above. It is enough to show that $r(\alpha_{\P^1_Y})=0$ in $H_i(C(\P^1_Y,\infty_Y))$ to conclude that $\alpha=0$ in $H_i(C(X))$, using \eqref{eq:lem;puritytrivial}.

Write $C_0(\P^1_Y)$ for the homotopy fiber of $C(\P^1_Y)\xrightarrow{s_{\infty}}C(\infty_Y)$. Since  $e(Z)$ is disjoint from $\infty_Y$,  the map $\delta_{\P^1_Y}$ factors as\[
\xymatrix{C_Z(\P^1_Y)\ar[r]^{\delta_{\P^1_Y}}\ar[d]&C(\P^1_Y)\ar[r]\ar@{=}[d]&C(\P^1_Y-e(Z))\ar[d]\\
C_0(\P^1_Y)\ar[r]_{\delta_0}&C(\P^1_Y)\ar[r]^{s_{\infty}}&C(\infty_Y)
}
\]
In particular, there exists $\alpha_0\in H_i(C_0(\P^1_Y))$ such that $\delta_0(\alpha_0)=\alpha_{\P^1_Y}$. 
We will conclude by showing that $r\delta_0$ is the zero map.

Since $C$ is $\bcube$-local, the projection $\pi\colon \bcube_Y\to Y$ induces a quasi-isomorphism $\pi^*\colon C(Y) \xrightarrow{\simeq} C(\bcube_Y)$. Since clearly $\pi$ factors through the natural map $\bcube_Y \to \P^1_Y$, we have a commutative diagram
%we have a quasi isomorphism $C(\P^1_Y,\infty_Y)\xrightarrow{\substack{q\\\simeq}} C(Y,\triv)$, and the following diagram commutes:

\[
\xymatrix{
C_0(\P^1_Y)\ar[r]_{\delta_0}& C(\P^1_Y)\ar[r]^{s_\infty}\ar[d]^{r}&C(\infty_Y)\ar@{=}^{\id_Y}[d]\\
&C(\bcube_Y)&C(Y) \ar[l]_{\simeq}^{\pi^*} \ar[lu]_{\pi^*}
}
\]
and this immediately shows that $r\delta_0$ factors through an acyclic complex, as required.
% the fiber of a quasi-inverse $q$ of $\pi^*$,Hence $r\delta_0$ factors through $\hofib(q)$, which is acyclic.

\noindent {\bf Case (ii):} Let us now suppose that $\dim(\underline{X})=1$ and $\partial X$ is nontrivial, supported on a finite set of $k$-rational points. 

If $x\not \in |\partial X|$, then we can suppose $X=(\underline{X}-|\partial X|,\triv)$ and conclude as before (this in fact does not use the assumption on the dimension of $\ul{X}$). So let's assume that $x\in |\partial X|$: since $\dim(\ul{X})=1$, by replacing $X$ with an open neighborhood of $x$ we can suppose $|\partial X|=x=Z$.

After replacing $X$ with an open neighborhood of $x$ we have a $\sNis$ distinguished square\[
\xymatrix{
U\ar[r]\ar[d]&X\ar[d]\\
(\P^1_{k(x)}-e(x),\triv)\ar[r]&(\P^1_{k(x)},e(x)).
}
\]
Since $x$ is a $k$-rational point, we conclude that $k=k(x)$ and $e(x)$ is a $k$-rational point of $\P^1_k$. We drop the subscript $k$ for simplicity. Write as before:
\begin{align*}C_{\{x\}}(X)& =\hofib(C(X)\to C(U))\\ 
    C_{\{e(x)\}}(\bcube^1)& =\hofib(C(\P^1,e(x))\to C(\P^1-e(x)))
\end{align*}
%Let $C_Z(X)=\hofib(C(X)\to C(U))$ and $C_Z(\P^1)=\hofib(C(\P^1,e(x))\to C(\P^1-e(x),\triv))$.
Since $C$ is $(\sNis,\bcube)$-fibrant, hence $\sNis$ fibrant, the left vertical arrow of the following diagram 
\begin{equation}\label{eq:lem;puritytrivial2}
\xymatrix{
C_{\{e(x)\}}(\bcube^1)\ar[r]^{\delta_{\bcube_Y}}\ar[d]_{s_{\bcube_Y}}&C(\P^1,e(x))\ar[r]\ar[d]&C(\P^1-e(x))\ar[d]\\
C_{\{x\}}(X)\ar[r]^{\delta}&C(X)\ar[r]&C(U)\\
}
\end{equation}
is a quasi-isomorphism. Now, since $C$ is $\bcube$-local, the complex $C(\P^1,e(x))$ is quasi-isomorphic to $C(\Spec(k))$, and by choosing any $k$-rational point of $\P^1-e(x)$ splitting the projection $(\P^1-e(x))\to \Spec(k)$, we see that the map  
\[H_i(C(\P^1,e(x)))\to H_i(C(\P^1-e(x)))\] is injective for every $i\in \Z$. This, together with the commutativity of \eqref{eq:lem;puritytrivial2}, allows us to conclude. %  the proof in the case $\underline{X}$ is a curve and $\partial X$ is supported on a finite set of $k$-rational points.
%We have that the map $H_iC(\P^1,e(x))\to H_iC(\P^1-e(x))$ is injective for all $i$ since $(\P^1,e(x))\cong \bcube$, hence similarly to the case where $\partial X=\triv$, by choosing any $k$-rational point $\Spec(k)\to \P^1_k-e(x)$ we have that $\delta=0$. This concludes the proof in the case $\underline{X}$ is a curve and $\partial X$ is supported on a finite set of $k$-rational points.
\end{proof}

\begin{cor}
\label{cor;puritytrivialandcurves}
Let $\tau$ be either $\sZar, \sNis$ or $\dNis$
\begin{enumerate}
\item[(i)] Let $X\in \Sm(k)$. Then the following map is injective:\[
a_{\tau} H_i(C(X,\triv))\hookrightarrow H_i(C(\eta_X,\triv))
\]
where $\eta_X$ is the generic point of $X$ and $(X, \triv)$ denotes the scheme $X$ seen as log scheme with trivial log structure.
\item[(ii)] Let $X\in\SmlSm(k)$ such that $\dim(\underline{X})=1$ and $|\partial X|$ is supported on a finite number of $k$-rational points. Then the following map is injective:\[
a_{\tau} H_i(C(X))\hookrightarrow H_i(C(\eta_X,\triv))
\]
where $\eta_X$ is the generic point of $X$.
\end{enumerate}

\end{cor}
\begin{proof}
We begin by observing that maps in (i) and (ii) exist since $H_iC(\eta_X,\triv)=a_{\tau}H_iC(\eta_X,\triv)$. %since $(\eta_X,\triv)$ is $\tau$-local, hence  
We first prove (i). Let $\alpha \in a_{\tau}H_i(C(X,\triv))$ be a section such that $\alpha_{|\eta_X}=0$. Let $V\to X$ be a $\tau$-cover such that there exists $\beta\in H_i(C(V,\triv))$ mapping to the image of $\alpha$ in $a_{\tau}H_iC(V,\triv)$. Let $\coprod\eta_V$ be the disjoint union of the generic points of $V$. The following  diagram is clearly commutative\[
\xymatrix{
&H_i(C(V,\triv))\ar[d]\\
a_{\tau}H_i(C(X,\triv))\ar[r]\ar[d]&a_{\tau}H_i(C(V,\triv))\ar[d]\\
H_i(C(\eta_X,\triv))\ar[r]&\bigoplus H_i(C(\eta_V,\triv)),
}
\]
hence $\beta$ maps to zero  in $\bigoplus H_i(C(\eta_V,\triv))$. 
By Lemma \ref{lem;puritytrivial}(i), for all $x\in V$ there exists an open neighborhood $V_x$ such that $\beta\mapsto 0$ in $H_i(C(V_x,\triv))$. Since we can cover $V$ by the $V_x$, and since for every topology $\tau$ as in the statement open sieves are covering, we conclude that $\beta$ maps to zero in $a_{\tau}H_iC(V,\triv)$, hence $\alpha=0$, since $(V,\triv)\to (U,\triv)$ is a $\tau$-cover. This proves (i).
The proof of (ii) is similar, replacing $(V,\triv)$ with $(V,\partial X_{|V})$ and using Lemma \ref{lem;puritytrivial}(ii).
\end{proof}

In order to prove Theorem \ref{thm;purity},  we need the following technical result, which is well known to the experts. Recall that an henselian $k$-algebra is said to be \emph{of geometric type} if there exists $X\in \Sm(k)$ and $x\in X$ such that $R\cong \mathcal{O}_{X,x}^h$, the henselization of the local ring $\mathcal{O}_{X,x}$ at $x$.
\begin{lemma}\label{lem;saveseverything}
Let $k$ be a perfect field, $R$ a henselian $k$-algebra of geometric type. Let $\mathfrak{p}\subseteq R$ such that $R/\mathfrak{p}$ is essentially smooth over $k$. Then the map $R_{\mathfrak{p}}\to k(\mathfrak{p})$ has a section.
\end{lemma}
\begin{proof}
	Let $\kappa$ be the residue field of $R$. By the properties of henselian $k$-algebras of geometric type (see for example \cite[Lemma 6.1]{S-purity}), there exists a regular sequence $t_1\ldots t_n\in R$  such that $R\cong \kappa\{t_1\ldots t_n\}$, the henselization of the local ring of $\A^n_\kappa$ at $(0)$, and $\mathfrak{p}=(t_{r+1},\ldots t_n)$, hence $R/\mathfrak{p}\cong  \kappa\{t_1\ldots t_r\}$.
	
	In particular the map $\pi:R\to R/\mathfrak{p}$ has an evident section $s:\kappa\{t_1,\ldots,t_r\}\to \kappa\{t_1,\ldots,t_n\}$. Moreover, it is also evident that $\textrm{Im}(s)\cap \mathfrak{p}=0$, thus there exists a unique map $s':\textrm{Frac}(\kappa\{t_1\ldots t_r\})\to A_\mathfrak{p}$ such that the following diagram commutes:\[
	\xymatrix{
		\kappa\{t_1,\ldots,t_r\}\ar@/_1.0pc/[r]_{s}\ar[d]_{\subseteq}&\kappa\{t_1,\ldots,t_n\}\ar[d]\ar[l]_{\pi}\\
		\textrm{Frac}(\kappa\{t_1\ldots t_r\})\ar@/_1.0pc/[r]_{s'}&\kappa\{t_1,\ldots,t_n\}_{(t_{r+1},\ldots,t_n)}\ar[l]_{\pi'}.
	}
	\]
	Hence $s'$ is a section of $\pi'$. This, together with the isomorphism $ k(\mathfrak{p})\cong \textrm{Frac}(\kappa\{t_1\ldots t_r\})$, concludes the proof.
\end{proof}

\begin{thm}\label{thm;purity}
Let $X\in \widetilde{\SmlSm}(k)$ such that $\underline{X}$ is an henselian local scheme. Then the map
\begin{equation}\label{eq;purity}
H_i(C(X))\to H_i(C(\eta_{X},\triv))
\end{equation}
is injective.
\end{thm}
\begin{proof}
Let $|\partial X|=D_1+\ldots + D_n$. We proceed by double induction on $\dim(\underline{X})$ and $n$.

If $\dim(\underline{X})=1$ and $n=0$. Then \eqref{eq;purity} is injective by Corollary \ref{cor;puritytrivialandcurves} (i). 
Assume then that  $\dim(\underline{X})=1$ but $n>0$. Then $\partial X$ is supported on the closed point $x$ (note that $\partial X$ is automatically irreducible, since $\ul{X}$ is 1-dimensional and local).
By Lemma \ref{lem;saveseverything}, the map $\Spec(k(x))\to \underline{X}$ has a retraction, hence $X\in \widetilde{\SmlSm(k(x))}$ and $|\partial X|$ is supported on a $k(x)$-rational point.

Let $\lambda\colon \Spec(k(x))\to \Spec(k)$. Since $C$ is $(\sNis,\bcube)$-fibrant in $\Cpx(\PShlog(k,\Lambda))$, $\lambda^*C$ is $(\sNis,\bcube)$-fibrant in $\Cpx(\PShlog(k(x),\Lambda))$ (see Remark \ref{rmk:changetopos}), hence we have:\[
H_iC(X)=H_i\lambda^*C(X)\xrightarrow{(*1)} H_i\lambda^*C(\eta_{X},\triv)=H_iC(\eta_{X},\triv)
\]
and $(*1)$ is injective by Corollary \ref{cor;puritytrivialandcurves} (ii). This proves the case for $\dim(\underline{X})=1$.

Suppose now that $\dim(\underline{X})>1$ and $n=0$. Then again \eqref{eq;purity} is injective by Corollary \ref{cor;puritytrivialandcurves} (i). 
We now pass to the case $\dim(\underline{X})>1$ and $n\geq 1$. For every $1\leq r\leq n$, let $\eta_{D_r}\in \underline{X}$ be the generic point of $D_r$ and $\iota_{D_r}:D_r\to X$ the inclusion. For $Y\in \SmlSm(k)$, we write $c(Y)$ for the number of irreducible components of the strict normal crossing divisor $\partial Y$.

We  make the following Claim:
\begin{claim}\label{claim;superdirtytrick}
Assume the induction hypothesis above, \emph{i.e.} suppose that Theorem \ref{thm;purity} holds for every $Y\in \widetilde{\SmlSm(k)}$ local henselian such that $\dim(\ul{Y})\leq n-1$ and $c(Y) \geq 0$ and with $\dim(\ul{Y})=\dim(\ul{X})$ and $c(Y)\leq n-1$. % ($\dim(\underline{X})-1$, all $n$) and ($\dim(X)$, $n-1$). 
Then, for every $U\subseteq {X}$ dense open such that $U\cap D_n\subseteq D_n$ is dense, the restriction map $H_iC(X)\to H_iC(U)$ is injective.
\end{claim}
We postpone the proof of Claim \ref{claim;superdirtytrick} and complete the proof of the Theorem. % For every open $U\subset \ul{X}$ with $\eta_{D_n}\in U$ and $U\cap D_n$  dense in $D_n$, we get $H_iC(X)\hookrightarrow H_iC(U)$. 
Since filtered colimits are exact in the category of $\Lambda$-modules, we get from Claim \ref{claim;superdirtytrick} an injective map: 
\begin{equation}\label{eq:thm;purity1}
H_i(C(X))\hookrightarrow \colim_{\substack{U\subseteq X\\ \eta_{D_n}\in U}} H_i(C(U))=H_i(C(\Spec(\mathcal{O}_{\underline{X},\eta_{D_n}}),\iota_{D_n}^*\partial X)).
\end{equation}
%since for every presheaf $F\in \PShlog(k,\Lambda)$, we have that\[
%F(\Spec(\mathcal{O}_{X,\eta_{D_n}}),\iota_{D_n}^*\partial X)=\colim_{\substack{U\subseteq X\\ \eta_{D_n}\in U}} F(U,\partial X_U).
%\] 
Let $\mathcal{O}_{\ul{X},\eta_{D_n}}$ be the local ring of $\ul{X}$ at $\eta_{D_n}$: it is a discrete valuation ring with generic point $\eta_X$ and infinite residue field $k(\eta_{D_n})$. Since $\mathcal{O}_{\ul{X},\eta{D_n}}$ is the localization of a henselian $k$-algebra at a prime ideal generated by a regular sequence, we can apply Lemma \ref{lem;saveseverything} to get a map $\Spec(\mathcal{O}_{\ul{X},\eta_{D_n}})\to \Spec(k(\eta_{D_n}))$ that splits $\eta_{D_n}\to \Spec(\mathcal{O}_{\ul{X},\eta_{D_n}})$, hence 
\[(\Spec(\mathcal{O}_{\underline{X},\eta_{D_n}}),\iota_{D_n}^*\partial X)\in \widetilde{\SmlSm(k(\eta_{D_n}))}\]
and $|\iota_{D_n}^*\partial X|$ is a $k(\eta_{D_n})$-rational point. 

Let $\lambda:\Spec(k(\eta_{D_n}))\to \Spec(k)$. We argue as above: since $C$ is $(\sNis,\bcube)$-fibrant in $\Cpx(\PShlog(k,\Lambda))$, $\lambda^*C$ is $(\sNis,\bcube)$-fibrant in $\Cpx(\PShlog(k(\eta_{D_n}),\Lambda))$ (see again Remark \ref{rmk:changetopos}), hence by Corollary \ref{cor;puritytrivialandcurves} (ii) we have an injective map:
\begin{align}\label{eq:thm;purity2}
H_i(C(\Spec(\mathcal{O}_{\underline{X},\eta_{D_n}})),\iota_{D_n}^*\partial X) & = H_i(\lambda^*C(\Spec(\mathcal{O}_{\underline{X},\eta_{D_n}})),\iota_{D_n}^*\partial X)\\ \nonumber
& \hookrightarrow  H_i(\lambda^*C(\eta_{X},\triv)) \\ 
& =H_i(C(\eta_{X},\triv)).\nonumber
\end{align}
Combining \eqref{eq:thm;purity1} with \eqref{eq:thm;purity2}, we get the desired injectivity. This reduces the proof of Theorem \ref{thm;purity} to the proof of Claim \ref{claim;superdirtytrick}.

\proof[Proof of Claim \ref{claim;superdirtytrick}]

Let $X^-:=(\underline{X},\partial X^-)\in \SmlSm(k)$, where $\partial X^-$ is the strict normal crossing divisor $D_1+\ldots+D_{n-1}$. Since $c(X^-)=n-1$, by hypothesis (this is the induction assumption on the number of components of $\partial X$), the map $H_iC(X^-)\to H_iC(\eta_X,\triv)$ is injective.

Let $\ul{U}$ be an open dense subset of $\ul{X}$ such that $\underline{U}\cap D_n$ is dense in $D_n$ and $U\cap D_i=\emptyset$ if $i\neq n$, and set  $U:=(\underline{U},\partial X_{|\underline{U}})$. Write  $U^-:=(U,\partial X^-_{|U}) = (U,\triv)$. Hence we have a commutative diagram:
\begin{equation}\label{eq:injectivity_induction}
\begin{tikzcd}
    H_i(C(X^-)) \arrow[d, hook, "(1)"] \arrow[r, hook, "(2)"] \arrow[dr, hook, "(3)"] & H_i(C(X)) \arrow[d] \\
    H_i(C(U^-))\arrow[r]&H_i(C(U)),
\end{tikzcd}
\end{equation}
where $(1)$, $(2)$ and $(3)$ are injective since they all factor the injective map $H_iC(X^-)\to H_iC(\eta_X,\triv)$.
\begin{comment}
By \cite[Theorem 7.5.4.]{bpo}, since $D_n$ is normal crossing to $|\partial X^-|$ there are exact triangles in $\mathbf{logDA}^{\rm eff}(k, \Lambda)$:
\begin{align}
\label{diagram1}
M(X)\to &M(X^-)\to MTh(N_{D_n/X^-})\\ 
M(U)\to &M(U^-)\to MTh(N_{D_n\cap U/U^-}) \label{diagram2}
\end{align}
where $N_{D_n/X^-}$ (resp. $N_{D_n\cap U/U^-}$) is the normal bundle of $D_n$ in $X^-$ (resp. of $D_n\cap U$ in $U^-$). See \cite[Definition 7.4.3]{bpo}. Since $C$ is $(\dNis, \bcube)$-fibrant, there is a commutative diagram
%\[
%\begin{small}
%\xymatrix{
%&0\to H_iC(X^-)\ar[d]^{\hookrightarrow}\ar[dr]^{\hookrightarrow}\ar[r]^{\hookrightarrow}&H_iC(X)\ar[d]\ar[r]&\Hom(MTh(N_{D_n/X^-}),C[i-1])\ar[d] \to0&\\
%\ar[r]&H_iC(U^-)\ar[r]&H_iC(U)\ar[r]&\Hom(MTh(N_{D_n\cap U/U^-}),C[i-1])\ar[r]&
%}.
%\end{small}\]
\begin{equation}\label{eq;longexactthom}
\begin{tikzcd}
0\to H_i(C(X^-)) \arrow[r] \arrow[d, hook] \arrow[dr, hook] & H_i(C(X)) \arrow[r] \arrow[d] &  \Hom_{\lDA}(MTh(N_{D_n/X^-}),C[i-1]) \arrow[d] \to 0 \\
  H_i(C(U^-)) \arrow[r] & H_i(C(U)) \arrow[r] & \Hom_{\lDA}(MTh(N_{D_n\cap U/U^-}),C[i-1]) 
\end{tikzcd}
\end{equation}
where the top line is  exact by \eqref{diagram1} and actually short exact by \eqref{eq:injectivity_induction} and  the bottom line is exact by \eqref{diagram2}.
\end{comment}
%\alberto{Joseph's argument

Since $\ul{X}$ is Henselian local of dimension $r\geq n$ with closed point $x$, there exists an isomorphism $\epsilon\colon X\cong \Spec(k(x)\{t_1,\ldots,t_r\})$. Without loss of generality, we can assume that $t_r$ is a local parameter for $D_n$, so that $\epsilon$  induces an isomorphism $D_n\cong \Spec(k(x)\{t_1,\ldots,t_{r-1}\})$. Hence the map \emph{henselization at $0$} 
\[k(x)\{t_1,\ldots,t_{r-1}\}[t_r]\to k(x)\{t_1,\ldots,t_{r}\}\]
induces a pro-Nisnevich square\footnote{i.e. a cofiltered limit of Nisnevich squares} of (usual) schemes:%}
	\begin{equation}\label{joseph1}
	\begin{tikzcd}
	{X-D_n}\ar[r]\ar[d] &X\ar[d,"p"]\\
	D_n\times (\A^1-\{0\})\ar[r] &D_n\times \A^1.
	\end{tikzcd}
	\end{equation}%\alberto{
By Lemma \ref{lm;mayervietoris}, the square
\begin{equation}\label{joseph1pb}
\begin{tikzcd}
C(D_n\times (\A^1,\triv))\ar[r]\ar[d] &C(D_n\times (\A^1,0))\ar[d]\\
C(X^-)\ar[r]&C(X)
\end{tikzcd}
\end{equation}
is a filtered colimit of homotopy pullbacks, hence it is itself a homotopy pullback. %Personal: we use the fact that we are working in an accessible localization of the presentable \infty-cat Cpx(Psh(C, \Lambda)), where filtered colimits are left exact.
Consider the system $\{\ul{V}\}$ of open neighborhoods of $\eta_{D_n}\times \{0\}$ in $D_n\times \A^1$: the system $\{p^{-1}(\ul{V})\}$ is cofinal in the system of open neighborhoods of $\eta_{D_n}$ in $X$.
Given any such $\ul{V}$, let $\ul{W}_{\ul{V}}$ be the subset of $D_n\times \A^1$ given as
\[(\pi(D_n\times \{0\} \cap V)\times \A^1)\cap V\]
where $\pi\colon D_n\times \A^1\to D_n$ is the projection. It is clear by construction that $\ul{W}_{\ul{V}}$ contains $(D_n\times \{0\}\cap V)$, and in fact 
\[\ul{V}\cap (D_n \times \{0\}) = \ul{W}_{\ul{V}}\cap (D_n\times \{0\}).
\]
Since $\ul{V}$ is an open neighborhood of $\eta_{D_n}\times \{0\}$, the projection $\pi(\ul{V}\cap (D_n \times \{0\}))$ is open dense in $D_n$, and thus $\ul{W}_{\ul{V}}$ is an open neighborhood of  $\eta_{D_n} \times \{0\}$, and the system  $\{\ul{W}_{\ul{V}}\}$ is cofinal in the system of open neighborhoods of $\eta_{D_n}\times \{0\}$ in $D_n\times \A^1$. 
Since $\{p^{-1}(\ul{W}_{\ul{V}})\}$ is then cofinal in the system of open neibghborhoods of $\eta_{D_n}$ in $X$, we can conclude that there exists $\ul{W}\subseteq \ul{U}$ such that $\ul{W}\cap D_n$ is dense in $D_n$ and induces a pro-Zariski square of (usual) schemes: %}
\begin{equation}\label{joseph2}
\begin{tikzcd}
\ul{W}-(D_n\cap \ul{W})\ar[r]\ar[d]&\ul{W}\ar[d]\\
(D_n\cap \ul{W})\times (\A^1-\{0\})\ar[r]&(D_n\cap \ul{W})\times \A^1
\end{tikzcd}
\end{equation}
%\alberto{
Hence up to refining $\underline{U}$ we can suppose that $\underline{U}$ itself fits in a pro-Zariski square like \eqref{joseph2}, so again using Lemma \ref{lm;mayervietoris} and the fact that a filtered colimit of homotopy pullbacks is itself a homotopy pullback, we get the following homotopy pullback square:
%  filtered colimits preserve fibrant objects in the category of $\Lambda$-modules, the following square is a homotopy pullback:
\begin{equation}\label{joseph2pb}
\begin{tikzcd}
C((D_n\cap U)\times (\A^1,\triv))\ar[r]\ar[d] &C((D_n\cap U)\times (\A^1,0))\ar[d]\\
C(U^-)\ar[r]&C(U)
\end{tikzcd}
\end{equation}

We conclude that for $C$ $\sNis$-fibrant the squares \eqref{joseph1pb} and \eqref{joseph2pb} induce the following equivalences:
\begin{align*}
\textrm{Cofib}(C(X^-)\to C(X)) & \cong \textrm{Cofib}(C(D_n\times (\A^1,\triv))\to C(D_n\times (\A^1,0))) \\ 
&\cong \Hom^\bullet(MTh(N_{D_n/X^-}),C)
\end{align*}
\begin{align*}
\textrm{Cofib}(C(U^-)\to C(U)) & \cong \textrm{Cofib}(C((D_n\cap U)\times (\A^1,\triv))\to C((D_n\cap U)\times (\A^1,0))) \\
& \cong \Hom^\bullet(MTh(N_{D_n\cap U/U^-}),C)
\end{align*}
where the last isomorphisms come from the definition of the motivic Thom space \cite[Def. 7.4.3]{bpo}, the fact that $X$ is local and $U\subseteq X$ is an open immersion, hence $N_{D_n/X^-}\cong D_n\times \A^1$ and $N_{D_n\cap U/U^-}\cong (D_n\cap U)\times \A^1$. Here, $\Hom^\bullet(K,C) \in \D(\Lambda)$ for $K\in \Cpx(\PSh^{\log}(k, \Lambda))$ is the mapping complex.
 In particular, we get for every $i\in\Z$ the following commutative diagram:\begin{equation}\label{eq;longexactthom}
\begin{tikzcd}
0\to H_i(C(X^-)) \arrow[r] \arrow[d, hook] \arrow[dr, hook] & H_i(C(X)) \arrow[r] \arrow[d] &  \Hom(MTh(N_{D_n/X^-}),C[i-1]) \arrow[d] \to 0 \\
H_i(C(U^-)) \arrow[r] & H_i(C(U)) \arrow[r] & \Hom(MTh(N_{D_n\cap U/U^-}),C[i-1]),
\end{tikzcd}
\end{equation}
where the top horizontal sequence is exact and the bottom horizontal sequence is exact in the middle. We will now show that for every $i$, the natural map 
%\[\Hom_{\mathbf{logDA}}(MTh(N_{Z_1/X}),C[i])\to \Hom_{\mathbf{logDA}}(MTh(N_{(Z_1-Z_2)/(X-Z_2)}),C[i])\]
\[\Hom(MTh(N_{D_n/X^-}),C[i])\to \Hom(MTh(N_{(D_n-Z)/((X^-) - Z)}),C[i])\]
is injective, where $Z= \ul{X}-\ul{U}$: 
assuming this, by diagram chase in \eqref{eq;longexactthom} we finally conclude that the map $H_i(C(X))\hookrightarrow H_i(C(U))$ is injective for every $U$ as above.

%Set $\partial D_n:= \partial X^-_{|D_n}$ and $\partial D_n^\circ:=\partial X^-_{|D_n\cap U}$. 
%Since $X$ is local, $D_n$ is local, hence $N_{D_n/X^-}\cong (D_n,\partial D_n)\times \A^1$, and by e.g.\ \cite[Prop. 6.3.11]{liu} since $i\colon D_n\cap U\to D_n$ is flat (it is an open immersion), $N_{(D_n-Z)/((X^-)-Z)}\cong i^*N_{D_n/X}\cong (D_n\cap U,\partial D_n^\circ)\times \A^1$. 
We can  use \cite[Proposition 7.4.5]{bpo} (note that the condition that $C$ is $(\sNis, \bcube)$-fibrant is enough) to compute the motivic Thom spaces: we get a commutative diagram where the rows are split exact sequences
%For $S\in \SmlSm$ and a vector bundle $\mathcal{E}\to S$ on $\SmlSm$, the Thom space can be constructed as (see \cite[Proposition 7.4.5]{bpo})\[
%MTh(\mathcal{E}):=Cone((\P(\mathcal{E})\to (\P(\mathcal{E}\oplus \mathcal{O}_{S})).
%\]
\begin{small}
\begin{equation}\label{eq;splitdiagram}
\begin{tikzcd}[column sep = tiny]
0\to H_iC(D_n,\partial D_n)\arrow[r]\arrow[d]&H_iC((D_n,\partial D_n)\times\P^1)\ar[r]\arrow[d]&\Hom(MTh(N_{D_n/X^-}),C[i])\arrow[d]\to 0\\
0\to H_iC(D_n\cap U,\partial D_n^\circ)\arrow[r]&H_iC((D_n\cap U,\partial D_n^\circ)\times\P^1)\arrow[r]&\Hom(MTh(N_{D_n-Z/X^--Z}),C[i])\to0. 
\end{tikzcd}
\end{equation}
\end{small}

We have that\[
H_iC(-\times \P^1) = H_i(\uHom((\P^1,\triv),C))(-)
\] 
and $\uHom((\P^1,\triv),C)$ is $(\sNis,\bcube)$-fibrant since $C$ is (see Lemma \ref{rmk;magic}). By induction on dimension we conclude that the middle vertical map of \eqref{eq;splitdiagram} is injective, and since the rows in \eqref{eq;splitdiagram} are split-exact sequences, the right vertical map is a retract of the middle one, hence it is injective. This concludes the proof.
\end{proof}

\begin{cor}
\label{cor;purity}
Let $X\in\SmlSm(k)$ and let $\tau$ be either $\sNis$ or $\dNis$. Then the following map is injective:\[
a_{\tau} H_iC(X)\hookrightarrow H_iC(\eta_X,\triv)
\]
where $\eta_X$ is the generic point of $X$.
\end{cor}
\begin{proof}
The case where $\tau=\dNis$ follows from the case of $\sNis$. Indeed, since filtered colimits are exact in the category of $\Lambda$-modules, and since for all $Y\in X_{div}$, the map $\underline{Y}\to \underline{X}$ is birational, so that $\eta_Y=\eta_X$, we get \[
	a_{\dNis}H_iC(X)=\colim_{Y\in X_{div}} a_{\sNis}H_iC(Y)\hookrightarrow \colim_{Y\in X_{div}} H_iC(\eta_Y,\triv) = H_iC(\eta_X,\triv)
	\]
	Thus, from now on let $\tau=\sNis$. For all $x\in X$, let $X_x^h$ be the henselization of $X$ at $x$ with log structure induced by the log structure of $X$, and let $\eta(X_x^h)$ be its fraction field, which is a field extension of $\eta_X$. We have a diagram\[
	\begin{tikzcd}
	a_\tau H_iC(X)\ar[r,"(*3)"]\ar[d,"(*1)"]&H_iC(\eta_X,\triv)\ar[d]\\
	\prod_{x\in X} H_iC(X_x^h)\ar[r,"(*2)"]&\prod_{x\in X}H_iC(\eta(X_x^h),\triv)
	\end{tikzcd}
	\] 
	The map $(*1)$ is injective by the sheaf condition, the map $(*2)$ is injective by Theorem \ref{thm;purity} and the fact that injective morphisms are stable under arbitrary products in $\Lambda$-modules. Hence the map $(*3)$ is injective, which concludes the proof.
\end{proof}

\section{The homotopy t-structure on logarithmic motives}
The goal of this section is to generalize to the logarithmic setting the results of Morel on the existence of the homotopy $t$-structure on the category of motives. Having the connectivity theorem \ref{connectivity} at disposal, the proofs are fairly straightforward. 

Recall that the triangulated categories 
\begin{align}
    \label{eq:equiv_Dcat} \D_{\dNis}(\PSh^{\rm log}(\lSm(k), \Lambda)) &\cong \D_{\dNis}(\PSh(\SmlSm(k), \Lambda)),\\ \nonumber
\D_{\dNis}(\PSh^{\rm ltr}(\lSm(k), \Lambda)) & \cong \D_{\dNis}(\PSh^{\rm ltr}(\SmlSm(k), \Lambda))\end{align}
are equipped with a natural $t$-structure. The heart is equivalent to the category of $\dNis$-sheaves, (with or without transfers) 
\begin{align}
    \label{eq:equiv_heart}\Shv_{\dNis}(\lSm(k), \Lambda) &\cong \Shv_{\dNis}(\SmlSm(k), \Lambda),\\ \nonumber
    \Shv_{\dNis}^{\rm ltr}(\lSm(k), \Lambda) &\cong \Shv_{\dNis}^{\rm ltr}(\SmlSm(k), \Lambda)
\end{align}
 The equivalences follow from \cite[Lemma 4.7.2]{bpo} (without transfers) and \cite[Prop. 4.7.5]{bpo} (with transfers), which hold for the $\dNis$-topology but not for the strict Nisnevich topology. We write $\tau_{\geq n}$ and $\tau_{\leq n}$ for the (homologically graded) truncation functors on $\D_{\dNis}(\PSh(\lSm(k), \Lambda)) $ and  $\tau_{\geq n}^{tr}$ and $\tau_{\leq n}^{tr}$  for the (homologically graded) truncation functors on $\D_{\dNis}(\PSh^{\rm ltr}(\lSm(k), \Lambda))$. In view of \eqref{eq:equiv_Dcat} and \eqref{eq:equiv_heart}, we will work with the category of sheaves on $\SmlSm(k)$ without further notice, and simply write $\Shlog(k,\Lambda)$ (resp.\ $\Shltr(k,\Lambda)$) for the abelian category of sheaves (resp.\ of sheaves with log transfers). The proof of the following theorem is formally identical to \cite[Thm. 4.15]{P1loc}.
%\begin{cor}
%Let $F\in \Cpx(\PShlog(k,\Lambda))$ be locally $n$-connected and $F\to C$ be a $(\dNis,\bcube)$-fibrant replacement. Then $C$ is locally $n$-connected.
%\begin{proof}
%By Lemma \ref{lem;replacepreconnected} $C$ is $n$-preconnected, hence it is generically $n$-connected. 
%By Theorem \ref{thm;purity}, $a_{\dNis}H_iC(X)\subseteq H_iC(\eta_X)$, hence $a_{\dNis}H_iC=0$ for $i<n$, hence $C$ is locally $n$-conncted.
%\end{proof}
%\end{cor}

\begin{thm}
\label{truncation}
Let $C\in \D_{\dNis}(\PSh(\SmlSm(k), \Lambda))$, and suppose that $C$ is $\bcube$-local (see Definition \ref{def:local-obj-local-complx}). Then for all $n\in \Z$, the truncated complexes $\tau_{\geq n}C$ and $\tau_{\leq n}C$ are $\bcube$-local.
\end{thm}

\begin{proof} Up to shifting, we can clearly assume that $n=0$, and by the standard properties of the $t$-structure, it is enough to show the statement for $\tau_{\geq0}C$.
Since $C$ is $\bcube$-local, the natural map $\tau_{\geq 0}C\to C$ factors through  $L(\tau_{\geq 0}C)$ as
\[\begin{tikzcd}
\tau_{\geq 0}C \arrow[r, "e_0"] \arrow[dr, "e"] & L(\tau_{\geq 0}C)\arrow[d, "\ell"]\\
& C
\end{tikzcd}
\]
where $L(\tau_{\geq 0}C)$ is any $(\dNis, \bcube)$-fibrant replacement. 
We have by Theorem \ref{connectivity} that $L(\tau_{\geq 0}C)$ is locally $-1$-connected, so the map $\ell$ factors as\[
\begin{tikzcd}
\tau_{\geq 0}C \arrow[r, "e_0"] \arrow[dr, "e"] & L(\tau_{\geq 0}C)\arrow[d, "\ell"] \arrow[r, "\ell_0"] & \tau_{\geq 0}C \arrow[dl, "e"]\\
& C.
\end{tikzcd}
\]
%L(\tau_{\geq 0}C)\ar[r]^{\ell_0}\ar[dr]_\ell&\tau_{\geq 0}C\ar[d]^{e}\\ 
%&C
%In particular, with the commutative triangle\[\xymatrix{
%\tau_{\geq 0}C\ar[r]^{\ell_0e_0}\ar[dr]_e&\tau_{\geq 0}C\ar[d]^{e}\\
%&C}
%\]
By the universal property of $\tau_{\geq 0}$ we get that $\ell_0e_o=id_{\tau_{\geq 0}C}$. Hence, $\tau_{\geq 0}C$ is a direct summand of $L(\tau_{\geq 0}C)$, so it is $\bcube$-local as required.
\end{proof}

\begin{cor}
\label{truncationtr}
Let $C\in \D_{\dNis}(\PSh^{\rm ltr}(\SmlSm(k), \Lambda))$, and suppose that $C$ is $\bcube$-local.  Then for all $n\in \Z$, the truncated complexes $\tau^{tr}_{\geq n}C$ and $\tau^{tr}_{\leq n}C$ are $\bcube$-local.
\end{cor}
\begin{proof}
As in the proof of Theorem \ref{truncation}, it is enough to prove the statement for $\tau_{\geq 0}C$.
Recall that the graph functor $\gamma\colon \SmlSm(k) \to \SmlCor(k)$, which sends a map $X\to Y$ to the finite correspondence $X\xrightarrow{\gamma(f)}Y$ induced by its graph,  is faithful: the category $\SmlCor$ is, by definition, the full subcategory of $\lCor(k)$ consisting of all objects in $\SmlSm(k)$ (it is denoted $lCor_{SmlSm}/k$ in \cite{bpo}). Presheaves with log transfers on $\SmlSm(k)$ are, by definition, presheaves (of $\Lambda$-modules) on $\SmlCor(k)$.

The $\dNis$-topology is compatible with log transfers by \cite[Theorem 4.5.7]{bpo}, hence $\gamma$ induces a functor 
\[\gamma^*\colon  \D_{\dNis}(\PSh^{\rm ltr}(\SmlSm(k), \Lambda)) \to \D_{\dNis}(\PSh(\SmlSm(k), \Lambda)) .\] 
It is immediate so see that $\gamma^*$ is $t$-exact, conservative and preserves flasque sheaves, hence and for all $X\in \SmlSm(k)$ and $F\in  \D_{\dNis}(\PSh^{\rm ltr}(\SmlSm(k), \Lambda))$, we have
\[ R\Gamma(X,\gamma^*F)=R\Gamma(X,F)\] 
In particular $F$ is $\bcube$-local if and only if $\gamma^*F$ is. To prove the Corollary, it  is then enough to show that $\gamma^*(\tau^{tr}_{\geq 0}C)$ is $\bcube$-local. But since $\gamma^*$ is  $t$-exact, we have  $\gamma^*(\tau^{tr}_{\geq 0}C)=\tau_{\geq 0}\gamma^*C$, which is $\bcube$-local by Theorem \ref{truncation}.
\end{proof}

\begin{definition}(see \cite[Def. 5.2.2]{bpo})
Let $F\in \Shv_{\dNis}^{\rm log}(k, \Lambda)$ (resp.\ $F \in \Shv_{\dNis}^{\rm ltr}(k, \Lambda)$). We say that $F$ is  \emph{strictly $\bcube$-invariant} if the cohomology presheaves $\mathbf{H}^i_{\dNis}(\_,F)$ are $\bcube$-invariant.

Analogously to \cite{shuji}, we denote by $\CIlog_{\dNis}$ (resp. $\CIltr_{\dNis}$) the full subcategory of $\Shlog(k,\Lambda)$ (resp. $\Shltr(k,\Lambda)$) of \emph{strictly} $\bcube$-invariant sheaves.
\end{definition}
\begin{remark} Note that the above definition is slightly non-standard: in the context of reciprocity sheaves we typically write $\CI_{\Nis}$ for the category of $\bcube$-invariant Nisnevich sheaves, without ``strictness'' condition, i.e. without asking the property that the cohomology presheaves are $\bcube$-invariant. If $F\in \CI_{\Nis}$ is moreover semipure in the sense of \cite[Def. 1.28]{S-purity}, the fact that the cohomology presheaves are $\bcube$-invariant (at least when restricted to the subcategory $\ulMCor_{ls}$ defined in \emph{loc.cit.}) is indeed a difficult result due to S.~Saito, \cite[Thm. 9.3]{S-purity}. In the $\A^1$-invariant context, the analogous statement is due to Voevodsky \cite[\S 24]{MVW}.
\end{remark}
Recall that, in general, a sheaf $F$ seen as an object of $\D_{\dNis}(\PSh^{\rm t}(k, \Lambda))$ for $t\in \{{\rm log}, {\rm ltr}\}$ is $\bcube$-local if and only if it is strictly $\bcube$-invariant.
\begin{cor}
\label{heart}
Let $C\in \D_{\dNis}(\PSh^{\rm t}(k, \Lambda))$ where $t\in \{{\rm log}, {\rm ltr}\}$. Then the following are equivalent:
\begin{enumerate}
    \item[(a)] $C$ is $\bcube$-local
    \item[(b)] the homology sheaves $a_{\dNis}H_iC$ are strictly $\bcube$-invariant for every $i\in \Z$. 
\end{enumerate}
\begin{proof} The implication $(b)\Rightarrow (a)$ holds very generally, and comes from a spectral sequence argument. The converse implication $(a)\Rightarrow (b)$ comes from the fact that $a_{\dNis}H_iC[i]=\tau_{\geq i}\tau_{\leq i}C$ and Theorem \ref{truncation}.\end{proof}
\end{cor}
The following Proposition is an instance of the more general fact that if $C \rightleftarrows D$ is an adjoint pair of triangulated categories equipped with $t$-structures such that the left adjoint is right $t$-exact, then the induced functors between the hearts are still adjoint. See e.g.~\cite[Prop. 1.3.17-(iii)]{BBD}.
\begin{prop}
\label{adjunction}
The inclusion $i\colon \CI_{\dNis}^{\rm log}\hookrightarrow \Shlog(k,\Lambda)$ (resp. $i^{\textrm{tr}}\colon\CI_{\dNis}^{\rm ltr}\hookrightarrow \Shltr(k,\Lambda)$) has a left adjoint 
\[h_0:=a_{\dNis}H_0L(-[0])\]
(resp. $h_0^{\textrm{ltr}}:=a_{\dNis}H_0^{tr}L^{tr}(-[0])$).

%\begin{proof}
%We only prove the statement for $i$, since the one for $i^{\textrm{ltr}}$ is identical.

%Let $F,G\in \Shlog(k,\Lambda)$ and suppose that $G\in \CIlog_{\dNis}$. In particular, $G[0]$ is $\bcube$-local as object of $\D_{\dNis}(\PSh^{\rm t}(k, \Lambda))$ thanks to Corollary \ref{heart}. By  Theorem \ref{truncation} we have that $\tau_{\geq 0}L(F[0])=a_{\dNis}H_0 L(F[0])[0]$. Then
%\begin{align*}
 %   \Hom_{\PShlog}(F, G) &= \Hom_{\D_{\dNis}(\PShlog)}(F[0],G[0])\\
 %  &=\Hom_{\D_{\dNis}(\PShlog)}(\tau_{\geq 0} L(F[0]),G[0])\\
   % &=\Hom_{\D_{\dNis}(\PShlog)}(a_{\dNis}H_0L(F[0])[0],G[0])\\
    %&=\Hom_{\PShlog}(a_{\dNis}H_0L(F[0]),G)\\
    %&=\Hom_{\CIlog_{\dNis}}(a_{\dNis}H_0L(F[0]),G),
%\end{align*}
%completing the proof for the left adjoint.
%\end{proof}
\end{prop}

We can finally state the promised result on the existence of the $t$-structure on the category of motives.
\begin{thm}
\label{tstructure}
Consider the inclusions
\begin{align}
\lDA(k,\Lambda) & \hookrightarrow \D_{\dNis}(\PSh^{\rm log}(k, \Lambda))  \\
\lDM(k,\Lambda) &\hookrightarrow\D_{\dNis}(\PSh^{\rm ltr}(k, \Lambda)) 
\end{align}
that identify $\lDA(k,\Lambda)$ (resp. $\lDM(k,\Lambda)$) with the subcategory of $\bcube$-local complexes. Then the standard $t$-structure of $\D_{\dNis}(\PSh^{\rm log}(k, \Lambda)) $ (resp. of $\D_{\dNis}(\PSh^{\rm ltr}(k, \Lambda)) $) restricts to a $t$-structure on the category of motives $\lDA(k,\Lambda)$ (resp. $\lDM(k,\Lambda)$), called the \emph{homotopy $t$-structure.}

The heart of this $t$-structre is naturally equivalent to $\CIlog_{\dNis}$ (resp. $\CIltr_{\dNis}$), which is then a Grothendieck abelian category. 

\begin{proof} The first assertion follows directly from Theorem \ref{truncation} (resp. Corollary \ref{truncationtr}), the second from Corollary \ref{heart} and Proposition \ref{adjunction}.	
The fact that the heart of a
$t$-structure is abelian is well-known \cite{BBD}.

Next, note that the homotopy $t$-structure is clearly accessible in the sense of \cite[Definition 1.4.4.12]{ha}.
	
	Moreover, filtered colimits commute with cohomology, hence if $\{F_\alpha\}$ is a filtered system of $(\dNis,\bcube)$ fibrant objects, then $\colim F_\alpha$ is $(\dNis,\bcube)$ fibrant since it is $\dNis$-fibrant (as observed in the proof of Theorem \ref{thm;loc}) and\[
	\mathbf{H}^i(X,\colim F_\alpha)=\colim\mathbf{H}^i(X, F_\alpha)\cong \mathbf{H}^i(X\times \bcube,\colim F_\alpha).
	\] 
	So if $H_i^{\bcube}F_\alpha = 0$ for $i\geq 0$ and all $\alpha$, then\[
	H_i^{\bcube}(\colim F_\alpha) = H_i(\colim F_{\alpha}) = \colim H_i F_{\alpha} = 0
	\]
	Hence the $t$-structure is compatible with colimits in the sense of \cite[1.3.5.20]{ha}.
	
	In particular, as observed in \cite[Remark 1.3.5.23]{ha}, the categories $\CIlog_{\dNis}$ and $\CIltr_{\dNis}$ are Grothendieck abelian categories.
\end{proof}
\end{thm}

\begin{prop}\label{prop;rightadjoint}
	The inclusions $i\colon \CI_{\dNis}^{\rm log}\hookrightarrow \Shlog(k,\Lambda)$ (resp. $i^{\textrm{tr}}\colon\CI_{\dNis}^{\rm ltr}\hookrightarrow \Shltr(k,\Lambda)$) has a right adjoint $h^0$ (resp. $h^0_{\tr}$) such that for $F\in \Shlog$ (resp. $\Shltr$):\[ih^0F(X)=\Hom_{\Shlog}(h_0(a_{\dNis}(\Lambda(X))),F)\]
	(resp. $i^{\textrm{ltr}}h^0_{\textrm{ltr}}F(X)=  \Hom_{\Shltr}(i^{\textrm{ltr}}h_0^{\textrm{ltr}}(a_{\dNis}(\Lambda_{\tr}(-))),F)$).
	\end{prop}
\begin{proof} We prove the assertion for $\CI_{\dNis}^{\rm log}$, since the statement for $\CI_{\dNis}^{\rm ltr}$ is identical. First, note that if the right adjoint $h^0$ exists, then for $F\in \Shv_\dNis$ and $X\in \SmlSm(k)$, we have
	\begin{multline*}
	i h^0(F) (X) = \Hom_{\Shv_{\dNis}}(a_{\dNis}\Lambda(X),ih^0F) =\Hom_{\CIlog_{\dNis}}(h_0(a_{\dNis}\Lambda(X)),h^0F)=\\
	\Hom_{\Shv_{\dNis}}(ih_0(a_{\dNis}\Lambda(X)),F).
	\end{multline*}
	as required. Hence we only have to prove that $h^0$ exists.
	
	By the Special Adjoint Functor Theorem (see \cite[p. 130]{maclane}), a functor between two Grothendieck abelian categories has a right adjoint if and only if it preserves all (small) colimits, so we need to show that this holds for  $i\colon \CI_{\dNis}^{\rm log}\to \Shlog(k,\Lambda)$, i.e. that  $\CIlog$ is closed under small colimits in $\Shlog(k, \Lambda)$.
	
	As it was observed in the proof of Theorem \ref{tstructure}, $\CIlog$ is stable under filtered colimits. Since (small) colimits are filtered colimits of finite colimits, it is enough to show that $\CIlog$ is stable under finite limits. Since it is an abelian subcategory, it is enough to show that it is stable under cokernels.
	
	Let $F,G\in \CIlog_{\dNis}$ and let $F\to G$ be a map in $\Shv_{\dNis}$. Then we have that\[
	\textrm{coker}_{\Shv_{\dNis}}(F\to G)=a_{\dNis}H_0(\textrm{Cofib}(F_{\dNis}\to G_{\dNis})),
	\]
	where $F_{\dNis}$ and $G_{\dNis}$ denote the $\dNis$-fibrant replacements. Since $F$ and $G$ are strictly $\bcube$-local, $F_{\dNis}$ and $G_{\dNis}$ are $(\dNis,\bcube)$-fibrant, hence $\textrm{Cofib}(F_{\dNis}\to G_{\dNis})$ is also $(\dNis,\bcube)$-fibrant.
	
	In particular,\begin{multline*}
	\textrm{coker}_{\Shv_{\dNis}}(F\to G)\simeq a_{\dNis}H_0(\textrm{Cofib}(F_{\dNis}\to G_{\dNis}))\\ \simeq a_{\dNis}H_0(L^{\tr}(\textrm{Cofib}(F_{\dNis}\to G_{\dNis}))) \overset{(*)}{\simeq} \textrm{coker}_{\CIlog_{\dNis}}(F\to G),
	\end{multline*}
	where $(*)$ comes from Proposition \ref{adjunction} and the fact that $h_0$ preserves colimits.
	\end{proof}

\begin{cor}\label{cor;ext} Let $G\in \CI_{\dNis}^{\rm log}$ (resp.\ $G\in \CI_{\dNis}^{\rm ltr}$). % and assume that $a_{\dNis}G \in \CI_{\dNis}^{\rm log}$ (resp.\ that $a_{\dNis}G\in \CI_{\dNis}^{\rm ltr}$).  
%Let $F,G\in \Shlog(k,\Lambda)$ (resp. $\Shltr(k,\Lambda)$) such that $G\in \CIlog$ (resp. $\CIltr$). 
Then \[\uExt^i_{\Shlog}(F,G)\in \CI^{\rm log}_{\dNis}\] (resp. $\uExt^i_{\Shltr}(F,G)\in \CI^{\rm ltr}_{\dNis}$) for every $F\in \Shlog(k,\Lambda)$ (resp.\ $F\in \Shltr(k,\Lambda)$)  
\begin{proof}
We only prove it for $\Shlog$, the proof for $\Shltr$ is identical.
Let $G[0]\to G_{\dNis}$ be a $\dNis$-fibrant replacement, hence\[
\uExt^i_{\Shlog}(F,G)=a_{\dNis}H_i(\uHom(F,G_{\dNis})).
\] 
Note that every $X\in \SmlSm(k)$, we have an isomorphism 
\[\uHom(\Lambda(X),G_{\dNis})\cong \uHom(\Lambda(X\times \bcube),G_{\dNis}),\] since by adjunction we have
\begin{align*}
\Gamma(Y,\uHom(\Lambda(X),G_{\dNis})) & \cong \Gamma(Y\times X,G_{\dNis}) \\
& \cong \Gamma(Y\times X\times \bcube,G_{\dNis}) \\ &\cong  \Gamma(Y,\uHom(\Lambda(X\times \bcube),G_{\dNis}))\\
\end{align*}
From this it easily follows that $\uHom(F,G_{\dNis})$ is $\bcube$-local, hence we conclude by Theorem \ref{tstructure} 
%\begin{align*}
%\Gamma(Y,\uHom(F,G_{\dNis}))&= \Hom(F,\uHom(\Lambda(Y),G_{\dNis}))\\
%  &=\Hom(F,\uHom(\Lambda(Y\times \bcube),G_{\dNis})) \\ &=\Gamma(Y\times \bcube,\uHom(F,G_{\dNis}))
%\end{align*}
%In particular 
\end{proof}
\end{cor}

\begin{thm}
\label{nonderivedpurity}
Let $F\in \CIlog_{\dNis}$ (resp. $F\in \CIltr_{\dNis}$). Then for all $X\in \SmlSm(k)$ and $U\subseteq X$ an open dense, the restriction $F(X)\to F(U)$ is injective.

\begin{proof} As before, we give a proof for the version without transfers. %Let $\alpha\in F(X)$ such that $\alpha_{|U}=0$. 
Let $F[0]\to G$ be a $\dNis$-fibrant replacement. Since $F[0]$ is $\bcube$-local, $G$ is $(\dNis,\bcube)$-fibrant.
Since $F=a_{\dNis}H_0G$, the  result follows from Theorem \ref{cor;purity}.
\end{proof}
\end{thm}

\subsection{Comparison with Voevodsky's motives}

Let $\Cor(k)$ be Voevodsky's category of finite correspondences over $k$ \cite[\S 1]{MVW}. We have a pair of adjoint functors:\[
\begin{tikzcd}\lambda\colon \Cor(k)\ar[r,shift right]&\lCor(k)\colon\omega \ar[l,shift right]\end{tikzcd}
\]
where $\lambda(X)=(X,\triv)$ and $\omega(X,\partial X)=X-|\partial X|$. They induce functors on the  categories of complexes of presheaves
\begin{equation}\label{eq:adjunctionomega}
\begin{tikzcd}
\Cpx(\PSh^{\rm ltr}(k,\Lambda))\arrow[rr,shift left=1.5ex,"\omega_\sharp "]\arrow[rr,"\omega^*" description,leftarrow]\arrow[rr,shift right=1.5ex,"\omega_*"']&& \Cpx(\PSh^{\rm tr}(k,\Lambda))
\end{tikzcd}
\end{equation} 
where $\omega^*$ denotes as usual the restriction functor, $\omega_\sharp$ its left Kan extension and $\omega_*$ the right Kan extension. Since $\lambda$ is left adjoint to $\omega$, we have $\lambda^*=\omega_\sharp$. By construction, $\omega^*$ and $\omega_\sharp$ are $t$-exact for the global $t$-structures.

The adjunction $(\omega_\sharp, \omega^*)$ is a Quillen adjunction with respect to the $\dNis$-local model structure on the left hand side and the $\Nis$-local model structure on the right hand side, see \cite[4.3.4]{bpo}, and with respect to the $(\dNis,\bcube)$-local model structure on the left hand side and the $(\Nis,\A^1)$-local model structure on the right side, see \cite[4.3.5]{bpo} and induces therefore the following derived adjunctions:
		\begin{equation}\label{eq:adjunctionomega-derived1} \begin{tikzcd}
		 L\omega_\sharp\colon \D_{\dNis}(\Cpx(\PSh^{\rm ltr}(k,\Lambda)) )\arrow[r,shift left=.5ex ]\arrow[r, description,leftarrow, shift right=.5ex ] & \D_{\Nis}(\Cpx(\PSh^{\rm tr}(k,\Lambda))): R\omega^*.\\
		 L^{\bcube}\omega_\sharp\colon \lDM(k,\Lambda) \arrow[r,shift left=.5ex ]\arrow[r, description,leftarrow, shift right=.5ex ] & \mathbf{DM}^{\rm eff}_{\Nis}(k,\Lambda) ) : R^\bcube\omega^*.
		\end{tikzcd}
		\end{equation}
Similar adjunctions hold for the categories without transfers.

\begin{prop}\label{prop;omeganis}
Let $F\in \Cpx(\PSh^{\rm tr}(k,\Lambda))$ (resp. $G\in \Cpx(\PSh^{\rm ltr}(k,\Lambda)$). Then $R\omega^*(F)=(\omega^*F)_{\dNis}$ (resp. $L\omega_\sharp(G)=(\omega_\sharp G)_{\Nis}$), in particular $R\omega^*$ is $t$-exact. 
\end{prop}
\begin{proof}
	Since $\omega^*$ and $\omega_\sharp$ from \eqref{eq:adjunctionomega} are $t$-exact functors, we have that for every $X\in \lSm(k)$ (resp. $Y\in \Sm(k)$):
	\begin{align*}
	H_n (\omega^*F)_{\dNis}(X)&=\mathbf{H}^{-n}_{\dNis}(X,\omega^*F) & H_n (\omega_\sharp G)_{\Nis}(Y)&=\mathbf{H}^{-n}_{\Nis}(Y,\omega_\sharp G)\\ 
	& \cong \mathbf{H}^{-n}_{\Nis}(	\omega(X),F) && \cong \mathbf{H}^{-n}_{\dNis}(\lambda(Y),G)\\
	&\cong H_n (F_{\Nis}(\omega(X))) && \cong H_n (G_{\dNis}(\lambda(X)))\\& =\omega^*(H_n F_{\Nis})(X) &&=\omega_\sharp(H_n G_{\dNis})(Y) \\ &=H_n (\omega^* F_{\Nis})(X) &&=H_n (\omega_\sharp G_{\dNis})(Y)\\& =H_n(R\omega^*F)(X) &&=H_n(L\omega_\sharp G)(Y).
	\end{align*}
	Finally, by \cite[(4.3.4)]{bpo}:\[
	a_{\dNis}H_n(R\omega^*(F))=a_{\dNis}\omega^*H_n(F)=\omega^*a_{\Nis}H_n(F)
	\] 
	since $\omega^*$ is fully faithful, we conclude that $a_{\Nis}H_n(F)=0$ if and only if $a_{\dNis}H_n(R\omega^*(F))=0$, hence $\omega^*$ is $t$-exact for the local $t$-structure. 
\end{proof}

\begin{prop}\label{prop:comparisontstructure}The functor $R^{\bcube}\omega^*$ is $t$-exact with respect to Voevodsky's homotopy $t$-structure on $\DM$ and to the homotopy $t$-structure on $\lDM$ of Theorem \ref{tstructure}.
%Let $F\in \Cpx(\PSh^{\rm tr}(k,\Lambda))$. Then $R\omega^*(F)=(\omega^*F)_{\dNis}$, in particular $R\omega^*$ is $t$-exact.
\end{prop}
\begin{proof}
If $K$ is $(\Nis, \A^1)$-fibrant, it is in particular $\Nis$-fibrant, hence by Proposition \ref{prop;omeganis}, $\omega^*K$ is $\dNis$-fibrant. Hence we have for every $n\in \Z$ and $X\in \SmlSm$,
\begin{equation}
    \mathbf{H}^{-n}_{\dNis}(X,\omega^*K)\cong\mathbf{H}^{-n}_{\Nis}(\omega(X),K)=\Hom(\Lambda_{\tr}(\omega(X))[n],K)
\end{equation}

In particular, since $\omega(X\times \bcube)=\omega(X)\times \A^1,$ we have that 
\begin{align*} \Hom((\Lambda_{\rm ltr}(X)\tensor \bcube)[n] , \omega^* K) &= \Hom(\Lambda_{\tr}(\omega(X)\tensor \A^1)[n], K) \\ &= \Hom(\Lambda_{\tr}(\omega(X))[n], K) \\ & = \Hom(\Lambda_{\rm ltr}(X)[n] , \omega^* K),\end{align*}
so $\omega^* K$ is $\bcube$-local if $K$ is $(\Nis, \A^1)$-local. It follows that $\omega^*$ sends $(\Nis,\A^1)$-weak equivalences to $(\dNis,\bcube)$-weak equivalences, so that the following diagram of triangulated categories commutes:
\begin{equation}\label{eq:diagramOmegacommutesinclusion} \begin{tikzcd}
\D_{\dNis}(\Cpx(\PSh^{\rm ltr}(k,\Lambda))) &  \D_{\Nis}(\Cpx(\PSh^{\rm tr}(k,\Lambda))) \arrow[l, "R\omega^*"] \\
\lDM(k,\Lambda) \arrow[u, hook, "\iota_{\mathbf{logDM}}"]&  \mathbf{DM}^{\rm eff}_{\Nis}(k,\Lambda) ) \arrow[u, hook,  "\iota_{\mathbf{DM}}"] \arrow[l, "R^{\bcube}\omega^*"]  
\end{tikzcd}
\end{equation}
where the vertical fully faithful functors are the right adjoint to the localizations $L^{\bcube}$ and $L^{\A^1}$ respectively. By \cite[Prop. 3.1.13]{V-TCM} and Theorem \ref{tstructure}, the $t$-structure on $\D_{\Nis}(\Cpx(\PSh^{\rm tr}(k,\Lambda)))$ (resp. on $\D_{\dNis}(\Cpx(\PSh^{\rm ltr}(k,\Lambda)))$) induces a $t$-structure on $\DM$ (resp. on $\lDM$), so that the inclusions $\iota_{\mathbf{DM}}$ and $\iota_{\mathbf{logDM}}$ are both $t$-exact.

To conclude, we need to show that $R^\bcube\omega^* \circ \tau^{\mathbf{DM}}_{\leq n} \cong \tau_{\leq n}^{\mathbf{logDM}} \circ R^{\bcube}\omega^*$. But since $R\omega^*$ is $t$-exact and \eqref{eq:diagramOmegacommutesinclusion} commutes,  we have
\begin{align*}R^{\bcube}\omega^*(\tau^{\mathbf{DM}}_{\leq n}K) &= R\omega^* \iota_{\mathbf{DM}}(\tau^{\mathbf{DM}}_{\leq n} K) = R\omega^* (\tau_{\leq n} \iota_{\mathbf{DM}}(K)) \\ &= \tau_{\leq n} R\omega^* \iota_{\mathbf{DM}}(K) = \tau_{\leq n}^{\mathbf{logDM}} R^{\bcube}\omega^*(K).\end{align*}
The same argument applies to the truncation $\tau_{\geq n}$, so that we can conclude.
\end{proof}

\begin{remark}Assume that $k$ satisfies resolution of singularities. Then the functor $R^\bcube\omega^*$ is fully faithful, and its essential image is identified with the subcategory of $\A^1$-local objects in $\lDM$ by \cite[Thm. 8.2.16]{bpo}. It follows from Proposition \ref{prop:comparisontstructure} that under $R^{\bcube}\omega^*$, the homotopy $t$-structure on $\DM$ is induced by the homotopy $t$-structure on $\lDM$. 
\end{remark}
\begin{cor}The functor $L^{\bcube}\omega_\sharp$ is right $t$-exact.
    \begin{proof}This follows immediately from the fact that its right adjoint is $t$-exact (in particular, left $t$-exact).
    \end{proof}
\end{cor}

\section{Application to reciprocity sheaves}\label{sec:ApplRSC}
In this section, we discuss some applications to the theory of reciprocity sheaves. As above, for $X\in \SmlSm(k)$, let $|\partial X|$ be the strict normal crossing divisor supporting the log structure of $X$. We will call the modulus pair $(\underline{X},|\partial X|_{\red})$ the \emph{associated reduced modulus pair}. We remark that the assignment $X\mapsto (\underline{X},|\partial X|_{\red})$ \emph{does not} give rise to a functor from $\SmlSm(k)$ to $\MCor$, since a priori there is no control on the multiplicities of the divisor $\partial X$ in the pullback along a morphism in $\SmlSm(k)$.      
%Let $\Log:\RSC_{\Nis}\to \Shltr$ be the functor defined in \cite{shuji}. Concretely, if $X=(\underline{X},\M_X)\in \SmlSm$ we have that\[
%\Log(F)(X)=\omegaCI F(\underline{X},D_{red}) 
%\]
%where $D_{red}$ is the reduced divisor that supports $\M_X$

However, thanks to \cite{shuji} there exists a functor\[
\Log:\RSC_{\Nis}(k)\to \Shv_{\dNis}^{\rm ltr}(k,\Z)
\] 
where for $X=(\underline{X},\partial X)\in \SmlSm(k)$ we have\[
\Log(F)(X):= F^{\log}(X) = \omega^{\CI}F (\ul{X},|\partial X|_{red}).
\]
Here, $\omega^{\bf CI}\colon \RSC_{\rm Nis}\to \CI_{\Nis}$ is the functor defined in \cite[Prop. 2.3.7]{KSY2} (see also \cite[Thm. 2.4.1-2.4.2]{KSY2} and \ref{para:riassuntoomegaCIetc} below), and $\CI_{\Nis}$ is the subcategory of $\bcube$-invariant Nisnevich sheaves on $\MCor$, defined in \cite{KMSY1} (not to be confused with $\CI_{\dNis}^{\rm ltr}$ introduced in the present paper). By \cite[Thm. 0.2]{shuji}, $\Log$ is fully faithful and exact.

\begin{prop}
The essential image of $\Log$ is a subcategory of $\CIltr_{\dNis}$.
\begin{proof}
	By \cite[Theorem 4.1]{shuji} we have that for $F\in \RSC_{\Nis}$ then $\Log(F)$ is strictly $\bcube$-invariant.
	\end{proof}
\end{prop}

One can wonder if the two categories agree, i.e. if $\Log$ is essentially surjective onto $\CIltr_{\dNis}$. This is not the case, as the following example indicates. %On the other hand, the category $\CIltr_{\dNis}$ is bigger than $\RSC_{\Nis}$.
\begin{example}{See \cite[Proposition 3.5]{P1loc}}
Let $\G_a\in \RSC_{\Nis}$, then\[
\Log(\G_a)(X)=\Gamma(\underline{X},\mathcal{O}_{\underline{X}})
\]
By e.g. \cite[Corolary 9.2.6]{bpo}, we have that $H^n_{\dNis}(X,\Log(\G_a))=H^n_{sZar}(X,\Log(\G_a))$. Let $\Log(\G_a)\to I^{\bullet}$ be an injective resolution of $\dNis$-sheaves. Thus, for all $U\subseteq X$ open affine, then there is a quasi-isomorphism\[
\Log(\G_a)(U)\to I^{\bullet}(U)
\]
It follows that for every set $A$, the map\[
\prod_{A}\Log(\G_a)(U)\to \prod_{A}I^{\bullet}(U)
\]
is a quasi-isomorphism. Thus $\prod_{A}\Log(\G_a)\to \prod_{A}I^{\bullet}$ is a $\sZar$-local equivalence, hence a $\sNis$-local equivalence, so $\prod_{A}I^{\bullet}$ is an injective resolution of $\prod_{A}\Log(\G_a)$.

We conclude that
\begin{multline*}
H^n_{\dNis}(X,\prod_A\Log(\G_a))=H^n(\prod_{A}I^{\bullet}(X))=\prod_AH^nI^\bullet(X)=\prod_A H^n_{\dNis}(X,\Log(\G_a))
\end{multline*}
In particular, $\prod_A\Log(\G_a)$ is strictly $\bcube$ invariant.
On the other hand, by \cite[Remark 6.1.2]{KSY1}, if $A$ is infinite $\prod_A\G_a$ does not belong to $\RSC_{\Nis}$.
\end{example}

\begin{para}\label{para:riassuntoomegaCIetc}We recall some further constructions from the theory of modulus (pre)sheaves with transfers. For $F\in \MPST$, write $h_0^{\bcube}(F)$ for the presheaf
    \[\sX\mapsto \Coker(F(\sX\otimes \bcube ) \xrightarrow{i_0^*-i_1^*}F(\sX)),\]
    where $i_0^*$ and $i_1^*$ are as usual the pullbacks along the zero section and the unit section of $\bcube$ respectively. Clearly, $h_0^{\bcube}(F)$ is $\bcube$-invariant in $\MPST$, i.e.\ $h_0^{\bcube}(F)\in \CI$. By \cite[Prop. 2.1.5]{KSY2},  $h_0^{\bcube}(-)$ is the left adjoint to the inclusion $\iota^{\bcube}\colon \CI\to \MPST$. Note that $\iota^{\bcube}$ has a right adjoint as well by \cite[2.1.7]{KSY2}, denoted $h^0_{\bcube}(-)$.
    
    Let $\omega_!\colon \MPST\to \PSh^{\tr}(k)$ be the left Kan extension of $\omega\colon \MCor\to \Cor(k)$, sending $\sX = (X, X_\infty) \mapsto X-|X_\infty|$.\footnote{we follow the notation in \cite{KMSY1}, to avoid confusion with $\omega_\sharp$ used before, but note that the functor $\omega$ in \emph{loc.cit.} and the functor $\omega$ used in this paper are very similar, even though they are defined on different categories.} We write $\omega^{\CI} \colon \RSC\to \CI$ for the composition $h^0_{\bcube}\circ \omega^*\circ \iota_{\RSC}$, where $\iota_{\RSC}$ is the inclusion of $\RSC$ in $\PSh^{\tr}(k)$ and $\omega^*\colon \PSh^{\tr}(k)\to\MPST$ is the restriction. If no confusion arises, we will use the same symbols to denote the corresponding functor on the subcategories of Nisnevich sheaves, $\omega^{\CI}\colon \RSC_{\Nis}\to \CI_{\Nis}$ and $\omega^*\colon \Shv_{\Nis}^{\tr}(k)\to \MNST$.
    By \cite[Prop. 2.3.7]{KSY2}, $\omega_! \omega^{\CI} F=F$ for every $F\in \RSC_{\Nis}$.

Using the above-defined functors, we can compute the sections of $\Log(F)$ on $X\in \SmlSm(k)$ for $F\in \RSC_{\Nis}$ as follows. Write $X=(\ul{X}, \partial X)$ and $X^o=\ul{X}-|\partial X|$. Choose a normal compactification $j\colon \ul{X}\hookrightarrow Y$ with the property that $X^o\to \ul{X}\to Y$ is open dense, and such that the complement $Y-X^o = D+ \partial X_Y$ for some effective Cartier divisors $D$ and $\partial X_Y$ on $Y$ satisfying $Y-|D| = j(\ul{X})$ and $\partial X_Y\cap \ul{X} = \partial X$ as reduced Cartier divisors. Such a compactification is called a Cartier compactification of $\ul{X}$, and it always exists (cfr. \cite[Def.1.7.3]{KMSY1}). Then we have
\begin{equation}
    \Log(F)(X) = (\omegaCI F)^{\rm log}(X) = \varcolim_n \Hom_{\MNST}(h_0^{\bcube}(\ul{X}, nD + \partial X_Y), \omega^* F)
\end{equation}
where $\omega^*F\in \MNST$ if $F\in \Shv^{\tr}_{\Nis}$. This follows from \cite[Lem. 1.7.4(b)]{KMSY1} and the definition of $\omega^{\CI}$.
\end{para} 

\begin{prop}
\label{colimitsRSC}
$\RSC_{\Nis}$ is closed under colimits in $\Shv^{\tr}_{\Nis}(k)$.
\end{prop}
\begin{proof} Recall that if $\{F_i\}_{i\in I}$ is a diagram in $\Shv^{\tr}_{\Nis}(k)$, then
    \begin{equation}\label{eq:colimipshshv}\varcolim_{i\in I} F_i =a_{\Nis}^V\varcolim_{i\in I}\iota_{\Shv^{\tr}_{\Nis}}(F_i),\end{equation}
        where $\iota_{\Shv^{\tr}_{\Nis}}\colon \Shv^{\tr}_{\Nis}(k) \to \PSh^{\tr}(k)$ is the inclusion, the colimit on the left-hand side of \eqref{eq:colimipshshv} is computed in $\Shv^{\tr}_{\Nis}$ and the colimit on the right-hand side is computed in $\PSh^{\tr}(k)$.      Since  $a_{\Nis}^V$ respects reciprocity by \cite[Theorem 0.1]{shuji}, it is enough to prove that $\RSC$ is closed under colimits in $\PSh^{\tr}(k)$. Consider then a diagram $\{F_i\}_{i\in I}$ in $\RSC$. Since  $\omega_!$ is a left adjoint and thus it preserves all colimits, we have 
\[
\varcolim_{i\in I}^{\PSh^{\tr}_{\Nis}(k)} F_i= \omega_!\varcolim_{i\in I}^{\MPST} \omega^{\CI}F_i.
\]
Since  $\CI$ is closed under colimits, and  $h_0^{\bcube}$ and $i^{\bcube}$ are left adjoints, we conclude  $i^{\bcube}h_0^\bcube\varcolim^{\MPST}F_i=\varcolim^{\MPST}i^{\bcube}h_0^\bcube F_i$, so that the colimit is in $\RSC$, as required. 
\end{proof}

\begin{remark}\label{rmk;productofmodulus}
	 For $X,Y\in\SmlSm(k)$, we have by e.g. \cite[III.2]{ogu}\[
	X\times Y = (\underline{X}\times \underline{Y},(pr_X^*\mathcal{M}_X\oplus pr_Y^*\mathcal{M}_Y)^{\textrm{fs}})
	\] 
	The divisor that supports the sheaf of monoids $pr_X^*\mathcal{M}_X\oplus pr_Y^*\mathcal{M}_Y$ is $D_X\times Y + X\times D_Y$, where the divisors $D_X$ and $D_Y$ support $\mathcal{M}_X$ and $\mathcal{M}_Y$ respectively, and the functor $(-)^{\rm fs}$ does not change the support. We conclude that the associated reduced modulus pair of $X\times Y$ is $\sX\otimes \sY$.
\end{remark}

\begin{lemma}\label{lem:defRSC}
$\Log$ has a pro-left adjoint $\Rsc$, given by the formula
 \[\Rsc(G):=\varcolim_{X\downarrow G}^{\pro\RSC_{\Nis}}\varplim{n}\omega_!h_0^\bcube(\overline{X},D+nD'),F),\]
 for every $G \in  \Shltr(k,\Z)$.
\begin{proof}
It follows directly from Saito's theorem \cite[Thm. 6.3]{shuji} that $\Log$ preserves finite limits, so the existence of a pro-left adjoint is formal (see e.g. \cite[I.8.11.4]{SGA4I}). In the rest of the proof we characterize the pro-adjoint explicitly: such description will be used later in the computation. Let $F\in \RSC_{\Nis}$ and $G\in \Shltr(k,\Z)$. For any $X\in \SmlSm(k)$,  let $(\underline{X},D)$ be the associated reduced modulus pair and choose $\overline{X}$ a Cartier compactification of $\underline{X}$. Set $D':=X'\setminus X$. 

Recall that\[
\Log(F)(X)=\omegaCI(F)(X,D)=\colim_n\Hom(\omega_!h_0^\bcube(\overline{X},D+nD'),F).
\]
Writing $G$ as colimit of representable sheaves, $G=\varcolim_{s:X\to G} a_{\dNis}\Z_{\mathrm{ltr}}(X)$, we have:
\begin{align*}
    \Hom_{\Shltr}(G,\Log(F))& = \Hom_{\Shltr}(\varcolim_{X\downarrow G}  a_{\dNis}\Z_{\mathrm{ltr}}(X),\Log(F))\\
     & = \varlim_{X\downarrow G} \Hom_{\Shltr}(  a_{\dNis}\Z_{\mathrm{ltr}}(X),\Log(F))\\
     & =\varlim_{X\downarrow G}\colim_n \Hom_{\RSC}(\omega_!h_0^\bcube(\overline{X},D+nD'),F)\\
     & =\varlim_{X\downarrow G}\Hom_{\pro\RSC}(\varplim{n}\omega_!h_0^\bcube(\overline{X},D+nD'),F).
\end{align*}
where the last equality simply follows from the definition of the morphisms in the pro-category $\pro\RSC$.
By Proposition \ref{colimitsRSC}, $\RSC_{\Nis}$ is cocomplete, hence $\pro\RSC_{\Nis}$ is cocomplete by e.g.\ \cite[Prop. 11.1]{isaksen} and we can pass the limit inside the Hom to get: \[
\Hom_{\Shltr}(G,\Log(F))=\Hom_{\pro\RSC}(\varcolim_{X\downarrow G}^{\pro\RSC_{\Nis}}\varplim{n}\omega_!h_0^\bcube(\overline{X},D+nD'),F).
\]
Thus, we can identify the pro-left adjoint to $\Log$ with the functor \[\Rsc(G):=\varcolim_{X\downarrow G}^{\pro\RSC_{\Nis}}\varplim{n}\omega_!h_0^\bcube(\overline{X},D+nD'),\]
from $ \Shltr(k,\Z)$ to $\RSC_{\Nis}$.
\end{proof}
\end{lemma}
\begin{para}The category of reciprocity sheaves is equipped with a \emph{lax} monoidal structure constructed in \cite{RSY}, given by
    \begin{equation}\label{eq:deftensorRSC}
\bigl(F, G\bigr)_{\RSC_{\Nis}}:=\ulomega_!(\omegaCI F\otCINis \omegaCI G),
\end{equation}
for $F, G\in \RSC_{\Nis}$. More generally, there are functors for $n\geq 1$
\[\RSC_{\Nis}^{\times n} \to \RSC_{\Nis}, (F_1, \ldots, F_n)\mapsto \bigl(F_1, F_2, \ldots, F_n)_{\RSC_{\Nis}}, \]
which satisfies only a weak form of associativity, see \cite[Cor. 4.18-4.21]{RSY}. See \cite{RSY} and \cite{ms} for some computations. In particular, a nontrivial argument (see \cite[Theorem 5.2]{RSY}) shows that that:\[
(F,G)_{\RSC_{\Nis}} = F\otimes_{\HI_{\Nis}} G
\]
whenever $F,G\in \HI_{\Nis}$ and $\ch(k)=0$. We can extend the bifunctor $(-,-)_{\RSC_{\Nis}}$ to the pro-category as follows. 
\end{para}
\begin{definition}\label{def:tensorproRSC}
Let $F=\plim F_i,G=\plim G_j\in \pro\RSC_{\Nis}$, then we define\[
(F,G)_{\RSC}^{pro}:=\varplim{i,j} (F_i,G_j)_{\RSC_{\Nis}}.
\]\end{definition}
\begin{prop}
$(\_,\_)_{\RSC}^{pro}$ is well defined and bifunctorial.

\begin{proof}
We first show that the assignment is well defined, i.e. that it doesn't depend on the chosen representation of $\plim F$ as object in $\pro\RSC_\Nis$. Thus, let $ \plim F\cong \plim F'$ be another representation of the pro-system $F$. For every $\plim G,\plim H \in \pro\RSC_{\Nis}$,  we have canonical identifications
\begin{align}\label{eq:well-defRSCpro-monoidal1}
\Hom_{\pro\RSC_{\Nis}}&((\plim F,\plim G)_{\RSC}^{pro},\plim H)\\\nonumber
 &=^{(1)}\lim_{H} \colim_{F,G}\Hom_{\RSC_{\Nis}}((F,G)_{\RSC}, H)\\\nonumber
&=^{(2)}\lim_{H} \colim_{F,G}\Hom_{\RSC_{\Nis}}(\omega_!(\omegaCI F\otimes_{\CI}^{\Nis} \omegaCI G), \omega_!\omegaCI H)\\\nonumber
&=^{(3)}\lim_{H} \colim_{F,G}\Hom_{\CIt_{\Nis}}(\omegaCI F\otimes_{\CI}^{\Nis} \omegaCI G, \omegaCI H),\\\nonumber
&=^{(4)}\lim_{H} \colim_{F,G}\Hom_{\CIt_{\Nis}}(\omegaCI F, \uHom(\omegaCI G, \omegaCI H)),\nonumber
\end{align}
    where (1) is given by the definition of the morphisms in the pro category, (2) is simply the definition of the monoidal structure, (3) follow from the fact that $\omega_!$ restricts to a functor $\CI_{\Nis} \to \RSC_{\Nis}$ which is left adjoint to the fully faithful functor $\omegaCI$, and finally (4) is the adjunction for the internal Hom structure in $\CI_{\Nis}$. %Next, observe that  by adjunction we have that for all $F,G,H \in \RSC_{\Nis}$ there is an isomorphism\[
%\Hom_{\CIt_{\Nis}}(\omegaCI F\otimes_{\CI}^{\Nis} \omegaCI G, \omegaCI H) = \Hom_{\CIt_{\Nis}}(\omegaCI F, \uHom(\omegaCI G, \omegaCI H))\]
The functor $\omegaCI$ preserves all limits being a right adjoint, hence it induces a functor on the pro categories defined level-wise\[
\pro\omegaCI\colon \pro\RSC_{\Nis}\to \pro \CIt_{\Nis}\quad \pro\omegaCI(\plim F) := \plim \omegaCI F
\]
Hence, since $\plim F\cong \plim F'$, we have that \[
 \pro\omegaCI(\plim F)\cong  \pro\omegaCI(\plim F'),
\]
in particular, for fixed $G$ and $H$ in $\RSC_{\Nis}$, we have isomorphisms  
\begin{align}\label{eq:well-defRSCpro-monoidal2}
\colim_F&\Hom_{\CIt_{\Nis}}(\omegaCI F, \uHom(\omegaCI G, \omegaCI H)) \\ \nonumber & = \Hom_{\pro\CIt_{\Nis}}(\pro\omegaCI \plim F, \uHom(\omegaCI G, \omegaCI H))\\\nonumber
& \simeq \Hom_{\pro\CIt_{\Nis}}(\pro\omegaCI \plim F', \uHom(\omegaCI G, \omegaCI H)) \\\nonumber
& \simeq \colim_{F'}\Hom_{\CIt_{\Nis}}(\omegaCI F', \uHom(\omegaCI G, \omegaCI H)).\nonumber
\end{align}
Combining \eqref{eq:well-defRSCpro-monoidal1} and \eqref{eq:well-defRSCpro-monoidal2} we have that
\begin{align*}
\lim_{H} \colim_{F,G}&\Hom_{\CIt_{\Nis}}(\omegaCI F,\uHom(\omegaCI G,\omegaCI H))\\
&=\lim_{H} \colim_{F',G}\Hom_{\CIt_\Nis}(\omegaCI F',\uHom(\omegaCI G,\omegaCI H))\\
&=\lim_{H} \colim_{F',G}\Hom_{\RSC_{\Nis}}((F',G)_{\RSC}, H)\\
 & = \Hom_{\pro \RSC_{\Nis}}((\plim F',\plim G)_{\RSC}^{pro}, \plim H)
\end{align*}
This shows that $(\_,\_)_{\RSC}^{pro}$ is indeed well defined.

We now prove the functoriality statement. Let $f\colon \plim{F}\to \plim{G}$ be a morphism in $\pro\RSC_{\Nis}$. We can use  e.g.\ \cite[Appendix 3.2]{am} to reindex the limit by choosing isomorphisms $a\colon \plim{F}\xrightarrow{\simeq} \plim_{\alpha} F_{\alpha}$ and $b\colon \plim{G}\xrightarrow{\simeq} \plim_{\alpha}  G_{\alpha}$ and level-wise defined morphisms $f_{\alpha}:F_{\alpha}\to G_\alpha$ in $\RSC_{\Nis}$ such that $f=b^{-1}\plim f_{\alpha} a$. 
Let $H=\plim H_{\beta}$ be another pro-reciprocity sheaf. Then for all $\alpha, \beta$ we have a map \[(f_\alpha,id)_{\RSC_{\Nis}} \colon (F_{\alpha } H_{\beta})_{\RSC_{\Nis}}\to (G_{\alpha},H_{\beta})_{\RSC_{\Nis}}\]
The previous computations  show that both $a$ and $b$ induce isomorphisms \begin{align*}
(a,id)& \colon (F,H)_{\RSC}^{pro} \to (\plim F_\alpha,H)_{\RSC}^{pro}\\ (b,id)& \colon  (G,H)_{\RSC}^{pro} \to (\plim G_\beta,H)_{\RSC}^{pro},
\end{align*}
which then induce a morphism:
\[
(F,H)_{\RSC}^{pro} \xrightarrow[\simeq]{(a,id)} (\plim F_{\alpha},H)_{\RSC}^{pro} \xrightarrow{\plim (f_\alpha,id) }(G_{\alpha},H)_{\RSC}^{pro}\xrightarrow[\simeq]{(b,id)^{-1}} (G,H)_{\RSC}^{pro}.
\]
It is clear that this morphism  depends only on $f$, since if $a':F\xrightarrow{\sim} \plim F_{\beta}$ and $b':G\xrightarrow{\sim} \plim G_{\beta}$, then the diagram below commutes:\[
\begin{tikzcd}
(\plim F_{\alpha},H)_{\RSC}^{pro}\ar[rr,"{\plim (f_\alpha,id)}"]&&(\plim G_{\alpha},H)_{\RSC}^{pro}\ar[d,"\simeq"]\\
(F,H)_{\RSC}^{pro}\ar[u,"\simeq"]\ar[d,swap,"\simeq"]&&(G,H)_{\RSC}^{pro}\\
(\plim F_{\beta},H)_{\RSC}^{pro}\ar[rr,"{\plim (f_\beta,id)}"]&&(\plim G_{\beta},H)_{\RSC}^{pro}\ar[u,swap,"\simeq"]
\end{tikzcd}
\]
The composition and the identity are clearly respected, and the same computation gives functoriality for the other component. 
\end{proof}
\end{prop}
\begin{remark}If $\mathcal{C}$ is a category equipped with a monoidal structure $\otimes$ (in particular, associative), then the category $\pro\mathcal{C}$ is equipped with the level-wise monoidal structure $\{X_\alpha\}\tensor \{Y_\beta\} = \{X_\alpha\tensor Y_\beta\}$. See \cite[11]{IsaksenFausk}. Since the construction \eqref{eq:deftensorRSC} gives a monoidal structure on $\RSC_{\Nis}$ only in a weak sense (in particular, associativity  is not known to hold), we need to verify explicitly that the level-wise assignment \ref{def:tensorproRSC} is indeed well defined. Note that the argument is \emph{ad hoc}, and only proves the existence of a bi-functor at the level of pro-categories.
\end{remark}

The functoriality statement of the previous Proposition implies in particular that if $(F_i)_{i\in I}$ and $(G_j)_{j\in J}$ are diagrams in $\pro\RSC_{\Nis}$, then there is a natural map
\begin{equation}
\label{eq;colimits}
\varcolim^{\pro\RSC_{\Nis}}_{i,j}(F_i,G_j)_{\RSC}^{pro}\to (\varcolim^{\pro\RSC_{\Nis}}_i F_i,\varcolim^{\pro\RSC_{\Nis}}_j G_j)_{\RSC}^{pro}. 
\end{equation}
In general, there is no reason to expect that \eqref{eq;colimits} is an isomorphism (see also \cite[Ex. 11.2]{IsaksenFausk} for a similar problem).
Using the explicit description of the pro-left adjoint to $\Log$, we get then the following result.
\begin{thm}\label{thm:natmapRSCtensor}
For $F,G\in \CIltr_{\dNis}$, there exists a natural map \[\Rsc(F\otimes^{ltr} G)\to (\Rsc(F),\Rsc(G))_{\RSC}^{pro}.\]
\begin{proof} The tensor product in $\Shv_{\dNis}^{\rm ltr}(k)$ is given by Day convolution from the monoidal structure on $\SmlSm(k)$. So, if $F=\varcolim_{X\downarrow F} a_{\dNis}\Z_{tr}(X)$ and $G=\varcolim_{Y\downarrow G} a_{\dNis}\Z_{tr}(Y)$. Then \[F\otimes^{ltr} G=\varcolim_{X,Y}^{\Shltr}a_{\dNis}\Z_{tr}(X\times Y).\]
A Cartier compactification of $X\times Y$ is given by $\overline{X}\times \overline{Y}$, where $\overline{X}$ and $\overline{Y}$ are Cartier compactifications of $X$ and $Y$. Let $D_X'=\overline{X}-X$ and $D_Y'=\overline{Y}-Y$. Using the explicit description of the functor $\Rsc$ given in the proof of Lemma \ref{lem:defRSC}, we get
\begin{align*}
\Rsc&(F\otimes^{ltr} G)\\
&= \varcolim^{\pro\RSC_{\Nis}}\plim \omega_!h_0^\bcube(\overline{X}\times \overline{Y},D_X\times Y + X\times D_Y+n(D'_X\times Y + X\times D'_Y))\\
&= \varcolim^{\pro\RSC_{\Nis}}\plim \omega_!(h_0^\bcube(\overline{X},D_X+nD_X')\otimes_{\CI}h_0^{\bcube}(\overline{Y},D_Y+nD_Y')).
\end{align*}
Consider now the natural maps \[
h_0^\bcube(\overline{X},D_X+nD_X')\to \omegaCI\omega_!(h_0^\bcube(\overline{X},D_X+nD_X'),\] they give a natural map
\begin{align*}
    \plim &\omega_!(h_0^\bcube(\overline{X},D_X+nD_X')\otimes_{\CI}h_0^{\bcube}(\overline{Y},D_Y+nD_Y')) \\
    &\longrightarrow \plim \omega_!(\omegaCI\omega_!h_0^\bcube(\overline{X},D_X+nD_X')\otimes_{\CI}\omegaCI\omega_! h_0^{\bcube}(\overline{Y},D_Y+nD_Y'))\\
    & = \plim (\omega_!h_0^\bcube(\overline{X},D_X+nD_X'),\omega_! h_0^{\bcube}(\overline{Y},D_Y+nD_Y'))_{\RSC}
\end{align*}
By definition the last term is equal to \[
(\plim h_0^\bcube(\overline{X},D_X+nD_X'),\plim h_0^{\bcube}(\overline{Y},D_Y+nD_Y'))_{\RSC}^{pro}
\]
Hence we obtained a natural map
\begin{multline}
    \Rsc(F\otimes^{ltr} G)\to\\
    \varcolim^{\pro\RSC_{\Nis}}(\plim h_0^\bcube(\overline{X},D_X+nD_X'),\plim h_0^{\bcube}(\overline{Y},D_Y+nD_Y'))_{\RSC}^{pro}
\end{multline}
Finally, as observed in \eqref{eq;colimits} there is a natural map
\begin{multline*}
\varcolim^{\pro\RSC_{\Nis}}(\plim h_0^\bcube(\overline{X},D_X+nD_X'),\plim h_0^{\bcube}(\overline{Y},D_Y+nD_Y'))_{\RSC}^{pro}
\to \\
(\varcolim^{\pro\RSC_{\Nis}}\plim h_0^\bcube(\overline{X},D_X+nD_X'),\varcolim^{\pro\RSC_{\Nis}}\plim h_0^{\bcube}(\overline{Y},D_Y+nD_Y'))_{\RSC}^{pro}
\end{multline*}
and the last term is equal to $(\Rsc(F),\Rsc(G))_{\RSC}^{pro}$.
\end{proof}
\end{thm}

\begin{corollary}\label{cor;monoidalfunctor}

Let $F,G\in \RSC$, then there exists a natural map
\begin{equation}
\label{equation2}
\Log(F)\otimes_{\CIlog_{\dNis}}\Log(G)\to \Log((F,G)_{\RSC_{\Nis}})
\end{equation}
\proof Let $(-)^p:\RSC_{\Nis}\to \pro\RSC_{\Nis}$ be the constant functor $F\mapsto \plim F$. Since $\Log$ is fully faithful, we have\[
F^p=\Rsc(\Log(F)),\quad G^p=\Rsc(\Log(G))
\]
By definition we have that\[
(F^p,G^p)_{\RSC}^{pro} = ((F,G)_{\RSC})^p
\]
so the previous lemma gives a natural map\[
\Rsc(\Log(F)\otimes^{ltr}\Log(G))\to ((F,G)_{\RSC})^p
\]
whose adjoint gives a map\[
\Log(F)\otimes^{ltr}\Log(G)\to \Log((F,G)_{\RSC}).
\]
Finally, since $\Log((F,G)_{\RSC})\in \CIltr_{\dNis}$, the previous map factors through the localization $h_0(\Log(F)\otimes^{ltr}\Log(G))=\Log(F)\otimes_{\CIltr_{\dNis}}\Log(G)$, giving the desired map.
\end{corollary}

\section{Log reciprocity sheaves}\label{sec:logrec}

In this final section, %we assume that our field satisfies resolution of singularities,  $(RS)$ for short (see e.g. \cite[Def. 7.6.3]{bpo} for a precise definition). W
we state a conjecture that will allow us to construct full subcategories $\LogRec$ and  $\LogRec^{\rm tr}$ respectively of $\Sh_{\Nis}(k,\Lambda)$ and $\Sh_{\Nis}^{\mathrm{tr}}(k,\Lambda)$ such that $\RSC_{\Nis}\subseteq \LogRec^{\tr}$.

Our definition generalizes the construction of \cite{KSY2} and it is very similar in spirit.

\begin{defn}\label{omegaCI}
We define two pairs of adjoint functors\[
\begin{tikzcd}
\omega_{\CI}^{\log}\colon \CIlog_{\dNis}\ar[r,shift right]&\Shv_{\Nis}(k,\Lambda)\colon\omega^{\CI}_{\log} \ar[l,shift right]&\omega_{\CI}^{\rm ltr}\colon \CIltr_{\dNis}\ar[r,shift right]&\Shv^{\mathrm{tr}}_{\Nis}(k,\Lambda)\colon\omega^{\CI}_{\rm ltr} \ar[l,shift right]
\end{tikzcd}
\]
where $\omega_{\CI}^{\log}:=\omega_{\sharp}i$ and $\omega^{\CI}_{\log}:=h^0\omega^*$, where $h^0$ is the right adjoint to the inclusion of Proposition \ref{prop;rightadjoint} (similarly for the case with transfers). The counit map $i_{\CIlog_{\dNis}}h^0\to id$ induces for all $F\in \Shv_{\dNis}(k,\Lambda)$ a canonical map
\begin{equation}\label{wanttobeinjective}
i\omega^{\CI}_{\log}F\to \omega^*F
\end{equation}
\end{defn}
\begin{prop}\label{lem:omegalowerCIffexact}
The compositions\[
\CIlog_{\dNis}\xrightarrow{i}\Shv^{\log}_{\dNis}\xrightarrow{\omega_{\sharp}^{\log}}\Shv_{\Nis}\qquad \CIltr_{\dNis}\xrightarrow{i^{\rm tr}}\Shv^{\rm ltr}_{\dNis}\xrightarrow{\omega_{\sharp}^{\rm ltr}}\Shv_{\Nis}^{\tr}
\]
are faithful and exact. In particular, both functors are conservative.
\begin{proof}
		Exactness follows from the exactness of $i^{\rm tr}$ and $\omega_{\sharp}^{\rm tr}$ (resp. $i^{\rm tr}$ and $\omega_{\sharp}^{\rm tr}$). To show faithfulness, it is enough to show that for all $F \in \CIlog_{\dNis}$ (resp. $\CIltr_{\dNis}$), the unit map\[
F\to \omega^{\CI}_{\log }\omega_{\sharp}^{\log} F\quad ({\rm resp. } 	F\to \omega^{\CI}_{\rm ltr}\omega_{\sharp}^{\rm ltr} F)
\]
is injective. By \cite[Theorem 5.10]{purity} we have that for all $X \in  \SmlSm(k)$, \[
F(X)\hookrightarrow  F(\ul{X}-|\partial X|) = \omega^{*}_{\log}\omega_{\sharp}^{\log} F\quad ({\rm resp.} \omega^{*}_{\rm ltr}\omega_{\sharp}^{\rm ltr} F)
\] 
hence $u\colon F\hookrightarrow \omega^{*}_{\log}\omega_{\sharp}^{\log} F$ (resp. $u^{\tr}\colon F\hookrightarrow \omega^{*}_{\log}\omega_{\sharp}^{\log} F$) is injective. Because $F$ is $\bcube$-local, the map $u$ (resp. $u^{\tr}$) factors through $\omega^{\CI}_{\log }\omega_{\sharp}^{\log} F$ (resp. $\omega^{\CI}_{\rm ltr}\omega_{\sharp}^{\rm ltr} F$), which concludes the proof.
\end{proof}
\end{prop}

We now state a conjecture, that we hope to prove soon:
\begin{conj}\label{conj:full}
The functors $\omega_{\CI}^{\log}$ and $\omega_{\CI}^{\rm ltr}$ are full.
\end{conj}	

%For the rest of the section, we assume Conjecture \ref{conj:full}.

\begin{defn} Assume Conjecture \ref{conj:full}
	Let $\LogRec$ (resp. $\LogRec^{\rm tr}$) denote the essential image of $\omega_{\CI}^{\log}$ (resp. $\omega_{\CI}^{\rm ltr}$), i.e. categories of sheaves $F\in \Shv_{\Nis}$ (resp. $\Shv^{\tr}_{\Nis}$) such that there exists $G\in \CIlog_{\dNis}$ (resp. $\CIltr_{\dNis}$) such that $F=\omega_{\CI}G$.
	
	By definition, $\omega_{\CI}^{\log}$ (resp. $\omega_{\CI}^{\rm ltr}$) induces an equivalence between $\CIlog_{\dNis}$ (resp. $\CIltr_{\dNis}$) and $\LogRec$ (resp. $\LogRec^{\tr}$) with quasi-inverse the restriction of $\omega^{\CI}$.
	\end{defn}
\begin{remark}\label{rem:LogRecebello}
	Assume Conjecture \ref{conj:full}. Let $F\in \LogRec$ and let $G\in \CIltr_{\dNis}$ such that $F=\omega_\sharp G$. We deduce some immediate properties: 
	\begin{enumerate}
		\item For all $X\in \Sm$ and $U\subseteq X$ dense open, Theorem \ref{nonderivedpurity} implies that $F(X)\hookrightarrow F(U)$ is injective.
		\item For all $n$ and all $X\in \Sm$, we have that 
		\begin{multline*}
		a_{\Nis}\mathbf{H}^n_{\Nis}(\_\times X,F)=a_{\Nis}H_n(\uHom(X,F_{\Nis}))\overset{(*1)}{=}a_{\Nis}H_n(\uHom(X,\omega_\sharp G_{\dNis})) \overset{(*2)}{=}\\
		 a_{\Nis}H_n(\omega_\sharp\uHom((X,\triv), G_{\dNis}))\overset{(*3)}{=}\omega_\sharp a_{\dNis}H_n(\uHom((X,\triv), G_{\dNis}))
		\end{multline*}
		where $(*1)$  comes from Proposition \ref{prop;omeganis}, $(*2)$ comes by definition of $\omega_\sharp$ and $(*3)$ from the fact that $\omega_\sharp$ is $t$-exact and \cite[(4.3.4)]{bpo}.
		By Corollary \ref{cor;ext}, $a_{\dNis}H_n(\uHom((X,\triv), G_{\dNis}))\in \CIltr_{\dNis}$, so the cohomology sheaf $a_{\Nis}\mathbf{H}^n_{\Nis}(\_\times X,F)\in \LogRec$.
	\end{enumerate}
\end{remark}

\begin{thm}\label{prop:RSCNissubcatLogRec}
		Assume Conjecture \ref{conj:full}. The category $\RSC_{\Nis}$ is a full subcategory of $\LogRec$. In particular,
	\begin{equation}\label{logvsomegaCI}
	\Log = \omega^{\CI}_{\log}i_{\RSC}
	\end{equation}
	\begin{proof}
	Since $\RSC_{\Nis}$ is a full subcategory of $\Sh_{\Nis}(k,\Lambda)$, it is enough to show that for every $F\in \RSC_{\Nis}$ there exists $G\in \CIlog_{\dNis}$ such that $F=\omega_\sharp G$. 
	
	By \cite[Section 4]{shuji} we have that
	\begin{equation}\label{logvsomegaCI2}
	\omega_\sharp \Log(F)(X) = \omega^{\CI}F(X,\emptyset) = F(X)
	\end{equation}
	Hence $\RSC_{\Nis}$ is a full subcategory of $\LogRec$.
	
	Finally, since $\omega_\sharp$ is an equivalence, \eqref{logvsomegaCI} follows directly from \eqref{logvsomegaCI2}.
	\end{proof}
	\end{thm}
\begin{cor}\label{cor:globalinjcohoRSC}	Assume Conjecture \ref{conj:full}. Let $F\in \RSC_{\Nis}$ and let $X\in \Sm(k)$. Then the cohomology  of $F$ satisfies
        \[  \mathbf{H}^n(X\times Y, F) \hookrightarrow   \mathbf{H}^n(X\times \eta_Y,F) \]
    for every $n\geq0$ and $Y$ henselian  local essentially smooth $k$-scheme with generic point $\eta_Y$.
\end{cor}
\begin{proof}It follows immediately from Theorem \ref{prop:RSCNissubcatLogRec} and Remark \ref{rem:LogRecebello}.
    \end{proof}
Let $i_{\RSC}$ (resp. $i_{\RSC}^{\log}$) denote the inclusion of $\RSC_{\Nis}$ in $\Shv^{\tr}_{\Nis}(k)$ (resp. in $\LogRec$).
Recall by \cite{KSY2} that $i_{\RSC}$ has a pro-left adjoint $\rho$ such that for $X\in \Sm(k)$ and $\overline{X}$ a Cartier compactification with $D=\overline{X}-X$, then\[
\rho(\Z_{\tr}(X))=\plim \omega_!h_0^{\bcube}(\overline{X},nD).
\]

\begin{prop} 	Assume Conjecture \ref{conj:full}. 
The functor $i_{\RSC}^{\log}$ has a pro-left adjoint $\rho_{\log}$, which factors $\rho$. In particular,\[
	\Rsc = \rho_{\log}\omega_{\CI}^{\log}
	\]	
\begin{proof}

Since $i_{\RSC}=i_{\LogRec}i_{\RSC}^{\log}$ and $i_{\LogRec}$ is fully faithful, for $F\in \Shv^{\tr}_{\Nis}$ $G\in \RSC_{\Nis}$ we have that\begin{multline*}
	\Hom_{\pro\RSC}(\rho i_{\LogRec}F,G) = \Hom_{\Shv^{\tr}_{\Nis}}(i_{\LogRec}F,i_{\LogRec}i_{\RSC}^{\log}G) =\\ \Hom_{\Shv^{\tr}_{\Nis}}(i_{\LogRec}F,i_{\LogRec}i_{\RSC}^{\log}G) = \Hom_{\LogRec}(F,i_{\RSC}^{\log}G).
\end{multline*}
	Finally, for $F\in \CIltr_{\dNis}$ and $G\in \RSC_{\Nis}$, we have that\begin{align*}
		\Hom_{\pro\RSC}(\Rsc(F),G) &= \Hom_{\CIltr_{\dNis}}(F,\Log(G))\\
		&=\Hom_{\CIltr_{\dNis}}(F,\omega^{\CI}_{\log}i_{\RSC}^{\log}G)\\
		&=\Hom_{\Shv^{\tr}_{\Nis}}(i_{\LogRec}\omega_{\CI}^{\log}F,i_{\RSC}G)\\
		&=\Hom_{\pro\RSC}(\rho_{\log}\omega_{\CI}F,G)
	\end{align*}
	\end{proof}
	\end{prop}
	\begin{remark} 	Assume Conjecture \ref{conj:full}.
	Since $\CIltr_{\dNis}$ is a symmetric monoidal Grothendieck abelian category, then $\LogRec$ is symmetric monoidal with tensor product given by\[
	F\otimes_{\LogRec} G:=\omega_{\sharp}(h_0(\omega^{\CI}_{\log}F\otimes^{\mathrm{ltr}} \omega^{\CI}_{\log} G)).
	\] 
	By \ref{cor;monoidalfunctor}, for all $F,G\in \RSC_\Nis$ we have a map\[
	F\otimes_{\LogRec} G\to (F,G)_{\RSC}
	\]
	\end{remark}
If $\ch(k)\neq 0$, this map is not an isomorphism (see below). We do not know whether we expect it to be an isomotphism when $\ch(k)=0$: this would prove that $(\_,\_)_{\RSC_{\Nis}}$ defines a monoidal structure on $\RSC_{\Nis}$.
\begin{para}Let $F,G\in \RSC_{\Nis}$ and let $F'\subseteq \omegaCI F$ such that $\omega_!F'=F$ (in the language of \cite{RS}, $F'$ corresponds to a semi-continuous conductor of $F$ different from the motivic conductor). By construction, there exists a canonical map
\begin{equation}\label{eq;otCIconductors}
\omega_!(F'\otCINis \omegaCI G)\to\omega_!(\omegaCI F\otCINis \omegaCI G) = (F,G)_{\RSC}
\end{equation}
This map is surjective: let $Q$ be the cokernel of the inclusion $F'\to \omega^{\CI}F$  such that $\omega_!Q=0$. Hence, since $\_\otimes_{\CI} \omegaCI G$ is right exact, there is a right exact sequence\[
F'\otCINis \omegaCI G\to \omega^{\CI}F\otCINis \omegaCI G \to Q\otCINis \omegaCI G\to 0
\]
and since $Q\otCINis \omegaCI G$ is a quotient of $Q\otimes_{\ulMNST} \omegaCI G$ and $\omega_!$ is exact and monoidal in $\ulMNST$, we conclude $\omega_!(Q\otimes_{\CI} \omegaCI G)=0$, which shows the surjectivity of \eqref{eq;otCIconductors}.
\end{para}
The kernel of \eqref{eq;otCIconductors} incapsulates the obstruction to the associativity of $(\_,\_)_{\RSC}$, and it seems to be very difficult to compute in general. We know that it is not trivial if $\ch(k)\neq 0$: see \cite[Theorem 4.17]{RSY} and \cite[Theorem 5.19]{RSY} for an explicit computation. 

On the other hand, we do not have any counterexamples if $\ch(k)=0$, hence we do not know whether to expect that the map above is an isomorphism. In this direction, we have the following result:
\begin{prop}\label{prop:tensorfactors} 	Assume Conjecture \ref{conj:full}.
Let $F,G\in \RSC_{\Nis}$. Then for all $F'\subseteq \omegaCI F$ (in $\ulMNST$) such that $\omega_! F' = F$, the canonical map\[
	F\otimes_{\LogRec}G \to (F,G)_{\RSC}
	\]
factors through $\omega_!(F'\otCINis \omegaCI G)$.
	
	\begin{proof}
		Let $(-)^{\log}$ be the functor of \cite{shuji} and recall that $\Log(F)=(\omegaCI F)^{\log}$. Since $\Log(F)=(F')^{\log}$ by construction, we can look at the  diagram\[
		\begin{tikzcd}
		\Log(F)\otimes_{\CIltr_{\dNis}} \Log(G)\ar[r] &(\omegaCI F\otCINissp\omegaCI G)^{\log}\\
		(F')^{\log}\otimes_{\CIltr_{\dNis}} \Log(G)\ar[u,equal] &(F'\otCINissp\omegaCI G)^{\log}\ar[u]
		\end{tikzcd}
		\] 
It is enough to show that there is a map \[
		(F')^{\log}\otimes_{\CIltr_{\dNis}} \Log(G)\to (F'\otCINissp\omegaCI G)^{\log}.
		\]
		that makes the diagram above commutative. By adjunction, it is enough to construct a map\[
		(F')^{\log}\to \uHom_{\Shltr}(\Log(G),(\omegaCI F\otCINissp \omegaCI G)^{\log}).
		\]
		that factors the map
		\begin{equation}\label{eq;easymap}
		(F')^{\log}=\Log(F)\to \uHom_{\Shltr}(\Log(G),(\omegaCI F\otCINissp \omegaCI G)^{\log}).
		\end{equation}
		Consider the following map given by the closed monoidal structure of $\CItsp_{\Nis}$ (see \cite[\S 3]{ms}):
		\begin{equation}\label{eq;firstpart}
		F'\to \uHom_{\CItsp_{\Nis}}(\omegaCI G,F'\otCINissp \omegaCI G).
		\end{equation}
		Let $X\in \SmlSm(k)$ and let $\sX\in \ulMCor$ be the corresponding reduced modulus pair. By construction, we have that
		\begin{multline}\label{eq;map1}
		(\uHom_{\CI}(\omegaCI G,F'\otCINissp \omegaCI G))^{\log}(X)=\Hom_{\ulMPST}(\omegaCI G\otimes \Z_{\tr}(\sX),F'\otCINissp \omegaCI G)=\\
		\Hom_{\ulMPST}(\omegaCI G,\uHom_{\CI}(h_0^{\bcube}(\sX),F'\otCINissp \omegaCI G)).
		\end{multline}
		Then the unit $id\to \omegaCI\omega_!$ induces the following map:
		\begin{multline}\label{eq;map2}
		\Hom_{\ulMPST}(\omegaCI G,\uHom_{\CI}(h_0^\bcube(\sX),F'\otCINissp \omegaCI G))\to \\
		\Hom_{\ulMPST}(\omegaCI G,\omegaCI\omega_!\uHom_{\CI}(h_0^{\bcube}(\sX),F'\otCINissp \omegaCI G))\overset{(*1)}{=}\\
		\Hom_{\RSC}(G,\omega_!\uHom_{\CI}(h_0^{\bcube}(\sX),F'\otCINissp \omegaCI G))\overset{(*2)}{=}\\
		\Hom_{\Shltr}(\Log(G),\Log(\omega_!\uHom_{\CI}(h_0^{\bcube}(\sX),F'\otCINissp \omegaCI G)))\overset{(*3)}{=}\\
		\Hom_{\Shltr}(\Log(G),(\uHom_{\CI}(h_0^{\bcube}(\sX),F'\otCINissp \omegaCI G))^{\log})
		\end{multline}
		where $(*1)$ (resp. $(*2)$, resp. $(*3)$) follows from the full faithfulness of $\omegaCI$ (resp. the full faithfulness of $\Log$, resp. the fact that $\Log(\omega_!) = (\_)^{\log}$, see \cite[Corollary 2.6 (3)]{shuji}).
		
		Finally, fix $Y\in \SmlSm(k)$ and let $\sY\in \ulMCor$ be the corresponding reduced modulus pair.
		
		We have that
		\begin{multline}\label{eq;map3}
		(\uHom_{\CI}(h_0^\bcube(\sX)F'\otCINissp \omegaCI G))^{\log}(Y)=\uHom_{\CI}(h_0^\bcube(\sX\otimes \sY),F'\otCINissp \omegaCI G)=\\
		(F'\otCINis \omegaCI G)(\sX\otimes \sY)\overset{(*)}{=}(F'\otCINissp \omegaCI G)^{\log}(X\times Y)=\\
		\uHom_{\Shltr}(\Z_{\mathrm{ltr}}(X),(F'\otCINissp \omegaCI G)^{\log})(Y)
		\end{multline}
		where $(*)$ is true by the observation in Remark \ref{rmk;productofmodulus}. We conclude that:
		\begin{multline}\label{eq;map4}
		\Hom_{\Shltr}(\Log(G),(\uHom_{\CI}(h_0^\bcube(\sX),F'\otCINissp \omegaCI G)^{\log})=\\
		\Hom_{\Shltr}(\Log(G),\uHom_{\Shltr}(\Z_{\mathrm{ltr}}(X),(F'\otCINissp \omegaCI G)^{\log}))=\\
		\uHom_{\Shltr}(\Log(G),(F'\otCINissp \omegaCI G)^{\log})(X)
		\end{multline}
		
		Putting \eqref{eq;map1}, \eqref{eq;map2}, \eqref{eq;map3} and \eqref{eq;map4} together, we have a map
		\begin{equation}\label{eq;secondpart}
		(\uHom_{\CI}(\omegaCI G,F'\otCINissp \omegaCI G))^{\log}\to \uHom_{\Shltr}(\Log(G),(F'\otCINissp \omegaCI G)^{\log})
		\end{equation} 
		
		Hence by applying $(\_)^{\log}$ to \eqref{eq;firstpart} and composing with \eqref{eq;secondpart} we get the map
		\begin{multline}\label{eq;desiredmap}
		(F')^{\log}\to (\uHom_{\CI}(\omegaCI G,F'\otCINissp \omegaCI G))^{\log}\to \uHom_{\Shltr}(\Log(G),(F'\otCINis \omegaCI G)^{\log})\\
		=\uHom_{\Shltr}(\Log(G),\Log\omega_!(F'\otCINissp \omegaCI G))
		\end{multline}
		
		Finally, notice that the map \eqref{eq;map2} factors the map
		\begin{multline*}
		\Hom_{\ulMPST}(\omegaCI G,\uHom_{\CI}(h_0^\bcube(\sX),F'\otCINissp \omegaCI G))\to\\ 
		\Hom_{\ulMPST}(\omegaCI G,\omegaCI\omega_!\uHom_{\CI}(h_0^{\bcube}(\sX),\omegaCI\omega_!F'\otCINissp \omegaCI G)).
		\end{multline*}
		So, since $\omega_!F'=F$ and $\omega_!(\omegaCI F\otCINissp \omegaCI G))=(F,G)_{\RSC}$, the equalities of \ref{eq;map2}, \ref{eq;map3} and \ref{eq;map4} with $\omegaCI F$ instead of $F'$ conclude that \eqref{eq;desiredmap} factor \eqref{eq;easymap}. This concludes the proof.
	\end{proof}
\end{prop}

\begin{remark}
For $F\in Shv_\Nis$, we denote by $h^0_{\A^1}(F)$ the biggest $\A^1$-local subsheaf as defined in \cite[4.34]{RS}: for $U\in \Sm$,\[
h^0_{\A^1}(F)(U):= \Hom(h_0^{\A^1}(U),F).
\] 
On the other hand, for $U\hookrightarrow \ul{X}$ a Cartier compactification such that $\ul{X}$ is proper and smooth over $k$ and $\ul{X}-U$ is a simple normal crossing divisor, then for $X=(\ul{X},\partial X)\in \SmlSm(k)$ such that $\partial X$ is supported on $\ul{X}-U$, by \cite[Proposition 8.2.4]{bpo} we have that\[
h_0^{\A^1}(X)=\omega_{\sharp} h_0(X)
\]
Hence if $F\in \LogRec$, then\[
h^0_{\A^1}(F)(U)=\Hom(\omega_{\sharp}h_0(X),F) = \omegaCI_{\log}F(X).
\]
Here we underline that this does not depend on $X$, as long as $\underline{X}$ is \emph{proper}.

We conclude with this observation: for $X$ as above and $\sX\in \MCor$ the associated reduced modulus pair, by \cite[Corollary 4.36]{RS} if $F\in \RSC_{\Nis}$, we have that\[
\Hom(\omega_!h_0^{\bcube,sp}(\sX),F)=h^0_{\A^1}F=
\Hom(\omega_{\CI}^{\log} h_0^{\bcube}(X),F)
\]
This implies that\[
\omega_!h_0^{\bcube,sp}(\sX) \cong \omega_{\CI}^{\log} h_0^{\bcube}(X)
\]
In particular, by \cite[Corollary 2.6 (3)]{shuji}, we have that\[
\Log(\omega_!h_0^{\bcube,sp}(\sX))=h_0^{\bcube,sp}(\sX)^{\log}
\]
hence, by the fact that $\omega_{\CI}^{\log}$ is an equivalence on $\LogRec$, we have that\[
h_0^{\bcube,sp}(\sX)^{\log}\cong h_0(X) \cong \omega^*h_0^{\A^1}(U)
\]
Again, we stress that these isomorphisms do not depend on $X$ nor $\sX$, as long as $\underline{X}$ is \emph{proper}.
\end{remark}

\bibliographystyle{amsalpha}
\bibliography{connbib}

\providecommand{\bysame}{\leavevmode\hbox to3em{\hrulefill}\thinspace}
\providecommand{\MR}{\relax\ifhmode\unskip\space\fi MR }
\providecommand{\MRhref}[2]{%
  \href{http://www.ams.org/mathscinet-getitem?mr=#1}{#2}
}
\providecommand{\href}[2]{#2}
\begin{thebibliography}{KMSY21b}

\bibitem[AGV72a]{SGA4I}
Michael {Artin}, Alexander {Grothendieck}, and J.~L. {Verdier},
  \emph{{S\'eminaire de g\'eom\'etrie alg\'ebrique du Bois-Marie 1963--1964.
  Th\'eorie des topos et cohomologie \'etale des sch\'emas. (SGA 4). Un
  s\'eminaire dirig\'e par M. Artin, A. Grothendieck, J. L. Verdier. Avec la
  collaboration de N. Bourbaki, P. Deligne, B. Saint-Donat. Tome 1: Th\'eorie
  des topos. Expos\'es I \`a IV. 2e \'ed.}}, vol. 269, Springer, Cham, 1972,
  [SGA 4, Volume 1].

\bibitem[AGV72b]{SGA4-2}
\bysame, \emph{{S\'eminaire de g\'eom\'etrie alg\'ebrique du Bois-Marie
  1963--1964. Th\'eorie des topos et cohomologie \'etale des sch\'emas. (SGA
  4). Un s\'eminaire dirig\'e par M. Artin, A. Grothendieck, J. L. Verdier.
  Avec la collaboration de N. Bourbaki, P. Deligne, B. Saint-Donat. Tome 2}},
  vol. 270, Springer, Cham, 1972, [SGA 4, Volume 2].

\bibitem[AM86]{am}
M.~Artin and B.~Mazur, \emph{Etale homotopy}, Lecture Notes in Mathematics,
  vol. 100, Springer-Verlag, Berlin, 1986, Reprint of the 1969 original.
  \MR{883959}

\bibitem[Ayo07]{AyoubThesis2}
Joseph Ayoub, \emph{Les six op\'{e}rations de {G}rothendieck et le formalisme
  des cycles \'{e}vanescents dans le monde motivique. {II}}, Ast\'{e}risque
  (2007), no.~315, vi+364. \MR{2438151}

\bibitem[Ayo21]{P1loc}
\bysame, \emph{{$\mathbb{P}^1$}-localisation et une classe de
  {K}odaira-{S}pencer arithm{\'e}tique [{$\mathbb{P}^1$}-localisation and an
  arithmetic {K}odaira-{S}pencer class]}, Tunisian J. of Math. \textbf{3}
  (2021), no.~2, 259--308.

\bibitem[BBD82]{BBD}
A.~A. Be\u{\i}linson, J.~Bernstein, and P.~Deligne, \emph{Faisceaux pervers},
  Analysis and topology on singular spaces, {I} ({L}uminy, 1981),
  Ast\'{e}risque, vol. 100, Soc. Math. France, Paris, 1982, pp.~5--171.
  \MR{751966}

\bibitem[BM21]{purity}
Federico Binda and Alberto Merici, \emph{Connectivity and purity for
  logarithmic motives}, J. Inst. Math. Jussieu (2021), 1--47, To appear; arXiv
  preprint arXiv:2012.08361.

\bibitem[BP{\O}]{bpo}
Federico Binda, Doosung Park, and Paul~Arne {{\O}}stv{\ae}r, \emph{Triangulated
  categories of logarithmic motives over a field}, Ast\'{e}risque, To appear.
  Preprint 2020 \url{https://arxiv.org/abs/2004.12298}.

\bibitem[CTHK97]{bog}
Jean-Louis Colliot-Th\'{e}l\`ene, Raymond~T. Hoobler, and Bruno Kahn, \emph{The
  {B}loch-{O}gus-{G}abber theorem}, Algebraic {$K$}-theory ({T}oronto, {ON},
  1996), Fields Inst. Commun., vol.~16, Amer. Math. Soc., Providence, RI, 1997,
  pp.~31--94.

\bibitem[D{\'{e}}g11]{DegModHomot}
Fr\'{e}d\'{e}ric D{\'{e}}glise, \emph{Modules homotopiques}, Doc. Math.
  \textbf{16} (2011), 411--455. \MR{2823365}

\bibitem[FI07]{IsaksenFausk}
Halvard Fausk and Daniel~C. Isaksen, \emph{t-model structures}, Homology
  Homotopy Appl. \textbf{9} (2007), no.~1, 399--438. \MR{2299805}

\bibitem[HK20]{gabbergeomfinite}
Amit Hogadi and Girish Kulkarni, \emph{Gabber's presentation lemma for finite
  fields}, J. Reine Angew. Math. \textbf{759} (2020), 265--289. \MR{4058181}

\bibitem[HPS97]{HoveyPalmieriStrickland}
Mark Hovey, John~H. Palmieri, and Neil~P. Strickland, \emph{Axiomatic stable
  homotopy theory}, Mem. Amer. Math. Soc. \textbf{128} (1997), no.~610, x+114.
  \MR{1388895}

\bibitem[Isa01]{isaksen}
Daniel~C. Isaksen, \emph{A model structure on the category of pro-simplicial
  sets}, Trans. Amer. Math. Soc. \textbf{353} (2001), no.~7, 2805--2841.
  \MR{1828474}

\bibitem[Kat]{FKato}
Fumiharo Kato, \emph{Exactness, integrality, and log modifications},
  \url{https://arxiv.org/pdf/math/9907124.pdf}.

\bibitem[KMSY]{KMSY3}
Bruno Kahn, Hiroyasu Miyazaki, Shuji Saito, and Takao Yamazaki, \emph{Motives
  with modulus, {III}: {T}he categories of motives.}, Preprint 2020
  \url{https://arxiv.org/abs/2011.11859}.

\bibitem[KMSY21a]{KMSY1}
\bysame, \emph{Motives with modulus, {I}: {M}odulus sheaves with transfers for
  non-proper modulus pairs}, \'{E}pijournal G\'{e}om. Alg\'{e}brique \textbf{5}
  (2021), Art. 1, 46. \MR{4213165}

\bibitem[KMSY21b]{KMSY2}
\bysame, \emph{Motives with modulus, {II}: {M}odulus sheaves with transfers for
  proper modulus pairs}, \'{E}pijournal G\'{e}om. Alg\'{e}brique \textbf{5}
  (2021), Art. 2, 31. \MR{4213166}

\bibitem[KSY]{KSY2}
Bruno Kahn, Shuji Saito, and Takao Yamazaki, \emph{Reciprocity sheaves, {II}},
  Preprint 2019, \url{https://arxiv.org/abs/1707.07398}.

\bibitem[KSY16]{KSY1}
\bysame, \emph{Reciprocity sheaves}, Compos. Math. \textbf{152} (2016), no.~9,
  1851--1898, With two appendices by Kay R\"{u}lling.

\bibitem[Lur17]{ha}
Jacob Lurie, \emph{Higher algebra}, 2017,
  \url{https://www.math.ias.edu/~lurie/papers/HA.pdf}.

\bibitem[Mac71]{maclane}
Saunders MacLane, \emph{Categories for the working mathematician}, Graduate
  Text in Mathematics, vol.~5, Springer-Verlag, New York, Inc., 1971.

\bibitem[Mor05]{Morelconectivity}
Fabien Morel, \emph{The stable {${\Bbb A}^1$}-connectivity theorems},
  $K$-Theory \textbf{35} (2005), no.~1-2, 1--68. \MR{2240215}

\bibitem[MS]{ms}
Alberto Merici and Shuji Saito, \emph{Cancellation theorems for reciprocity
  sheaves}, Preprint 2020 \url{https://arxiv.org/abs/2001.07902}.

\bibitem[MVW06]{MVW}
Carlo Mazza, Vladimir Voevodsky, and Charles Weibel, \emph{Lecture notes on
  motivic cohomology}, Clay Mathematics Monographs, vol.~2, American
  Mathematical Society, Providence, {RI}, 2006.

\bibitem[Ogu18]{ogu}
Arthur Ogus, \emph{Lectures on logarithmic algebraic geometry}, Cambridge
  Studies in Advanced Mathematics, vol. 178, Cambridge University Press,
  Cambridge (UK), 2018. \MR{3838359}

\bibitem[RS]{RS}
Kay R\"{u}lling and Shuji Saito, \emph{Reciprocity sheaves and their
  ramification filtration}, Preprint 2019,
  \url{https://arxiv.org/abs/1812.08716}.

\bibitem[RYS]{RSY}
Kay R\"{u}lling, Takao Yamazaki, and Rin Sugiyama, \emph{Tensor structures in
  the theory of modulus presheaves with transfers}, Preprint 2019,
  \url{https://arxiv.org/abs/1911.05291}.

\bibitem[Sai20a]{S-purity}
Shuji Saito, \emph{Purity of reciprocity sheaves}, Adv. Math. \textbf{366}
  (2020), 1--70. \MR{4070301}

\bibitem[Sai20b]{shuji}
\bysame, \emph{Reciprocity sheaves and logarithmic motives}, 2020, Preprint,
  \url{http://www.lcv.ne.jp/~smaki/articles/RSClog-v2.pdf}.

\bibitem[Ser84]{Serre-GACC}
Jean-Pierre Serre, \emph{Groupes alg\'{e}briques et corps de classes
  [{A}lgebraic groups and class fields]}, second ed., Publications de
  l'Institut Math\'{e}matique de l'Universit\'{e} de Nancago [Publications of
  the Mathematical Institute of the University of Nancago], vol.~7, Hermann,
  Paris, 1984, Actualit\'{e}s Scientifiques et Industrielles [Current
  Scientific and Industrial Topics], 1264. \MR{907288}

\bibitem[{Sta}20]{stack}
The {Stacks Project Authors}, \emph{\textit{Stacks Project}},
  \url{https://stacks.math.columbia.edu}, 2020.

\bibitem[Voe00]{V-TCM}
Vladimir Voevodsky, \emph{Triangulated categories of motives over a field},
  Cycles, transfers, and motivic homology theories, Ann. of Math. Stud., vol.
  143, Princeton Univ. Press, Princeton, NJ, 2000, pp.~188--238.

\end{thebibliography}

\end{document}